\documentclass[11pt,a4paper,review]{siamart190516}
\usepackage[utf8]{inputenc}
\usepackage[english]{babel}
\usepackage{graphicx}
\usepackage{fancyhdr}

\usepackage{amsfonts}
\usepackage{amssymb}

\usepackage{amsmath}
\usepackage{amsfonts}
\usepackage{amssymb}
\usepackage{mathtools}
\usepackage{enumitem}
\usepackage[percent]{overpic}
\usepackage{tikz}
\usepackage{mathrsfs}
\usepackage{xargs}
\usepackage[colorinlistoftodos,prependcaption,textsize=tiny]{todonotes}
\newcommandx{\at}[2][1=]{\todo[linecolor=red,backgroundcolor=red!25,bordercolor=red,#1]{#2}}

\usepackage{algorithm}
\usepackage{algorithmic}

\usepackage{mathtools,booktabs}
\DeclarePairedDelimiter\ceil{\lceil}{\rceil}

\usepackage{varwidth}
\usepackage{hyperref}
\usepackage{cleveref}
\usepackage{cite}

\title{Computing spectral measures of self-adjoint operators\thanks{Submitted to the editors \today.
\funding{The first author was supported by EPSRC grant EP/L016516/1. The work of the second and third authors is supported by the National Science Foundation grant no.~1818757.}}}
\author{Matthew Colbrook\thanks{Department of Applied Mathematics and Theoretical Physics, University of Cambridge,
Cambridge, CB3 0WA. (\email{m.colbrook@damtp.cam.ac.uk})} \and Andrew Horning\thanks{Center for Applied Mathematics, Cornell University, Ithaca, NY 14853. (\email{ajh326@cornell.edu})} \and Alex Townsend\thanks{Department of Mathematics, Cornell University, Ithaca, NY  14853. (\email{townsend@cornell.edu})}}
\headers{Computing spectral measures}{Matthew Colbrook, Andrew Horning, and Alex Townsend}
\begin{document}
\maketitle

\begin{abstract}
Using the resolvent operator, we develop an algorithm for computing smoothed approximations of spectral measures associated with self-adjoint operators. The algorithm can achieve arbitrarily high-orders of convergence in terms of a smoothing parameter for computing spectral measures of general differential, integral, and lattice operators. Explicit pointwise and $L^p$-error bounds are derived in terms of the local regularity of the measure. We provide numerical examples, including a partial differential operator, a magnetic tight-binding model of graphene, and compute one thousand eigenvalues of a Dirac operator to near machine precision without spectral pollution. The algorithm is publicly available in \texttt{SpecSolve}, which is a software package written in MATLAB.
\end{abstract}



\section{Introduction}\label{sec:introduction}
The spectrum of a finite matrix consists only of discrete eigenvalues; however, many of the infinite-dimensional operators in mathematical analysis and physical applications include a continuous spectral component~\cite{stone2009mathematics,kato2013perturbation}. Notably, eigenvalues and eigenvectors do not diagonalize operators with continuous spectra, and one needs extra information to fully describe the operator and associated dynamics of physical models~\cite{gustafson1997operator,touhei2000scattering}.  Given a self-adjoint operator $\mathcal{L}$ acting on a Hilbert space $\mathcal{H}$, the spectral measure (see~\cref{eqn:spec_meas}) of $\mathcal{L}$ is a quantity of great interest because it provides an analogue of diagonalization through the spectral theorem (see~\cref{sec:theory}). Spectral measures are related to correlation in stochastic processes and signal-processing~\cite{kallianpur1971spectral,girardin2003semigroup}~\cite[Ch.~7]{rosenblatt1991stochastic}, scattering cross-sections in particle physics~\cite{efros2007lorentz,efros1994response,efros2019calculating}, the local density-of-states in crystalline materials~\cite{beer1984recursion,haydock1972electronic,lin2016approximating}, and many other quantities~\cite{wilkening2015spectral,killip2003sum,MR838253,damanik2006jost,trogdon2012numerical}. Furthermore, through spectral measures one can compute the functional calculus of $\mathcal{L}$, which is used to solve evolution equations such as the Schr\"odinger equation in quantum mechanics \cite{lubich_qm_book,hochbruck2010exponential}. 

The eigenvalues and eigenvectors of an infinite-dimensional operator with discrete spectrum are usually computed by discretizing and employing a matrix eigensolver~\cite{boyd2001chebyshev,chatelin1983spectral}. Computing spectral measures is more subtle, and previous efforts have mainly focused on operators where analytical formulas or heuristics are available
(see~\cref{sec:spec_meas_connect}). Building on~\cite{colbrook2019computing,horning2019feast}, we develop a general framework for computing approximations to spectral measures of operators that only requires two capabilities: 
\begin{enumerate}[,noitemsep]
\item A numerical solver for shifted linear equations, i.e., $(\mathcal{L}-z)u=f$ with $z\in\mathbb{C}$.
\item Numerical approximations to inner products of the form $\langle u,f\rangle$.
\end{enumerate} 
Here, $\langle\cdot,\cdot\rangle$ is the inner product associated with $\mathcal{H}$, which can be general provided one can compute $\langle u,f\rangle$. We develop high-order rational convolution kernels that allow us to construct accurate approximations to spectral measures by solving the shifted linear equations (see~\cref{tab_kernel} and~\cref{fig:kernels}). Error bounds show that our approximations to the spectral measure converge rapidly (see~\cref{thm:kernel_rates,thm:Lp_rates}). We apply our algorithm to differential (see~\cref{sec:diff_ops}), integral (see~\cref{sec:FredholmIntegralOperator}), and lattice (see~\cref{sec:graphene}) operators to demonstrate its versatility, high accuracy, and robustness. We also use our approximations of spectral measures to compute the first thousand eigenvalues of a Dirac operator (corresponding to bound states in the gap of the essential spectrum) without spectral pollution (see~\cref{sec:dirac}). Thus, spectral measures are also a useful tool for the computation of discrete spectra when there are gaps in the essential spectrum or when discrete spectra cluster (see~\cref{sec:graphene,sec:dirac}). To accompany this paper, we have developed a publicly available MATLAB package called \texttt{SpecSolve} for computing spectral measures of a large class of self-adjoint operators~\cite{SpecSolve}. 

The paper is organized as follows. We recall the definition of the spectral measure of an operator in~\cref{sec:theory} and survey existing algorithms in~\cref{sec:spec_meas_connect}. In~\cref{sec:comput_meas}, we introduce our computational framework and derive high-order versions in~\cref{sec:HigherOrderKernels}. In~\cref{sec:practical_considerations} we discuss algorithmic issues and tackle challenging applications in~\cref{sec:examples}. Finally, we point out additional capabilities and use of the algorithm in~\cref{sec:conclusions}.

\section{The spectral measure of a self-adjoint operator}\label{sec:theory}
Any linear operator acting on a finite-dimensional Hilbert space has a purely discrete spectrum consisting of eigenvalues. In particular, the spectral theorem for self-adjoint $A\in \mathbb{C}^{n\times n}$ states that there exists an orthonormal basis of eigenvectors $v_1,\dots,v_n$ for $\mathbb{C}^n$ such that
\begin{equation}\label{eqn:disc_decomp}
v = \left(\sum_{k=1}^n v_kv_k^*\right)v, \quad v\in\mathbb{C}^n \qquad\text{and}\qquad Av = \left(\sum_{k=1}^n\lambda_k v_kv_k^*\right)v, \quad v\in\mathbb{C}^n,
\end{equation}
where $\lambda_1,\ldots,\lambda_n$ are eigenvalues of $A$, i.e., $Av_k = \lambda_kv_k$ for $1\leq k\leq n$. In other words, the projections $v_kv_k^*$ decompose $\mathbb{C}^n$ and diagonalize $A$.

In the infinite-dimensional setting, we replace $v\in\mathbb{C}^n$ by $f\in\mathcal{H}$, and $A$ by a self-adjoint operator $\mathcal{L}$ with domain $\mathcal{D}(\mathcal{L})\subset\mathcal{H}$.\footnote{Considering $\mathcal{L}:\mathcal{D}(\mathcal{L})\rightarrow\mathcal{H}$ allows us to treat unbounded operators such as differential operators.} If $\mathcal{L}$ has non-empty continuous spectrum, then eigenfunctions of $\mathcal{L}$ do not form a basis for $\mathcal{H}$ or diagonalize $\mathcal{L}$. However, the spectral theorem for self-adjoint operators states that the projections $v_kv_k^*$ in~\cref{eqn:disc_decomp} can be replaced by a projection-valued measure $\mathcal{E}$~\cite[Thm.~VIII.6]{reed1972methods}. The measure $\mathcal{E}$ assigns an orthogonal projector to each Borel-measurable set such that
\[
	f=\left(\int_\mathbb{R} d\mathcal{E}(y)\right)f,\quad f\in\mathcal{H} \qquad\text{and}\qquad \mathcal{L}f=\left(\int_\mathbb{R} y\,d\mathcal{E}(y)\right)f, \quad f\in\mathcal{D}(\mathcal{L}).
\]
Analogous to~\cref{eqn:disc_decomp}, $\mathcal{E}$ decomposes $\mathcal{H}$ and diagonalizes the operator $\mathcal{L}$.

The spectral measure of $\mathcal{L}$ with respect to $f\in\mathcal{H}$ is a scalar measure defined as $\mu_f(\Omega):=\langle\mathcal{E}(\Omega)f,f\rangle$, where $\Omega\subset\mathbb{R}$ is a Borel-measurable set~\cite{reed1972methods}. It is useful to examine Lebesgue's decomposition of $\mu_f$~\cite{stein2009real}, i.e.,
\begin{equation}\label{eqn:spec_meas}
d\mu_f(y)= \underbrace{\sum_{\lambda\in\Lambda^{{\rm p}}(\mathcal{L})}\langle\mathcal{P}_\lambda f,f\rangle\,\delta({y-\lambda})dy}_{\text{discrete part}}+\underbrace{\rho_f(y)\,dy +d\mu_f^{(\mathrm{sc})}(y)}_{\text{continuous part}}.
\end{equation}
The discrete part of $\mu_f$ is a sum of Dirac delta distributions, supported on the set of eigenvalues of $\mathcal{L}$, which we denote by $\Lambda^{{\rm p}}(\mathcal{L})$. The coefficient of each $\delta$ in the sum is $\langle\mathcal{P}_\lambda f,f\rangle=\|\mathcal{P}_\lambda f\|^2$, where $\mathcal{P}_\lambda$ is the orthogonal spectral projector associated with the eigenvalue $\lambda$, and $\|\cdot\|=\sqrt{\langle\cdot,\cdot\rangle}$ is the norm on $\mathcal{H}$. The continuous part of $\mu_f$ consists of an absolutely continuous\footnote{We take ``absolutely continuous'' to be with respect to Lebesgue measure.} part with Radon--Nikodym derivative $\rho_f\in L^1(\mathbb{R})$ and a singular continuous component $\smash{\mu_f^{(\mathrm{sc})}}$.  Without loss of generality, we assume throughout that $\|f\|=1$, which ensures that $\mu_f$ is a probability measure.

Many operators have non-empty continuous spectra \cite[Ch.~10]{kato2013perturbation} such as self-adjoint Toeplitz operators on $\ell^2(\mathbb{N})$ (square summable sequences, where $\mathbb{N}=\{1,2,...\}$) \cite{Bottcher}, differential operators on bounded domains with singular variable coefficients~\cite{izrail1965direct,levitan1975introduction} and unbounded domains~\cite[Ch.~V]{titchmarsh2011elgenfunction}\cite[Ch.~XIII,Ch.~XIV]{dunford1963linear}, and integral perturbations of multiplication operators and Cauchy-type integral operators~\cite{friedrichs1948perturbation,koppelman1960spectral}. In physical systems that scatter or radiate energy, the associated operator typically has a mix of continuous and discrete spectra, e.g., see RAGE theorem~\cite{amrein1973characterization,enss1978asymptotic,ruelle1969remark}. We aim to evaluate smoothed approximations of $\mu_f$ when $\mathcal{L}$ has a non-empty continuous spectrum. This means that we compute samples from a smooth function $g_\epsilon$, with smoothing parameter $\epsilon>0$, that converges weakly to $\mu_f$~\cite[Ch.~1]{billingsley2013convergence}. That is, 
\[
\int_\mathbb{R} \phi(y)g_\epsilon(y)\,dy\rightarrow \int_\mathbb{R}\phi(y)\,d\mu_f(y), \qquad\text{as}\qquad\epsilon\downarrow 0,
\]
for any bounded, continuous function $\phi$. Approximation properties and explicit convergence bounds are studied in~\cref{sec:comput_meas,sec:HigherOrderKernels}.%The approximation properties of 

\section{Applications of spectral measures}\label{sec:spec_meas_connect}
Spectral measures appear in many traditional topics of applied analysis, such as ordinary (ODEs) and partial differential equations (PDEs), stochastic processes, orthogonal polynomials, and random matrix theory. Here, we give a brief survey of existing algorithms for computing $\mu_f$ and closely related quantities.

\subsection{Particle and condensed matter physics}\label{QM_leb_kernels}
Spectral measures are prominent in quantum mechanics~\cite{reed1972methods,hall2013quantum}, where a self-adjoint operator $\mathcal{L}$ represents an observable quantity, and $\mu_f$ describes the likelihood of different outcomes when the observable is measured (see~\cref{sec:alt_convergence}). In this setting, $f\in\mathcal{H}$ with $\|f\|=1$ represents a quantum state. For example, in quantum models of interacting particles, spectral measures of many-body Hamiltonians are used to study the response of a quantum system to perturbations~\cite{efros1994response}. In condensed matter physics, spatially-resolved statistical properties of materials are analyzed using the local density-of-states\footnote{This is distinct from the global density-of-states (DOS), which is formally obtained from the LDOS via an averaging procedure~\cite[Ch.~6.4]{kiejna1996metal}.} (LDOS) of an $n\times n$ matrix $A_n$~\cite[Ch.~6.4]{kiejna1996metal}, which is the spectral measure of $A_n$ taken with respect to a vector $b$~\cite{lin2016approximating}. Here, $A_n$ is typically a discretized or truncated Hamiltonian and one is interested in the thermodynamic limit $n\rightarrow\infty$, so that $A_n$ is too large to compute a full eigenvalue decomposition.

There are two main classes of numerical methods for computing these measures. One class constructs smooth global approximations of the measure with explicit moment-matching procedures~\cite{silver1994densities,langhoff1980stieltjes,weisse2006kernel}, while another class exploits a connection between the spectral measure and the resolvent operator to evaluate samples from a smoothed approximation to the measure~\cite{haydock1972electronic,beer1984recursion,efros2007lorentz}. For example, the so-called recursion method~\cite{haydock1972electronic,beer1984recursion} evaluates the resolvent of tridiagonal Hamiltonians using associated continued-fraction expansions. Resolvent techniques to compute the DOS of finite matrices also appear in the study of random matrices and Schr\"odinger operators, where the connection is made through the Stieltjes transform~\cite{bai2010spectral,carmona2012spectral}.

The resolvent of an operator $\mathcal{L}$ with spectrum $\Lambda(\mathcal{L})$ is given by~\cite[p.~173]{kato2013perturbation} % don't need to restrict to self-adjoint here
\begin{equation}\label{eqn:resolvent}
\mathcal{R}_\mathcal{L}(z)=(\mathcal{L}-z)^{-1}, \qquad z\in\mathbb{C}\setminus\Lambda(\mathcal{L}).
\end{equation}
In~\cref{sec:comput_meas}, we evaluate a smoothed approximation of $\mu_f$ by evaluating the resolvent function $\langle\mathcal{R}_\mathcal{L}(z)f,f\rangle$ in the upper half-plane, i.e., ${\rm Im}(z)>0$. Our approach is closely related to the second class of methods developed for operators in quantum mechanics. A key theme in the above moment-matching and resolvent-based approaches is smoothing, which is introduced by convolution with a smoothing kernel to avoid difficulties associated with the singular part of the measure~\cite{lin2016approximating}. The smoothed approximations of the spectral measures that we compute in~\cref{sec:comput_meas,sec:HigherOrderKernels} also have the form of $K_\epsilon*\mu_f$, where $K_\epsilon$ is a smoothing kernel with smoothing parameter $\epsilon>0$.

Our framework is ``discretization-oblivious," in the sense that it directly resolves the spectral measure of an infinite-dimensional $\mathcal{L}$, and not an underlying discretization. This means that our algorithms do not suffer from spectral pollution.\footnote{Spectral pollution is the phenomenon of eigenvalues of finite discretizations/truncations clustering at points not in the spectrum of $\mathcal{L}$ as the truncation size increases.} Moreover, our framework can be used with any accurate numerical method for solving linear operator equations and computing inner products, making it applicable to differential, integral, and lattice operators. Achieving a discretization-oblivious framework requires balancing refinement in the computation of $\langle\mathcal{R}_\mathcal{L}(z)f,f\rangle$ and refinement in the smoothing parameter, which we do in a principled way (see~\cref{num_balance_act}).

\subsection{Time evolution and spectral density estimation}\label{time_ev_label}
Spectral measures provide a useful lens when studying processes that evolve over time. Suppose that $u: [0,T] \rightarrow\mathcal{H}$ evolves over time according to the abstract Cauchy problem\begin{equation}\label{eqn:time_evo}
\frac{d u}{d t}=-i\mathcal{L}u, \qquad u(0)=f\in\mathcal{H},
\end{equation}
where $\mathcal{L}$ is a self-adjoint operator. For example,~\cref{eqn:time_evo} could describe the evolution of a quantum system according to the Schr\"odinger equation~\cite{lubich_qm_book}. Semigroup theory~\cite{pazy2012semigroups} shows that the solution to~\cref{eqn:time_evo} is given by the operator exponential $\smash{e^{-i\mathcal{L}t}}f$. The autocorrelation function of $u$ is of interest, i.e., 
\[
\langle u(t),f\rangle = \langle e^{-i\mathcal{L}t}f,f\rangle=\int_\mathbb{R}e^{-iyt}\,d\mu_f(y),\qquad t\in[0,T],
\]
which can reveal features that persist over time~\cite{tannor2007introduction}. This interpretation of a time evolution process is quite flexible and can be adapted to describe many signals, $u$, generated by PDEs~\cite{simon1982schrodinger,hundertmark2013operator,dell2019second} and stochastic processes~\cite{kallianpur1971spectral,girardin2003semigroup}~\cite[Ch.~7]{rosenblatt1991stochastic}.

In certain evolution processes, $\mu_f$ is referred to as the spectral density of $u$~\cite{cramer1961some}. The task of spectral density estimation is to recover $\mu_f$ from samples of $u(t)$~\cite[Ch.~1.5]{stoica2005spectral}. A popular technique used in spectral density estimation, related to statistical kernel density estimation \cite{wand1994kernel,tsybakov2008introduction}, reconstructs a smoothed approximation to $\mu_f$ by convolving the empirical measure (a discrete measure supported on the observed samples) with a smoothing kernel~\cite{parzen1957consistent,priestley1962basic}. The particular choice of smoothing kernel affects the convergence properties of the smoothed spectral density~\cite{parzen1961mathematical}. 

In analogy to the variance-bias tradeoff encountered when selecting the smoothing parameter in statistical kernel density estimation \cite{rosenblatt1956remarks,parzen1962estimation}, our smoothed approximations, $K_\epsilon*\mu_f$, exhibit a tradeoff between numerical cost and smoothing (see~\cref{num_balance_act}). In~\cref{sec:HigherOrderKernels}, we adapt arguments from kernel density estimation to determine what properties a smoothing kernel needs to achieve a high-order of convergence in the smoothing parameter. 

\subsection{Sturm--Liouville and Jacobi operators}
Spectral density functions are used in the analysis of singular Sturm--Liouville problems and related classes of self-adjoint operators~\cite{marchenko2011sturm}. A subtle distinction between spectral measures $\mu_f$ and the spectral density function associated with a Sturm--Liouville problem is that the latter does not depend on a given vector $f$. Instead, the spectral density function corresponds to the multiplicative version of the spectral theorem~\cite[Thm.~VIII.4]{reed1972methods}, which induces a Fourier transform-type pair~\cite{coddington}. However, computational methods for both spectral quantities share similarities. For example, one can compute spectral density functions using a Plemelj-type formula~\cite{wilkening2015spectral}, which is similar to~\cref{eqn:jump_relation}.

A common approach to computing spectral density functions associated with Sturm--Liouville operators on unbounded domains is to truncate the domain and take an appropriate limit of an eigenvalue counting function, as implemented in the software package SLEDGE \cite{pruess1993mathematical,fulton1994parallel,fulton1998computation}. This is similar in spirit to DOS calculations, though convergence analysis remains challenging due to the truncation of the infinite interval~\cite{pruess1996error}. This approach can be computationally expensive since the eigenvalues cluster as the domain size increases; often, hundreds of thousands of eigenvalues and eigenvectors need to be computed. One can avoid this cost for certain operators by leveraging analytic limit formulas and solving an ODE at each evaluation point of the spectral density function~\cite{fulton2005computing,fulton2008new}.  Similar methods apply to compute the inverse scattering transform for the Toda lattice and the KdV equation~\cite{bilman2017numerical,trogdon2012numerical}. 

For a Jacobi operator $J$ on $\ell^2(\mathbb{N})$, under suitable conditions~\cite[Ch.~2]{teschl2000jacobi}, the spectral measure $\mu_{e_1}$ of $J$ ($e_1$ denotes the first canonical basis vector) coincides with the measure given by the multiplicative version of the spectral theorem. Moreover, $\mu_{e_1}$ is the probability measure associated with the orthonormal polynomials whose three-term recurrence relation is associated with $J$~\cite{deift1999orthogonal}. Due to this connection, the study of spectral measures has a rich history in the theory of orthogonal polynomials and quadrature rules for numerical integration~\cite{szeg1939orthogonal,MR838253,teschl2000jacobi,deift1999orthogonal,liesen2013krylov}. In special cases, one can recover a distribution function for $\mu_{e_1}$ as a limit of functions constructed using Gaussian quadrature~\cite[Ch.~2]{chihara2011introduction}. One can even use connection coefficients between families of orthogonal polynomials to compute spectral measures of Jacobi operators that arise as compact perturbations of Toeplitz operators~\cite{webb2017spectra}. Applications in this direction include quantum theory and random matrix theory \cite{killip2003sum,simon2010szegHo,gamboa2016sum}. 

While these approaches are specialized to a selected class of operators, we focus on developing a general framework to deal rigorously with arbitrary order ODEs and PDEs (see~\cref{sec:diff_ops}), and integral operators (see~\cref{sec:FredholmIntegralOperator}). The price we pay for this generality is the need to solve shifted linear systems close to the operator's spectrum. We demonstrate that this can be done robustly with fast, well-conditioned, and spectrally accurate methods (see~\cref{sec:practical_considerations}). Similarly, we aim to build a framework that treats general discrete or lattice operators (see~\cref{discrete_specfun_sec}).

\section{Resolvent-based approach to evaluate the spectral measure}\label{sec:comput_meas}
The key to our framework for computing spectral measures is the resolvent of $\mathcal{L}$ (see~\cref{eqn:resolvent}). A classical result in operator theory is Stone's formula, which says that the spectral measure of $\mathcal{L}$ can be recovered from the jump in the resolvent $\mathcal{R}_\mathcal{L}(z)$ across the real axis~\cite{stone1932linear}\cite[Thm.~VII.13]{reed1972methods}. More precisely, if we select $\epsilon>0$ and regard $\mathcal{R}_\mathcal{L}(x+i\epsilon)$ as a function of the real variable $x$, then we have that
\begin{equation}\label{eqn:jump_relation}
\tfrac{1}{2\pi i}\langle\left(\mathcal{R}_\mathcal{L}(\,\cdot +i\epsilon)-\mathcal{R}_\mathcal{L}(\,\cdot-i\epsilon)\right)\!f,f\rangle=\tfrac{1}{\pi}\!{\rm Im}\!\left(\langle\mathcal{R}_\mathcal{L}(\,\cdot+i\epsilon)f,f\rangle\right) \rightarrow \mu_f\text{  as }\epsilon\downarrow 0.
\end{equation}
Here, the equality is due to the conjugate symmetry of $\mathcal{R}_\mathcal{L}(z)$ across the real axis and the limit should be understood in the sense of weak convergence of measures.

Stone's formula is a consequence of the functional calculus identity
\begin{equation}\label{eqn:fa_identity}
\langle\mathcal{R}_\mathcal{L}(x+i\epsilon)f,f\rangle=\int_\mathbb{R}\frac{d\mu_f(y)}{y-(x+i\epsilon)}.
\end{equation}
By using~\cref{eqn:fa_identity} to rewrite~\cref{eqn:jump_relation}, we arrive at an expression for the jump over the real axis as a convolution of the spectral measure with the Poisson kernel, i.e.,
\begin{equation}\label{eqn:pk_identity}
\frac{1}{\pi}{\rm Im}\!\left(\langle\mathcal{R}_\mathcal{L}(x+i\epsilon)f,f\rangle\right)=\frac{1}{\pi}\int_\mathbb{R}\frac{\epsilon}{\epsilon^2+(x-y)^2}d\mu_f(y).
\end{equation}
The Poisson kernel is one of the most common kernels used to smooth approximations of measures in particle and condensed matter physics (see the discussion in~\cref{QM_leb_kernels}). When $\mathcal{L}$ has no singular continuous spectrum, substituting the spectral measure given in~\cref{eqn:spec_meas} into the expression~\cref{eqn:pk_identity} shows that $\mathcal{R}_\mathcal{L}(x+i\epsilon)$ provides an approximation to both the discrete and continuous components of the measure $\mu_f$ for $\epsilon>0$. That is, 
\begin{equation}\label{eqn:disc_cont}
\frac{1}{\pi}{\rm Im}\!\left(\langle\mathcal{R}_\mathcal{L}(x+i\epsilon)f,f\rangle\right)=\frac{1}{\pi}\sum_{\lambda\in\Lambda^{\mathrm{p}}(\mathcal{L})}\frac{\epsilon\,\langle\mathcal{P}_\lambda f,f\rangle}{\epsilon^2+(x-\lambda)^2} + \frac{1}{\pi}\int_\mathbb{R}\frac{\epsilon\,\rho_f(y)}{\epsilon^2+(x-y)^2}\,dy.
\end{equation}
The contribution from the sum in~\cref{eqn:disc_cont} is a series of Poisson kernels centered at the eigenvalues and scaled by the corresponding coefficients $\langle\mathcal{P}_\lambda f,f\rangle$ for $\lambda\in\Lambda^{\mathrm{p}}(\mathcal{L})$. As $\epsilon\downarrow 0$, the sum converges to a series of Dirac delta distributions representing the discrete part of the measure in~\cref{eqn:spec_meas}. Meanwhile, the integral in~\cref{eqn:disc_cont} contributes a smoothed approximation to the Radon--Nikodym derivative $\rho_f$. 

Motivated by~\cref{eqn:disc_cont}, we select $\epsilon>0$ and approximate samples of $\mu_f$ by evaluating 
\begin{equation}\label{eqn:smoothed_meas}
\mu_f^\epsilon(x):=\tfrac{1}{\pi}{\rm Im}\!\left(\langle\mathcal{R}_\mathcal{L}(x+i\epsilon)f,f\rangle\right).
\end{equation}
From~\cref{eqn:jump_relation}, we know that as $\epsilon\downarrow 0$ we have $\mu_f^\epsilon\rightarrow \mu_f$ in the sense of weak convergence of measures. Moreover, if $\mu_f$ has some additional local regularity about a point $x_0\in\mathbb{R}$, then $\mu_f^{\epsilon}(x_0)\rightarrow \rho_f(x_0)$ as $\epsilon\downarrow0$ (see~\cref{thm:poisson_rates}). There is a two-step procedure for evaluating $\mu_f^\epsilon(x_0)$ at some $x_0\in\mathbb{R}$, which is immediate from~\cref{eqn:smoothed_meas}: 
\begin{enumerate}
\item Solve the shifted linear equation for $u^\epsilon$: 
\begin{equation}\label{eqn:shift_sys}
(\mathcal{L}-x_0-i\epsilon)u^\epsilon=f, \qquad u^\epsilon\in\mathcal{D}(\mathcal{L}).
\end{equation}
\item Compute the inner product $\mu_f^{\epsilon}(x_0) = \tfrac{1}{\pi} {\rm Im} (\langle u^\epsilon,f\rangle)$. 
\end{enumerate} 
In practice, the smaller $\epsilon\!>\!0$, the more computationally expensive it is to evaluate~\cref{eqn:smoothed_meas} because if $x_0\in\Lambda(\mathcal{L})$ then the resolvent operator $\mathcal{R}_{\mathcal{L}}(x_0+i\epsilon)$ is unbounded in the limit $\epsilon \downarrow 0$.  One often computes $\mu_f^{\epsilon}(x_0)$ for successively smaller $\epsilon$ to obtain a sequence that converges to $\mu_f(x_0)$. For example, Richardson's extrapolation can improve the convergence rate in $\epsilon$~\cite{colbrook2019computing}, which can be proven using the machinery of~\cref{sec:HigherOrderKernels}.

Typically, one wants to sample $\mu_f^\epsilon$ at several points $x_1,\dots,x_m\in\mathbb{R}$, and then construct a local or global representation of $\mu_f^{\epsilon}$ for visualization or further computations. If one wants to visualize $\mu_f^{\epsilon}$ in an interval, then we recommend evaluating at equispaced points in that interval. However, when one wants to calculate an integral with respect to $\mu_f^\epsilon$, it is better to evaluate $\mu_f^{\epsilon}$ at quadrature nodes (see~\cref{sec:alt_convergence}). Note that if $x_j\not\in\Lambda(\mathcal{L})$, then $\mu_f^\epsilon(x_j)\rightarrow 0$ as $\epsilon\downarrow 0$ (for example, see~\cref{fig:fred_int_meas}).

Although singular continuous spectrum may appear to be an exotic phenomenon, it occurs in applications of practical interest. For example, discrete Schr\"odinger operators with aperiodic potentials on $\ell^2(\mathbb{Z})$ (such as the Fibonacci Hamiltonian) can have spectra that are Cantor sets with purely singular continuous spectral measures (see~\cite{avila2009ten,gordon1997duality,sutHo1989singular,damanik1998singular,damanik2015spectral,puelz2015spectral}). When $\Lambda(\mathcal{L})$ has a non-zero singular continuous component, $\mu_f^\epsilon\rightarrow\mu_f$ weakly as $\epsilon\downarrow 0$ and our algorithms can compute $\mu_f(U)$ (for open sets $U$) and the functional calculus of $\mathcal{L}$.\footnote{In general, it is also impossible to design a black-box method that separates the singular continuous component of $\mu_f$ from the other components. This is made precise in \cite{colbrook2019computing}, which uses the framework of the Solvability Complexity Index (SCI) hierarchy \cite{colbrook2019infinite,colbrook3,colbrookthesis}.}

\subsection{Evaluating the spectral measure of an integral operator}\label{sec:FredholmIntegralOperator} 
To illustrate our evaluation strategy, consider the integral operator defined by
\begin{equation}\label{eqn:fred_int}
[\mathcal{L}u](x)=x u(x)+\int_{-1}^1 e^{-(x^2+y^2)}u(y)\,dy, \qquad x\in [-1,1], \qquad u\in L^2([-1,1]).
\end{equation}
The integral operator $\mathcal{L}$ in~\cref{eqn:fred_int} has continuous spectrum in $[-1,1]$, due to the $x u(x)$ term, and discrete spectrum in $\mathbb{R}\setminus [-1,1]$ from the integral term (a compact perturbation~\cite{kato2013perturbation}).~\Cref{fig:fred_int_meas} (left) shows three smoothed approximations of $\mu_f$ with $f(x)=\sqrt{3/2}\,x$, for smoothing parameter $\epsilon=0.1,0.01$, and $0.001$. We see the presence of an eigenvalue near $x\approx 1.37$ from a spike in the smoothed measure that approximates a Dirac delta . 

\begin{figure}[!tbp]
  \centering
  \begin{minipage}[b]{0.48\textwidth}
    \begin{overpic}[width=\textwidth]{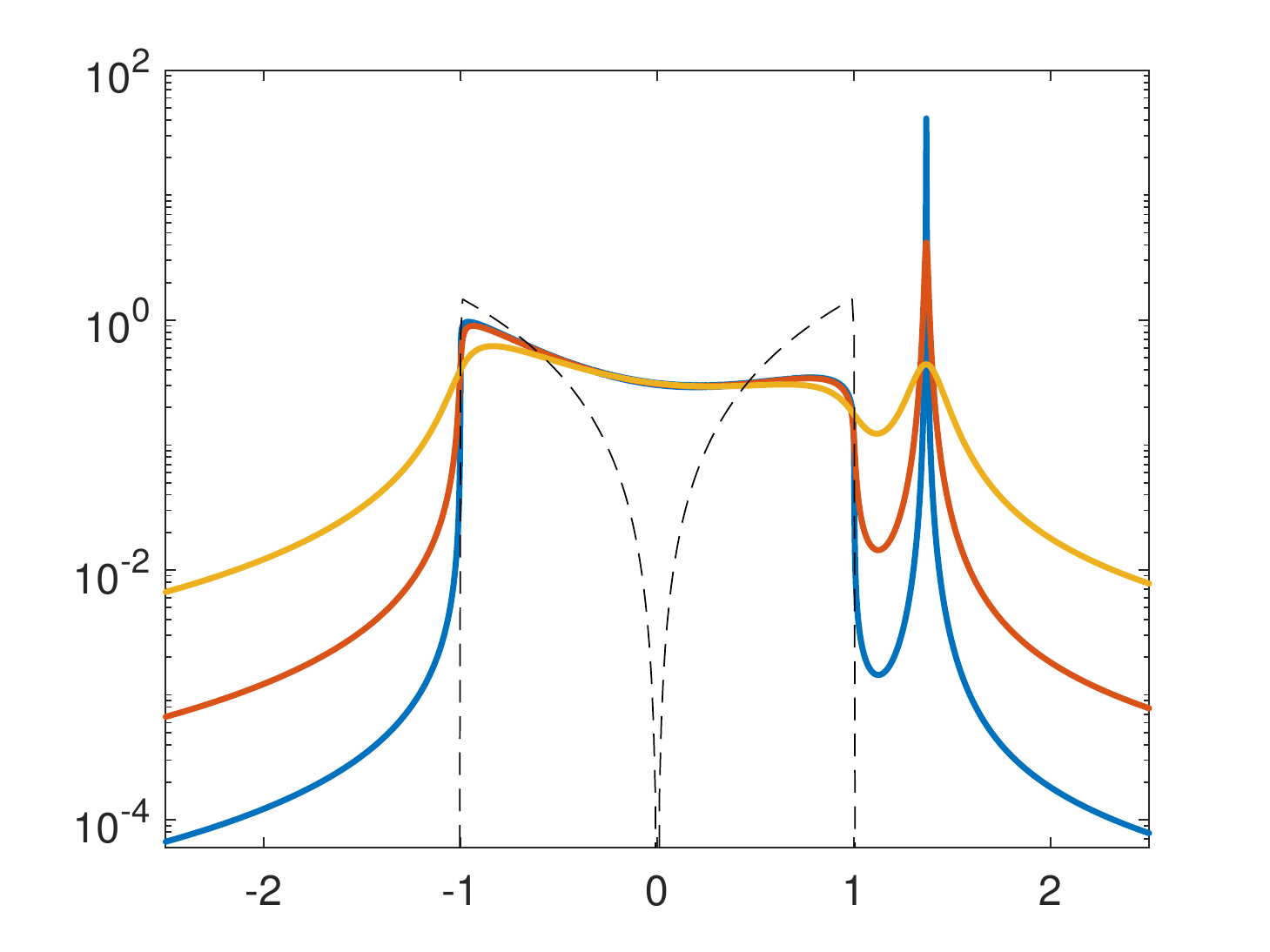}
     \put (47,73) {$\displaystyle \mu_f^\epsilon(x)$}
     \put (50,-2) {$\displaystyle x$}
    \put(13,10) {\rotatebox{30}{$\epsilon = 0.001$}}
    \put(13,20) {\rotatebox{27}{$\epsilon = 0.01$}}
       \put(13,30) {\rotatebox{24}{$\epsilon = 0.1$}}
     \end{overpic}
  \end{minipage}
  \hfill
  \begin{minipage}[b]{0.48\textwidth}
    \begin{overpic}[width=\textwidth]{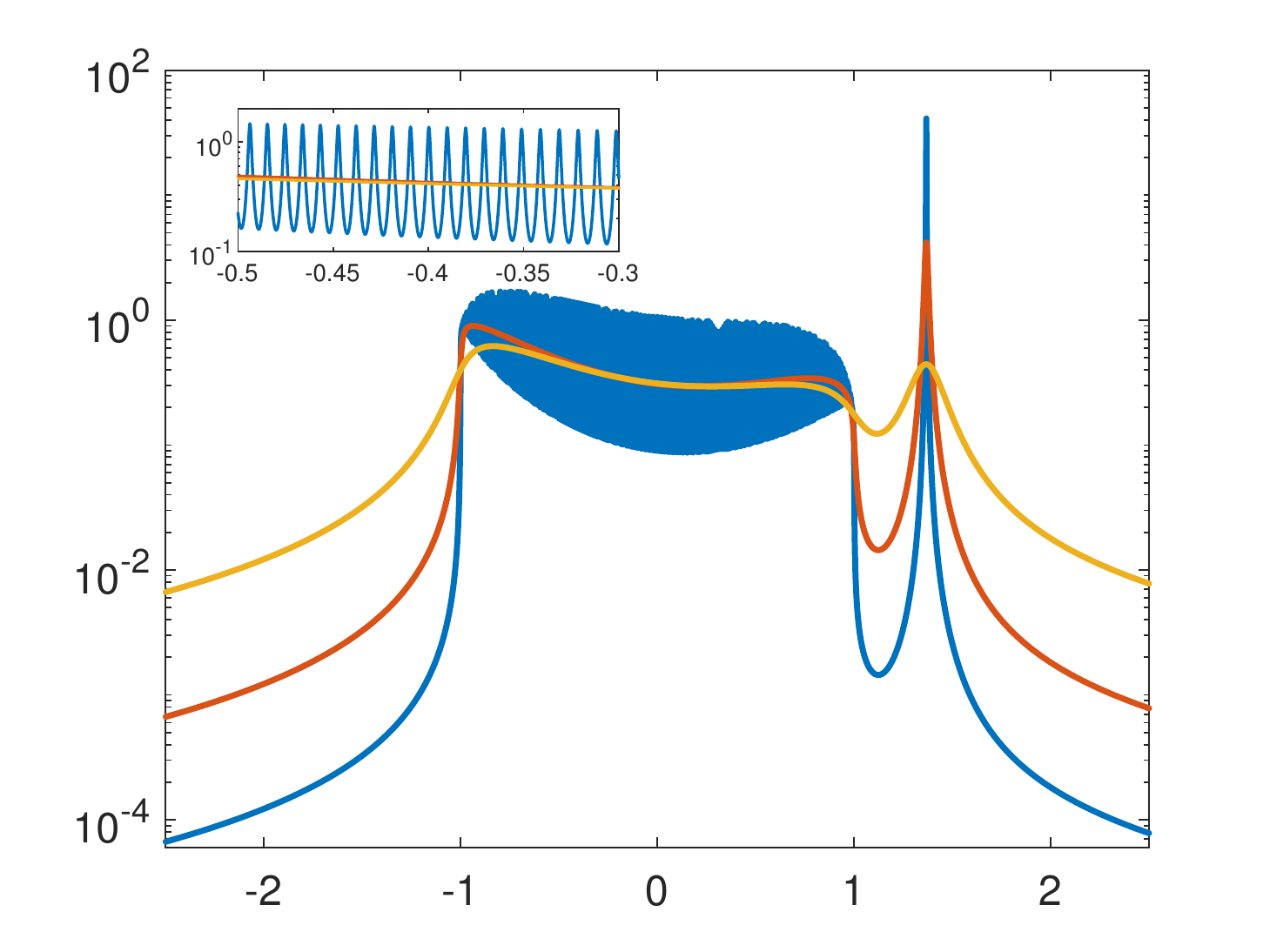}
    \put (47,73) {$\displaystyle \mu_f^\epsilon(x)$}
     \put (50,-2) {$\displaystyle x$}
        \put(13,10) {\rotatebox{30}{$\epsilon = 0.001$}}
    \put(13,20) {\rotatebox{27}{$\epsilon = 0.01$}}
       \put(13,30) {\rotatebox{24}{$\epsilon = 0.1$}}
     \end{overpic}
  \end{minipage}
  \caption{Left: The smoothed approximation $\smash{\mu_f^\epsilon}$ for the integral operator in~\cref{eqn:fred_int} and different $\epsilon$. The discretization sizes for solving the shifted linear systems are adaptively selected. The dashed line corresponds to the spectral measure of the operator given by $u(x)\rightarrow xu(x)$. Adding the compact perturbation (the integral term) alters the shape of the measure over $[-1,1]$ and there is an additional eigenvalue near $x\approx 1.37$. Right: The same computation except with a fixed discretization size of $N=300$ to solve~\cref{eqn:shift_sys}. The magnified region shows spurious high-frequency oscillations for $\epsilon=0.001$, an artifact caused by the discrete spectrum of the underlying discretization.}
\label{fig:fred_int_meas}
\end{figure}
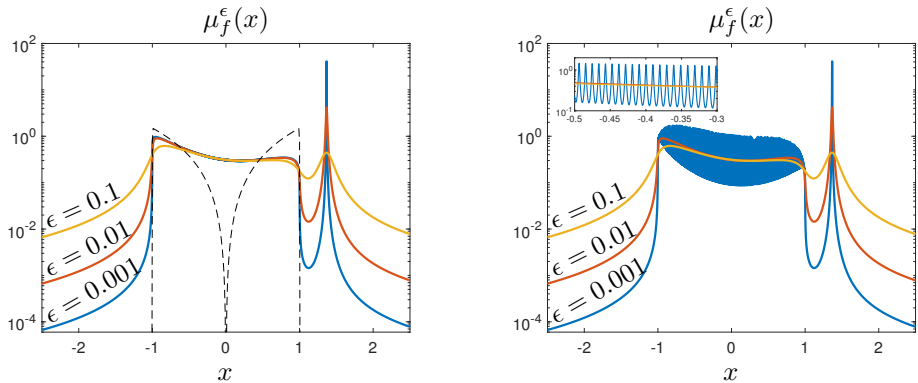

To perform the two-step procedure described above on a computer, one must discretize the operator $\mathcal{L}$, and we do this by discretizing $\mathcal{L}$ with an $N\times N$ matrix corresponding to an adaptive Chebyshev collocation scheme.\footnote{While $N\times N$ discretizations converge for Fredholm operators~\cite{joe1985discrete}, square truncations of spectral discretizations of operators may not always converge. Instead, one may need to take rectangular truncations to ensure that discretizations of $\mathcal{R}_\mathcal{L}(z)f$ converge~\cite{colbrook2019computing}.} While the precise discretization details are delayed until~\cref{sec:fred_disc}, \cref{fig:fred_int_meas} illustrates the critical role that $N$ plays when evaluating $\mu_f^{\epsilon}$. In particular, there are two limits to take in theory: $N\rightarrow \infty$ and $\epsilon\downarrow 0$. It is known that these two limits must be taken with considerable care~\cite{colbrook2019computing}. If $N$ is kept fixed as one takes $\epsilon\downarrow 0$, then the computed samples of $\mu_f^{\epsilon}$ do not converge (see~\Cref{fig:fred_int_meas} (right)) because the computed samples get polluted by the discrete spectrum of the discretization. Instead, as one takes $\epsilon\downarrow 0$, one must appropriately increase $N$ too. In practice, we increase $N$ by selecting it adaptively to ensure that we adequately resolve solutions to~\cref{eqn:shift_sys} (see~\cref{fig:fred_int_meas} (left)). The precise details on how we adequately resolve solutions are given in~\cref{sec:fred_disc}.

\subsection{Pointwise convergence of smoothed measure}\label{sec:smooth_approx}
It is known that if $\mu_f$ is locally absolutely continuous with continuous Radon--Nikodym derivative $\rho_f$ (see~\cref{eqn:spec_meas}), then $\mu_f^\epsilon$ converges pointwise to $\rho_f$~\cite[p. 22]{hoffman2007banach}. However, under additional smoothness assumptions on $\mu_f$, it is useful to understand how rapidly $\mu_f^\epsilon$ converges to $\mu_f$. The connection between $\mu_f^\epsilon$ and the Poisson kernel in~\cref{eqn:pk_identity} allows us to do this on intervals for which $\mu_f$ possesses some local regularity so that $\rho_f$ is H\"{o}lder continuous. We let $\mathcal{C}^{k,\alpha}(I)$ denote the H\"older space of functions that are $k$ times continuously differentiable on an interval $I$ with an $\alpha$-H\"older continuous $k$th derivative~\cite{evans2010partial}. For $h_1\in\mathcal{C}^{0,\alpha}(I)$ and $h_2\in\mathcal{C}^{k,\alpha}(I)$ we define the seminorm and norm, respectively, as
$$
|h_1|_{\mathcal{C}^{0,\alpha}(I)}=\sup_{x\neq y\in I}\frac{|h_1(x)-h_1(y)|}{|x-y|^{\alpha}},\quad \|h_2\|_{\mathcal{C}^{k,\alpha}(I)}=|h_2^{(k)}|_{\mathcal{C}^{0,\alpha}(I)}+\max_{0\leq j\leq k}\|h_2^{(j)}\|_{\infty,I}.
$$

\begin{theorem}\label{thm:poisson_rates}
Suppose that the measure $\mu_f$ in (\ref{eqn:spec_meas}) is absolutely continuous on the interval $I=(x_0-\eta,x_0+\eta)$ for some $x_0\in\mathbb{R}$ and $\eta>0$, let $\mu_f^\epsilon$ be defined as in~\cref{eqn:smoothed_meas}, and let $0\leq \alpha < 1$. If $\rho_f\in \mathcal{C}^{0,\alpha}(I)$, then
$$
|\rho_f(x_0)-\mu_f^\epsilon(x_0)|=\mathcal{O}(\epsilon^\alpha), \qquad \text{as}\quad \epsilon \downarrow 0.  
$$
\end{theorem}
\begin{proof}
First, decompose $\rho_f$ into two non-negative parts so that $\rho_f=\rho_1+\rho_2$, where the support of $\rho_1$ is in $I$ and $\rho_2$ vanishes on $(x_0-\eta/2,x_0+\eta/2)$. Since $\rho_f(x_0)=\rho_1(x_0)$ and the Poisson kernel integrates to $1$, we can use the convolution representation for $\mu_f^\epsilon$ (see~\cref{eqn:pk_identity} and~\cref{eqn:smoothed_meas}) and the commutativity of convolution to bound the approximation error as
\begin{equation}\label{eqn:approx_error}
\begin{aligned}
\pi|\rho_f(x_0)-\mu_f^\epsilon(x_0)|&= \! \left|\int_\mathbb{R}\!\frac{\epsilon}{\epsilon^2+y^2}\rho_1(x_0)dy - \int_\mathbb{R}\!\frac{\epsilon d\mu_f(y)}{\epsilon^2+(x_0-y)^2}\right| \\
&\leq \! \left|\int_\mathbb{R}\!\frac{\epsilon}{\epsilon^2+y^2}\left(\rho_1(x_0)-\rho_1(x_0-y)\right)\!dy\right| +
\int_\mathbb{R}\!\frac{\epsilon d\mu_f^{(\mathrm{r})}(y)}{\epsilon^2+(x_0-y)^2}.
\end{aligned}
\end{equation}
Here, $\smash{d\mu_f^{(\mathrm{r})}(y):=d\mu_f(y)-\rho_1(y)dy}$ is a non-negative measure with support in $\mathbb{R}\setminus(x_0-\eta/2,x_0+\eta/2)$. Since $\mu_f$ is a probability measure, we have that $\int_\mathbb{R}d\mu_f^{(\mathrm{r})}(y)\leq 1$, and the second term in~\cref{eqn:approx_error} is bounded via
\begin{equation}\label{eqn:approx_error_term1}
\int_\mathbb{R}\frac{\epsilon d\mu_f^{(\mathrm{r})}(y)}{\epsilon^2+(x_0-y)^2}= \int_{|x_0-y|\geq\eta/2}\frac{\epsilon d\mu_f^{(\mathrm{r})}(y)}{\epsilon^2+(x_0-y)^2}
\leq\frac{\epsilon}{\epsilon^2+\frac{\eta^2}{4}}.
\end{equation}
Since $\rho_f\in \mathcal{C}^{0,\alpha}(I)$, standard arguments using cutoff functions \cite{evans2010partial} show that we can choose $\rho_1$ so that $|\rho_1|_{\mathcal{C}^{0,\alpha}(I)}\leq|\rho_f|_{\mathcal{C}^{0,\alpha}(I)}+C\eta^{-\alpha}\|\rho_f\|_{\infty,I}$ for some universal constant $C$.  Consequently, we have that
$$
|\rho_1(x_0)-\rho_1(x_0-y)\rvert\leq|\rho_1|_{\mathcal{C}^{0,\alpha}(I)}|y|^\alpha\leq(|\rho_f|_{\mathcal{C}^{0,\alpha}(I)}+C\eta^{-\alpha}\|\rho_f\|_{\infty,I})|y|^\alpha.
$$
Substituting this bound into the first term on the right-hand side of~\cref{eqn:approx_error} and combining with~\cref{eqn:approx_error_term1}, yields
$$
|\rho_f(x_0)-\mu_f^\epsilon(x_0)|\leq \frac{|\rho_f|_{\mathcal{C}^{0,\alpha}(I)}+C\eta^{-\alpha}\|\rho_f\|_{\infty,I}}{\pi}\int_\mathbb{R}\frac{\epsilon}{\epsilon^2+y^2}|y|^\alpha\,dy + \frac{\epsilon}{\pi\left(\epsilon^2+\frac{\eta^2}{4}\right)}.
$$
Calculating the integral explicitly leads to
\begin{equation}\label{eqn:statement1}
|\rho_f(x_0)-\mu_f^\epsilon(x_0)|\leq\left(|\rho_f|_{\mathcal{C}^{0,\alpha}(I)}+C\eta^{-\alpha}\|\rho_f\|_{\infty,I}\right)\mathrm{sec}\!\left(\frac{\alpha\pi}{2}\right)\epsilon^\alpha + \frac{\epsilon}{\pi\left(\epsilon^2+\frac{\eta^2}{4}\right)}.
\end{equation}
The right-hand side of~\cref{eqn:statement1} is $\mathcal{O}(\epsilon^\alpha)$ as $\epsilon\downarrow 0$, which concludes the proof.
\end{proof}

In~\cref{thm:poisson_rates}, we see that the convergence rate of $|\rho_f(x_0) - \mu_f^\epsilon(x_0)|$ as $\epsilon\downarrow 0$ depends on the local regularity of $\mu_f$. One can also show (see~\cref{thm:kernel_rates}) that $|\rho_f(x_0)-\mu_f^\epsilon(x_0)|=\mathcal{O}(\epsilon\log(1/\epsilon))$ if $\rho_f\in \mathcal{C}^{1}(I)$ as well as the fact that any additional smoothness assumptions on $\rho_f$ no longer improve the convergence rate.\footnote{The logarithmic term occurs due to the non-integrability of $x/(\pi(x^2+1))$. One can also show that the error rate of $\mathcal{O}(\epsilon\log(1/\epsilon))$ is achieved if $\rho_f\in \mathcal{C}^{0,1}(I)$ is Lipschitz continuous.} Since our procedure is local, the convergence rate is not affected by far away discrete and singular continuous components of $\mu_f$. However, the convergence degrades near singular points in the spectral measure because the constants in~\cref{eqn:statement1} blow up as $\eta\rightarrow0$. While $|\rho_f(x_0)-\mu_f^\epsilon(x_0)|=\mathcal{O}(\epsilon^\alpha)$ in~\cref{thm:poisson_rates} is stated as an asymptotic statement, we can also obtain explicit bounds for adaptive selection of $\epsilon$ (see~\cref{thm:kernel_rates}).

\subsection{A numerical balancing act}\label{num_balance_act}
To explore the practical importance of the convergence rates in~\cref{thm:poisson_rates}, we examine the numerical cost associated with solving the shifted linear systems in~\cref{eqn:shift_sys}. When the real component of the shift is in the continuous spectrum of $\mathcal{L}$ and $\epsilon$ is small, we typically require large discretizations to avoid the situation observed in~\cref{fig:fred_int_meas} (right). There are many potential reasons why we require large discretization sizes as $\epsilon\downarrow0$. Here are two illustrative examples:
\begin{figure}[!tbp]
  \centering
  \begin{minipage}[b]{0.48\textwidth}
    \begin{overpic}[width=\textwidth]{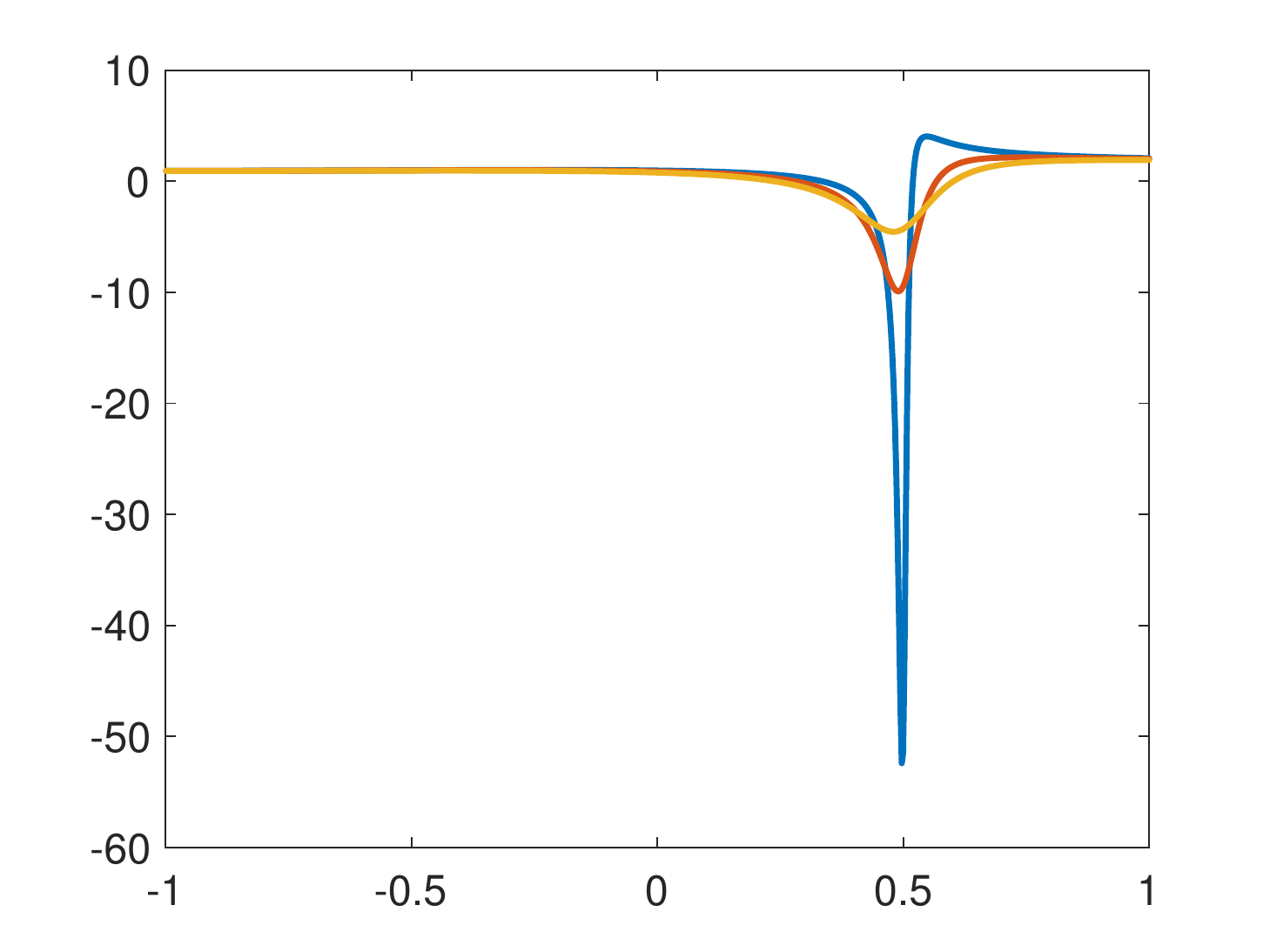}
 	\put (47,73) {$\displaystyle u^\epsilon(x)$}
 	\put (50,-2) {$\displaystyle x$}
 	\end{overpic}
  \end{minipage}
  \hfill
  \begin{minipage}[b]{0.48\textwidth}
    \begin{overpic}[width=\textwidth]{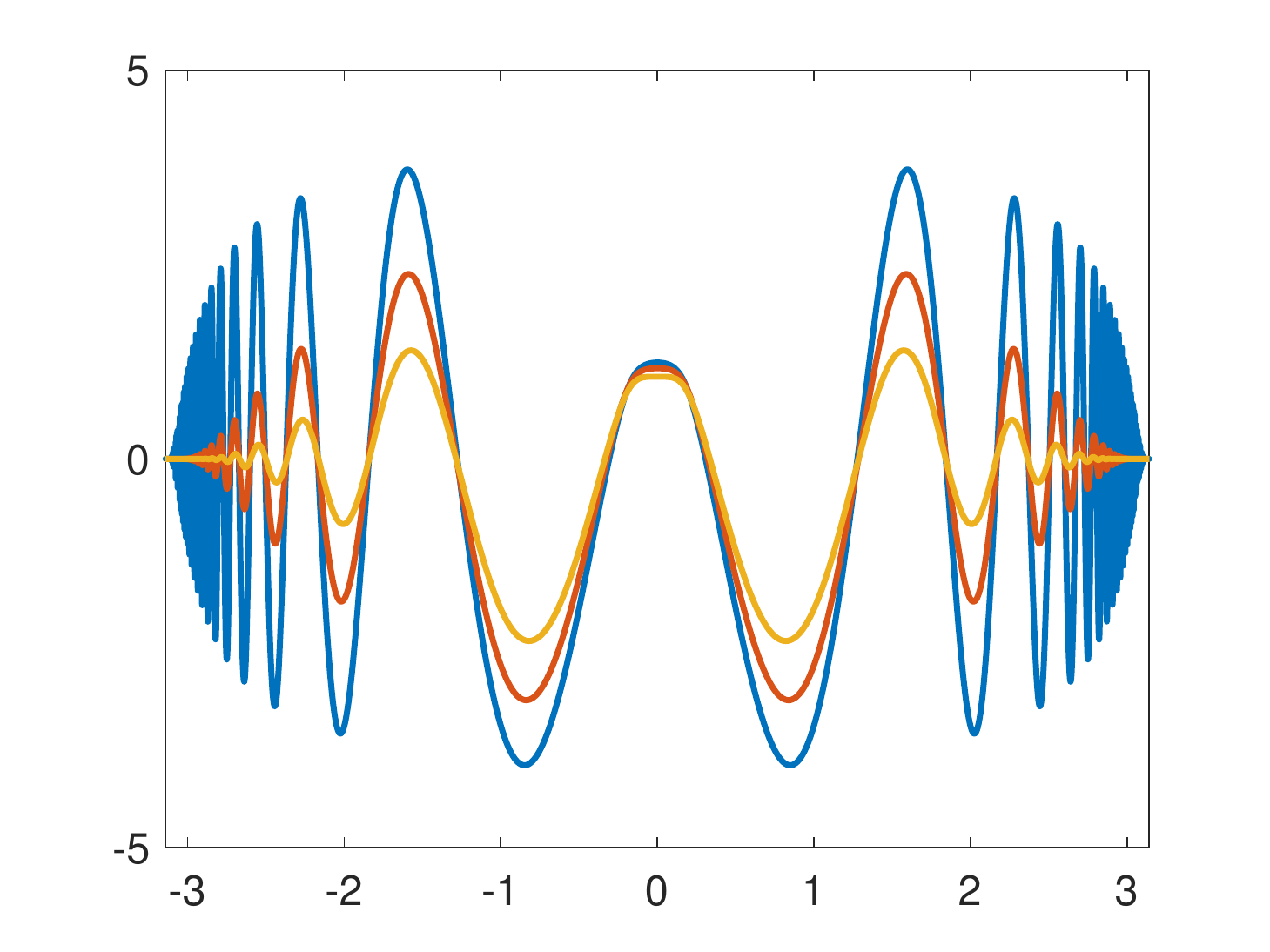}
    \put (47,73) {$\displaystyle u^\epsilon(\theta)$}
 	\put (50,-2) {$\displaystyle \theta$}
 	\end{overpic}
  \end{minipage}
  \caption{Real part of the numerical solutions to the shifted linear equations in~\cref{eqn:shift_sys} for the integral operator in~\cref{eqn:fred_int} (left) and the Schr\"odinger operator in~\cref{eqn:schrodinger} (right), with $\epsilon=0.1$ (yellow), $\epsilon=0.05$ (orange), and $\epsilon=0.01$ (blue). The solutions to~\cref{eqn:schrodinger} are mapped to $[-\pi,\pi]$ via $x=10i(1-e^{i\theta})/(1+e^{i\theta})$. We discretize using sparse, well-conditioned spectral methods and the discretization sizes are selected adaptively to accurately resolve $u^\epsilon(x)$ and $u^\epsilon(\theta)$.  \label{fig:singular_solns}}
\end{figure}

\paragraph{1) Interior layers} Revisiting the integral operator example in~\cref{eqn:fred_int}, we select $x_0=1/2$ in the continuous spectrum of $\mathcal{L}$, and $f(x)=\sqrt{3/2}\,x$. In \cref{fig:singular_solns} (left), we observe that the solution $u^\epsilon(x)$ develops an interior layer and blows up at $x_0=1/2$ as $\epsilon\downarrow 0$. The blow-up occurs because the multiplicative term in $\mathcal{L}-(x_0+i\epsilon)$ has a root at $x_0=1/2$ when $\epsilon=0$, giving rise to a pole in $u^\epsilon(x)$. For $\epsilon>0$, the pole of $u^\epsilon(x)$ is located at a distance of $\mathcal{O}(\epsilon)$ away from the real axis. A large discretization size is needed to resolve $u^\epsilon(x)$ for small $\epsilon$ due to the thin interior layer in $u^\epsilon(x)$. 

\paragraph{2) Oscillatory behavior} Consider the second-order differential operator given by
\begin{equation}\label{eqn:schrodinger}
[\mathcal{L}u](x)=-\frac{d^2u}{dx^2}(x)+\frac{x^2}{1+x^6}u(x), \qquad x\in\mathbb{R}.
\end{equation}
We select $x_0=0.3$ in the continuous spectrum of $\mathcal{L}$, and $f(x)=\sqrt{9/\pi}\cdot{}x^2/(1+x^6)$. In~\cref{fig:singular_solns} (right), we plot solutions mapped onto the domain $[-\pi,\pi]$ by the change-of-variables $x=10i(1-e^{i\theta})/(1+e^{i\theta})$. The solutions $u^\epsilon(x)$ are highly oscillatory with slow decay as $\theta\rightarrow\pm\pi$. As $\epsilon\downarrow 0$ the decay degrades and the persistent oscillations correspond to a transition in the nature of the singular points of~\cref{eqn:shift_sys} at $\pm\infty$. This means a large discretization is needed to resolve $u^\epsilon(x)$ for small $\epsilon$.

The dominating computational expense in evaluating $\mu_f^\epsilon$ is solving the shifted linear systems in~\cref{eqn:shift_sys}, and the cost of computing $u^\epsilon(x)$ generally increases as $\epsilon\downarrow 0$. There is a balancing act. On the one hand, we wish to stay as far away from the spectrum as possible, so that the evaluation of $\mu_f^\epsilon$ is computationally efficient. On the other hand, we desire samples of $\mu_f^\epsilon$ to be good approximations to $\rho_f$, which requires a small $\epsilon>0$. Even though we use sparse, well-conditioned spectral methods to discretize~\cref{eqn:shift_sys} (see~\cref{sec:practical_considerations}), the trade-off between computational cost and accuracy means that the slow convergence rate determined in~\cref{thm:poisson_rates} is a severe limitation. In~\cref{fig:N_vs_epsilon}, we explore the discretization sizes that are needed to evaluate spectral measures with the Poisson kernel accurately. For the integral operator in~\cref{eqn:fred_int} and $\epsilon = 0.05$, $0.01$, and $0.005$, we observe that we need $N=400$, $1700$, and $3100$, respectively (see~\cref{fig:N_vs_epsilon} (left)). Unfortunately, to obtain samples of the spectral measure with two digits of relative accuracy, we require that $\epsilon\approx0.01$ (see~\cref{fig:N_vs_epsilon}). For this example, we observe that we require $N\approx 20/\epsilon$ for small $\epsilon>0$, so it is computationally infeasible to obtain more than five or six digits of accuracy with the Poisson kernel.

\begin{figure}[!tbp]
  \centering
  \begin{minipage}[b]{0.48\textwidth}
    \begin{overpic}[width=\textwidth]{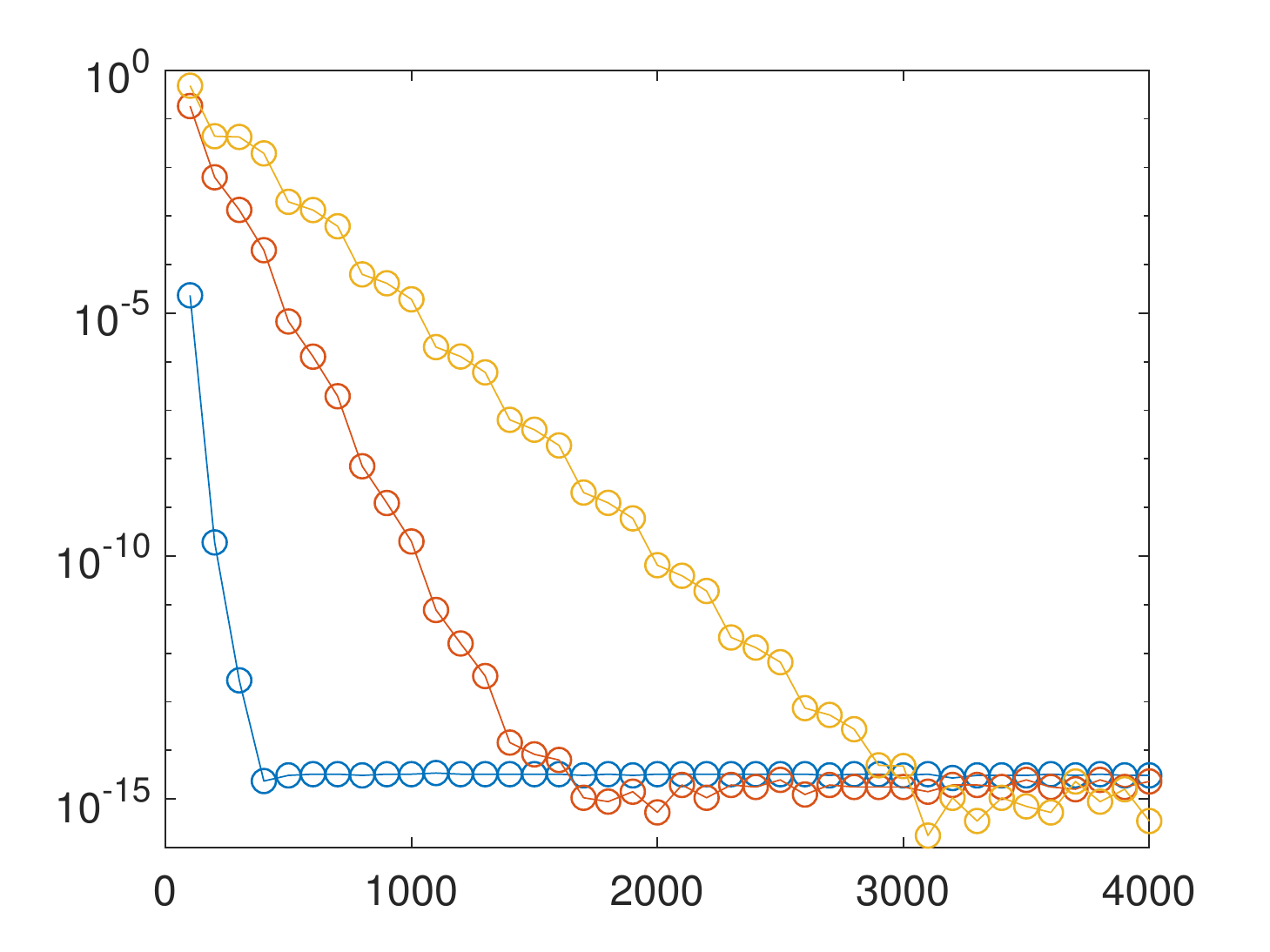}
 	\put (16,73) {$\displaystyle | \mu_{f,N}^\epsilon(x_0)-\mu_f^\epsilon(x_0)|/|\mu_f^\epsilon(x_0)|$}
 	\put (50,-2) {$\displaystyle N$}
 	\put (19,40) {\rotatebox{-80} {$\displaystyle \epsilon=0.05$}}
 	\put (30,42)  {\rotatebox{-60} {$\displaystyle \epsilon=0.01$}}
 	\put (41,46) {\rotatebox{-39} {$\displaystyle \epsilon=0.005$}}
 	\end{overpic}
  \end{minipage}
  \hfill
  \begin{minipage}[b]{0.48\textwidth}
    \begin{overpic}[width=\textwidth]{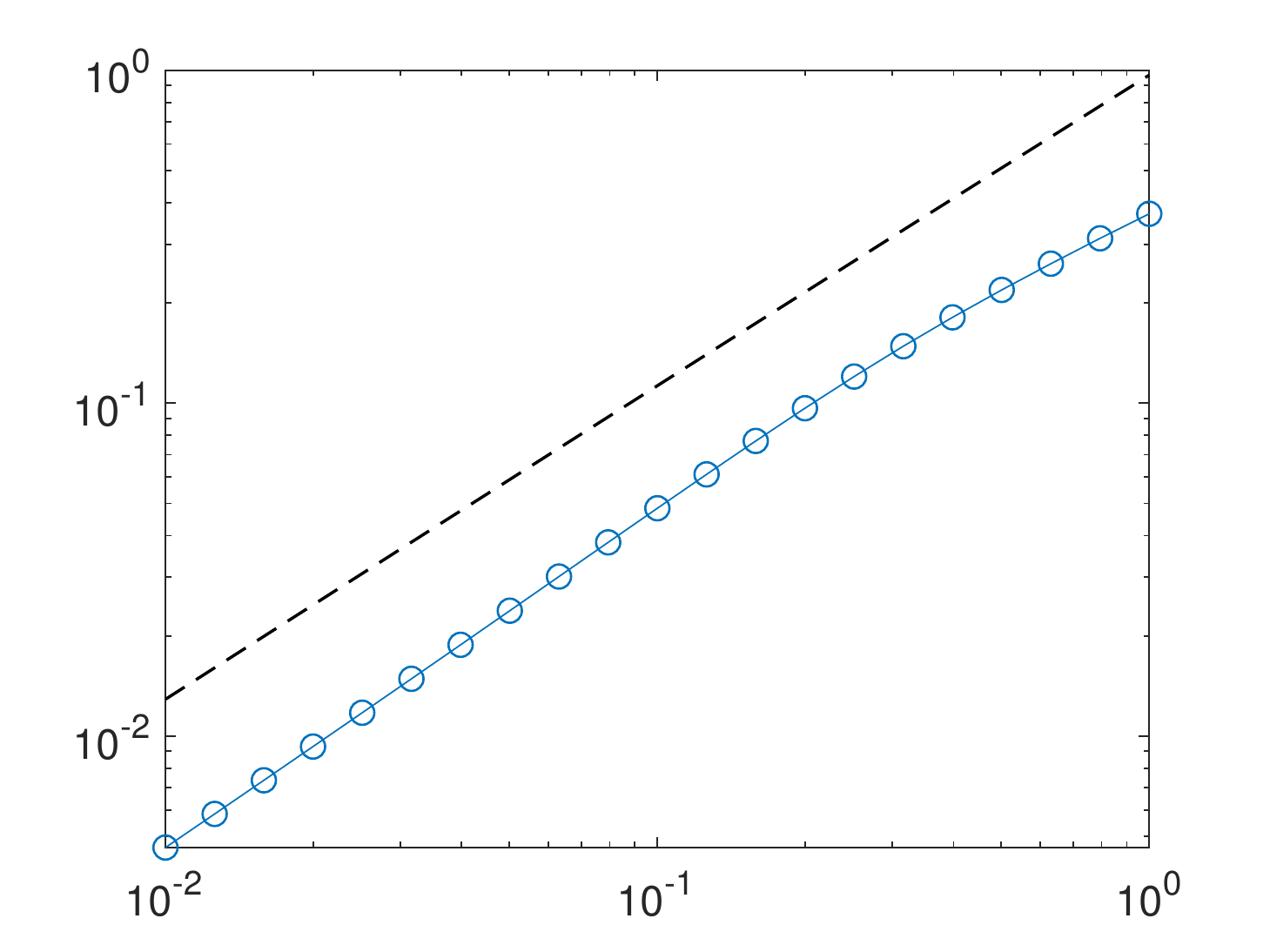}
    \put (20,73) {$\displaystyle |\rho_f(x_0)-\mu_f^\epsilon(x_0)|/|\rho_f(x_0)|$}
    \put (25,33)  {\rotatebox{33} {$\displaystyle \mathcal{O}(\epsilon\log(\epsilon^{-1}))$}}
 	\put (50,-2) {$\displaystyle \epsilon$}
 	\end{overpic}
  \end{minipage}
  \caption{Left: The relative error in the numerical approximation $\smash{\mu_{f,N}^\epsilon}$, corresponding to discretization size $N$, of the smoothed measure in~\eqref{eqn:smoothed_meas} for the integral operator in~\cref{eqn:fred_int} with $\epsilon=0.05$, $\epsilon=0.01$, and $\epsilon=0.005$. Right: The pointwise relative difference between the smoothed measure $\smash{\mu_f^\epsilon(x)}$ and the density $\rho_f(x)$, evaluated at $x_0=1/2$, compared with the $\mathcal{O}(\epsilon\log(\epsilon^{-1}))$ error bound in~\cref{thm:kernel_rates} for the integral operator in~\cref{eqn:fred_int}. The relative error is computed by comparing with a numerical solution that has been adaptively resolved to machine precision.\label{fig:N_vs_epsilon}}
\end{figure}
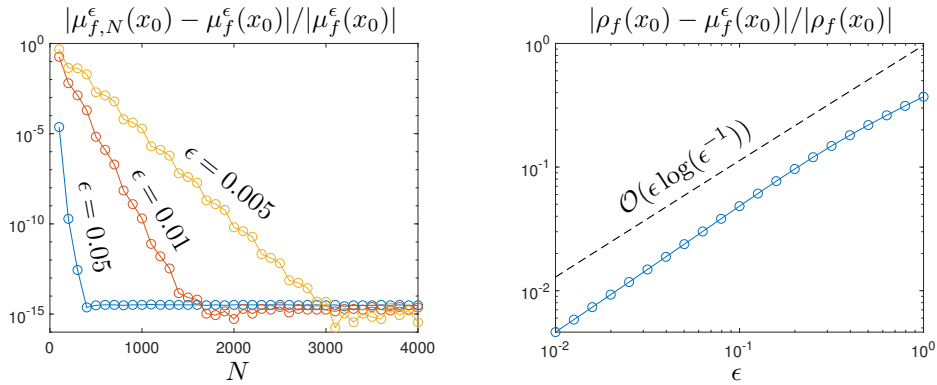

In addition to the computational cost of increasing $N$, the discretizations used to solve the linear systems in~\cref{eqn:shift_sys} become increasingly ill-conditioned when $x_0\in\Lambda(\mathcal{L})$ and $\epsilon\downarrow 0$ (a reflection of $\|\mathcal{R}_\mathcal{L}(x_0+i\epsilon)\|=\epsilon^{-1}$). This can limit the attainable accuracy. Moreover, the performance of iterative methods, if used to accelerate the solution of the large shifted linear systems, may also suffer. In our experience, the cost of increasing $N$ is usually the limiting factor and we rarely take $\epsilon<10^{-2}$.

\section{High-order kernels}\label{sec:HigherOrderKernels} 
~\cref{thm:poisson_rates} demonstrates that $\mu_f^\epsilon\rightarrow \rho_f$ pointwise in intervals for which $\mu_f$ is absolutely continuous with H\"older continuous density $\rho_f$, where the rate of convergence depends on the H\"older exponent of $\rho_f$. However, even when $\rho_f$ possesses additional regularity, the best rate of convergence for smoothed measures using the Poisson kernel is $\mathcal{O}(\epsilon\log(1/\epsilon))$. A natural question is: 
\begin{quote}
``Can we use other kernels to exploit additional regularity in $\mu_f$?''
\end{quote}

In this section, we construct kernels that can be used to compute smoothed measures that approximate $\rho_f$ to high-order in $\epsilon$ when $\rho_f$ is smooth. This allows us to obtain accurate samples of $\mu_f$ while avoiding extremely small $\epsilon$ and the associated computational cost of solving the shifted linear equations in~\cref{eqn:shift_sys} when the shifts are close to the real line. We use $K(x)$ to denote a kernel for which $K_\epsilon(x)=\epsilon^{-1}K(x/\epsilon)$ is an approximation to the identity, i.e., $K_\epsilon\rightarrow\delta$ as $\epsilon\downarrow 0$ in the sense of distributions~\cite[Ch.~3]{stein2011functional}, where $\delta$ is the Dirac delta distribution.

To gain intuition about the conditions that $K(x)$ must satisfy so that $K_\epsilon*\mu_f$ approximates $\mu_f$ to high-order, consider an absolutely continuous probability measure $\mu$ with density $\rho$ supported on an interval $I=(x_0-\eta,x_0+\eta)$, for some $x_0\in\mathbb{R}$ and $\eta>0$. The following argument is common in statistical non-parametric regression~\cite{wand1994kernel,tsybakov2008introduction}. Since we want $K_\epsilon$ to be an approximation to the identity, our first property is that $\int_{\mathbb{R}} K(x) dx = 1$. For further properties, we examine the approximation error 
$$
[K_\epsilon*\mu](x_0)-\rho(x_0)=\int_\mathbb{R} K_\epsilon(y)(\rho(x_0-y)-\rho(x_0))\,dy.
$$
Assuming that $\rho\in \mathcal{C}^{n,\alpha}(I)$ for some $0<\alpha<1$, we can use an $n$th order Taylor expansion of $\rho(x_0-y)-\rho(x_0)$ to rewrite the approximation error as
$$
[K_\epsilon*\mu](x_0)-\rho(x_0)=\sum_{k=1}^{n-1}\frac{(-1)^k\rho^{(k)}(x_0)}{k!}\int_\mathbb{R} K_\epsilon(y)y^k\,dy + \int_\mathbb{R}K_\epsilon(y)R_n(x_0,y)\,dy,
$$
where $R_n(x_0,y)$ denotes the $\mathcal{O}(|y|^n)$ remainder term in the Taylor series and $\rho^{(k)}$ is the $k$th derivative of $\rho$. The change-of-variables $y\rightarrow\epsilon y$ reveals that the $k$th term in the series is of size $\mathcal{O}(\epsilon^k)$, provided that $K(y)y^k$ is integrable. Meanwhile, the H\"older continunity of $\rho^{(n)}$ shows that the term involving $R_n(x_0,y)$ is of size $\mathcal{O}(\epsilon^{n+\alpha})$ provided that $K(y)y^{n+\alpha}$ is integrable and $\int_\mathbb{R}K(y)y^n\,dy=0$. Therefore, a kernel that achieves an $\mathcal{O}(\epsilon^{n+\alpha})$ approximation error has vanishing moments, i.e., $\int_\mathbb{R}K(y)y^k\,dy=0$ for $1\leq k\leq n.$

In practice, $\mu$ may not be absolutely continuous and its absolutely continuous part may have a density $\rho$ with singular points or unbounded support. As in~\cref{thm:poisson_rates}, we can deal with the general case by decomposing $\rho=\rho_1+\rho_2$ into two non-negative parts, where $\rho_1$ is sufficiently smooth and compactly supported on $I$, and where $\rho_2$ vanishes in a neighborhood of $x_0$. The cost of this decomposition is a second term in the approximation error (analogous to the second term on the right-hand side of~\cref{eqn:approx_error})
$$
[K_\epsilon*\mu](x_0)-\rho(x_0)=\int_\mathbb{R} K_\epsilon(y)(\rho_1(x_0-y)-\rho_1(x_0))\,dy + \int_\mathbb{R} K_\epsilon(x_0-y)\,d\mu^{(\mathrm{r})}(y),
$$
where $d\mu^{(\mathrm{r})}(y)=d\mu(y)-\rho_1(y)dy$. To ensure that this additional term does not dominate as $\epsilon\downarrow 0$, it is necessary that the kernel $K(y)$ decays at an appropriate rate as $|y|\rightarrow\infty$. This ensures that $K_\epsilon(x_0-y)$ is sufficiently small on the support of $d\mu^{(\mathrm{r})}(y)$ (see~\cref{eqn:approx_error_term1} for the decay in the Poisson kernel). Motivated by this discussion, we make the following definition (similar to~\cite[Def.~1.3]{tsybakov2008introduction}). 

\begin{definition}[$m$th order kernel]
\label{def:mth_order_kernel}
Let $m$ be a positive integer and $K\in L^1(\mathbb{R})$. We say $K$ is an $m$th order kernel if it satisfies the following properties:
\begin{itemize}[leftmargin=20pt]
	\item[(i)] Normalized: $\int_{\mathbb{R}}K(x)dx=1$.
	\item[(ii)]  Zero moments: $K(x)x^j$ is integrable and $\int_{\mathbb{R}}K(x)x^jdx=0$ for $0<j<m$.
	\item[(iii)] Decay at $\pm\infty$: There is a constant $C_K$, independent of $x$, such that
	\begin{equation}
	\label{decay_bound}
	\left|K(x)\right|\leq \frac{C_K}{(1+\left|x\right|)^{m+1}}, \qquad x\in \mathbb{R}.
	\end{equation}
\end{itemize}
\end{definition}

It is straightforward to verify that the Poisson kernel is a first-order kernel and the Gaussian kernel, i.e., $h(x)=(2\pi)^{-1/2}e^{-x^2/2}$, is a second-order kernel. While the Gaussian kernel plays an important role in DOS calculations~\cite{lin2016approximating} and kernel density estimation~\cite{silverman2018density}, it is not as useful in our framework since the evaluation of $h_\epsilon\ast\mu_f$ is not immediately related to pointwise evaluations of the resolvent (see~\cref{sec:rat_kernels}).

Since an $m$th order kernel, $K$, is an approximation to the identity, one can show that $K_\epsilon\ast\mu_f$ converges weakly to $\mu_f$. Moreover, in intervals where $\mu_f$ is absolutely continuous and sufficiently regular, $K_\epsilon\ast\mu_f$ converges pointwise to $\rho_f$ and the rate of convergence increases with the smoothness of $\rho_f$, up to a maximum of $\mathcal{O}(\epsilon^m\log(1/\epsilon))$.

\begin{theorem}
\label{thm:kernel_rates}
Let $K$ be an $m$th order kernel and suppose that the measure $\mu_f$ is absolutely continuous on $I=(x_0-\eta,x_0+\eta)$ for $\eta>0$ and a fixed $x_0\in\mathbb{R}$. Let $\rho_f$ be the Radon--Nikodym derivative of the absolutely continuous component of $\mu_f$, and suppose that $\rho_f\in\mathcal{C}^{n,\alpha}(I)$ with $\alpha\in[0,1)$. Denote the pointwise error by $E_{\epsilon}(x)=\left|\rho_f(x)-[K_\epsilon*\mu_f](x)\right|$. Then it holds that
\begin{enumerate}[leftmargin=*,noitemsep]
	\item[(i)] If $n+\alpha<m$, then, for a constant $C(n,\alpha)$ depending only on $n$ and $\alpha$,
\begin{equation}\label{pt_bd_cite1}
E_{\epsilon}(x_0)\!\leq \!\frac{C_K\epsilon^m}{(\epsilon+\frac{\eta}{2})^{m+1}}+C(n,\alpha)\|\rho_{f}\|_{\mathcal{C}^{n,\alpha}(I)}\!{\int_{\mathbb{R}}\!\!\left|K(y)\right|\!\left|y\right|^{n+\alpha}\!dy}\!\left(1+\eta^{-n-\alpha}\!\right)\!\epsilon^{n+\alpha}\!.
\end{equation}

\item[(ii)] If $n+\alpha \geq m$, then, for a constant $C(m)$ depending only on $m$,
\begin{equation}\label{pt_bd_cite2}
E_{\epsilon}(x_0)\!\leq \!\frac{C_K\epsilon^m}{(\epsilon+\frac{\eta}{2})^{m+1}}+C(m)\|\rho_{f}\|_{\mathcal{C}^{m}(I)}\!\!\left(\!\! C_K\!+\!\!{\int_{-\frac{\eta}{\epsilon}}^{\frac{\eta}{\epsilon}}\!\!\!\left|K(y)\right|\!\left|y\right|^m\!dy}\!\!\right)\!\!\left(1+\eta^{-m}\!\right)\!\epsilon^m.
\end{equation}

\end{enumerate}
Here, $C_K$ is from~\cref{decay_bound}.
\end{theorem}
\begin{proof}
See~\cref{subsec:pointwise_bounds}.
\end{proof} 

Using~\cref{decay_bound} to bound $|K(y)|$ in~\cref{pt_bd_cite1,pt_bd_cite2}, \Cref{thm:kernel_rates} shows that, under local regularity conditions near $x_0\in\mathbb{R}$ and for fixed $\eta>0$, an $m$th order kernel has
\[
\left|\rho_f(x_0)-[K_\epsilon*\mu_f](x_0)\right| = \mathcal{O}(\epsilon^{n+\alpha})+\mathcal{O}(\epsilon^{m}\log(1/\epsilon)),\quad \text{ as} \quad\epsilon\downarrow 0.
\]
The logarithmic term appears in the case that $K(x)x^{m}$ is not integrable.  The upper bounds on $E_\epsilon(x_0)$ in~\cref{thm:kernel_rates} deteriorate as the interval of regularity shrinks ($\eta\rightarrow 0$), which is to be expected.\footnote{Similar results to~\cref{thm:kernel_rates}, without the first term on the right-hand side of~\cref{pt_bd_cite1} and~\cref{pt_bd_cite2}, for absolutely continuous probability measures with globally H\"older continuous density functions are used in kernel density estimation in statistics (see, for example,~\cite[Prop.~1.2]{tsybakov2008introduction}).} 

\subsection{Rational kernels}\label{sec:rat_kernels}
Now that we know the necessary properties of a kernel $K$ so that $K_\epsilon*\mu_f$ achieves high-order convergence (see~\cref{def:mth_order_kernel}), we can develop a resolvent-based approach to approximately evaluate a spectral measure more efficiently. The key to our computational framework (see~\cref{sec:comput_meas}) is the connection between the smoothed measure and the resolvent in~\cref{eqn:pk_identity}. This relation allows us to compute the convolution of the measure $\mu_f$ with the Poisson kernel by evaluating the resolvent operator at the poles of the (rescaled) Poisson kernel. In other words, we can sample the smoothed measure by solving the shifted linear equations in~\cref{eqn:shift_sys}.

Using the identity in~\cref{eqn:fa_identity}, we can build generalizations of~\cref{eqn:pk_identity} for convolutions with rational functions. Suppose that the kernel $K$ is of the form
\begin{equation}
K(x)=\frac{1}{2\pi i}\sum_{j=1}^{n_1}\frac{\alpha_j}{x-a_j}-\frac{1}{2\pi i}\sum_{j=1}^{n_2}\frac{\beta_j}{x-b_j},
\label{eqn:rat_kernel_form}
\end{equation} 
where $a_1,\ldots,a_{n_1}$ are distinct points in the upper half-plane and $b_1,\ldots,b_{n_2}$ are distinct points in the lower half-plane. We restrict $K$ to have only simple poles to avoid having to compute powers of the resolvent. Using~\cref{eqn:fa_identity}, the convolution $K_{\epsilon}*\mu_f$ is given by
\begin{equation}
[K_{\epsilon}*\mu_f](x)=\frac{-1}{2\pi i}\left[\sum_{j=1}^{n_1}\alpha_j\langle \mathcal{R}_{\mathcal{L}}(x-\epsilon a_j)f,f \rangle-\sum_{j=1}^{n_2}\beta_j\langle \mathcal{R}_{\mathcal{L}}(x-\epsilon b_j)f,f \rangle\right].
\label{eq:mthConvFormula}
\end{equation}
Our goal is to choose the poles and residues in~\cref{eqn:rat_kernel_form} so that $K$ is an $m$th order kernel. Given an integer $m\geq 1$, we are interested in finding the smallest possible $n_1$ and $n_2$ in~\cref{eqn:rat_kernel_form} so that~\cref{eq:mthConvFormula} is as efficient to evaluate as possible.

We want $K(x)=\mathcal{O}(|x|^{-(m+1)})$ as $|x|\rightarrow\infty$, which forces linear constraints to hold between the $\alpha_1,\dots,\alpha_{n_2}$ and $\beta_1,\dots,\beta_{n_2}$ parameters, as follows. Generically, $K$ in~\cref{eqn:rat_kernel_form} is a type $(n_1+n_2-1,n_1+n_2)$ rational function, which means it can be written as the quotient of a degree $n_1+n_2-1$ polynomial and a degree $n_1+n_2$ polynomial. In this form, the coefficient of highest power of $x$ in the numerator is a multiple of
$$
\sum_{j=1}^{n_1}\alpha_j-\sum_{j=1}^{n_2}\beta_j,
$$
which must vanish for $K$ to have sufficient decay. Under this condition, we find that 
$$
K(x)x=\frac{1}{2\pi i}\sum_{j=1}^{n_1}\frac{\alpha_ja_j}{x-a_j}-\frac{1}{2\pi i}\sum_{j=1}^{n_2}\frac{\beta_jb_j}{x-b_j}.
$$
We can apply the same argument as before to see that when $m\geq2$, we require that
$$
\sum_{j=1}^{n_1}\alpha_ja_j-\sum_{j=1}^{n_2}\beta_jb_j=0.
$$
We repeat this process $m-1$ times (each time multiplying each term in the sum by the appropriate $a_j$ or $b_j$) to find that $K(x)=\mathcal{O}(|x|^{-(m+1)})$ as $|x|\rightarrow\infty$ if and only if
\begin{equation} 
\sum_{j=1}^{n_1} \alpha_j a_j^k=\sum_{j=1}^{n_2} \beta_j b_j^k, \qquad k=0,\dots,m-1.
\label{eqn:relations} 
\end{equation} 
Assuming~\cref{eqn:relations} is satisfied, the normalization and zero moment conditions (see~\cref{def:mth_order_kernel} (i) and (ii)) provide us with $m$ linear conditions on the moments of $K$, which can be computed explicitly via contour integration. Employing a semi-circle contour in the upper half-plane, applying Cauchy's residue theorem, and taking the radius of the semi-circle to infinity, we find that the moments are given in terms of the poles and residues of $K$, i.e.,
\[
\int_\mathbb{R} K(y)y^k\,dy = \sum_{j=1}^{n_1} \alpha_j a_j^k=\sum_{j=1}^{n_2} \beta_j b_j^k, \qquad k=0,\dots,m-1,
\]
where the second equality follows from~\cref{eqn:relations} or closing the contour in the lower half-plane. Therefore, the rational kernel in~\cref{eqn:rat_kernel_form} is an $m$th order kernel provided that the following (transposed) Vandermonde systems are satisfied:
\begin{equation}\label{eqn:vandermonde_condition}
\begin{pmatrix}
1 & \dots & 1 \\
a_1 & \dots & a_{n_1} \\
\vdots & \ddots & \vdots \\
a_1^{m-1} &  \dots & a_{n_1}^{m-1}
\end{pmatrix}
\!\!
\begin{pmatrix}
\alpha_1 \\ \alpha_2\\ \vdots \\ \alpha_{n_1}
\end{pmatrix}\!
=\!
\begin{pmatrix}
1 & \dots & 1 \\
b_1 & \dots & b_{n_2} \\
\vdots & \ddots & \vdots \\
b_1^{m-1} &  \dots & b_{n_2}^{m-1}
\end{pmatrix}
\!\!
\begin{pmatrix}
\beta_1 \\ \beta_2\\ \vdots \\ \beta_{n_2}
\end{pmatrix}
=\begin{pmatrix}
1 \\ 0 \\ \vdots \\0
\end{pmatrix}\!.
\end{equation}
The systems in~\cref{eqn:vandermonde_condition} are guaranteed to have solutions when $n_1,n_2\geq m$. For computational efficiency, we select $n_1 = n_2 = m$ poles in the upper and lower half-planes. The Poisson kernel fits into this setting with $m=1$, $a_1=\overline{b_1}=i$ and $\alpha_1=\beta_1=1$.

It may appear from~\cref{eq:mthConvFormula} that we need $2m$ resolvent evaluations to evaluate $K_\epsilon*\mu_f$ at a single point $x$. However, if the poles are selected so that $b_j = \overline{a_j}$ and $\beta_j=\overline{\alpha_j}$, then the conjugate symmetry of the resolvent, i.e., $\langle\mathcal{R}_\mathcal{L}(\bar z)f,f\rangle=\overline{\langle\mathcal{R}_\mathcal{L}(z)f,f\rangle}$, reduces the number of resolvent evaluations to $m$. With this choice, we find that
\[
[K_{\epsilon}*\mu_f](x)=\frac{-1}{\pi}\sum_{j=1}^{m}{\rm Im}\left(\alpha_j\,\langle \mathcal{R}_{\mathcal{L}}(x-\epsilon a_j)f,f \rangle\right),
\]
which is analogous to~\cref{eqn:smoothed_meas}. While the properties of an $m$th order kernel determine the number of poles and the residues of $K$ (see~\cref{eqn:vandermonde_condition}), the locations of the poles in the upper half-plane are left to our discretion.

\subsubsection{Equispaced poles}\label{sec:equi_pts}
\begin{table}
\renewcommand{\arraystretch}{1.6}
\centering
\begin{tabular}{l|c|c}
$m$ & $\pi K(x)\prod_{j=1}^m(x-a_j)(x-\overline{a_j})$ & $\{\alpha_1,\ldots,\alpha_{\ceil{m/2}}\}$\\
\hline
$2$ & $\frac{20}{9}$ & $\left\{\frac{1+3i}{2}\right\}$\\
$3$ &$-\frac{5}{4}x^2+\frac{65}{16}$ & $\left\{-2+i,5\right\}$\\
$4$ & $-\frac{3536}{625}x^2+\frac{21216}{3125}$ & $\left\{\frac{-39-65i}{24},\frac{17+85i}{8}\right\}$\\
$5$ & $\frac{130}{81}x^4 - \frac{12350}{729}x^2 + \frac{70720}{6561}$ & $\left\{\frac{15-10i}{4},\frac{-39+13i}{2},\frac{65}{2}\right\}$\\
$6$ & $\frac{1287600}{117649}x^4 - \frac{34336000}{823543}x^2 + \frac{667835200}{40353607}$ & $\left\{\frac{725+1015i}{192},\frac{-2775-6475i}{192},\frac{1073+7511i}{96}\right\}$\\
\end{tabular}
\caption{The numerators and residues of the first six rational kernels with equispaced poles (see~\cref{eqn:equi_poles}). We give the first ${\ceil{m/2}}$ residues because the others follow by the symmetry $\alpha_{m+1-j}=\overline{\alpha_j}$.\label{tab_kernel}}
\end{table}
\renewcommand{\arraystretch}{1}

As a natural extension of the Poisson kernel, whose two poles are at $\pm i$, we consider the family of $m$th order kernels with equispaced poles in the upper and lower half-planes given by
\begin{equation}\label{eqn:equi_poles}
a_j=\frac{2j}{m+1}-1+i, \qquad b_j=\overline{a_j}, \qquad 1\leq j\leq m.
\end{equation}
We then determine the residues by solving the Vandermonde system in~\cref{eqn:vandermonde_condition}. The first six kernels are plotted in~\cref{fig:kernels} (left) and are explicitly written down in~\cref{tab_kernel}. 

Empirically, we found that the choice in~\cref{eqn:equi_poles} performed slightly better than other natural choices such as Chebyshev points with an offset $+i$, rotated roots of unity or dyadic poles $a_j = i 2^{-j}$. Dyadic poles have the advantage that if $\epsilon$ is halved, the resolvent only needs to be computed at one additional point. The ill-conditioning of the Vandermonde system did not play a role for the values of $m$ here. Moreover, equispaced poles are particularly useful when one wishes to sample the smoothed measure $K_\epsilon*\mu_f$ over an interval since samples of the resolvent can be reused for different points in the interval. Finally, if $\epsilon$ is found to be insufficiently small, instead of re-evaluating the resolvent at $m$ points, one can add poles closer to the real axis (with a smaller $\epsilon$) and reuse the old resolvent evaluations. This effectively increases $m$, and hence the coefficients $\alpha_j$ need to be recomputed. This may be computationally beneficial since the cost of solving the Vandermonde system is typically negligible compared to the cost of evaluating the resolvent close to the real axis.

\begin{figure}[!tbp]
  \centering
  \begin{minipage}[b]{0.49\textwidth}
    \begin{overpic}[width=\textwidth]{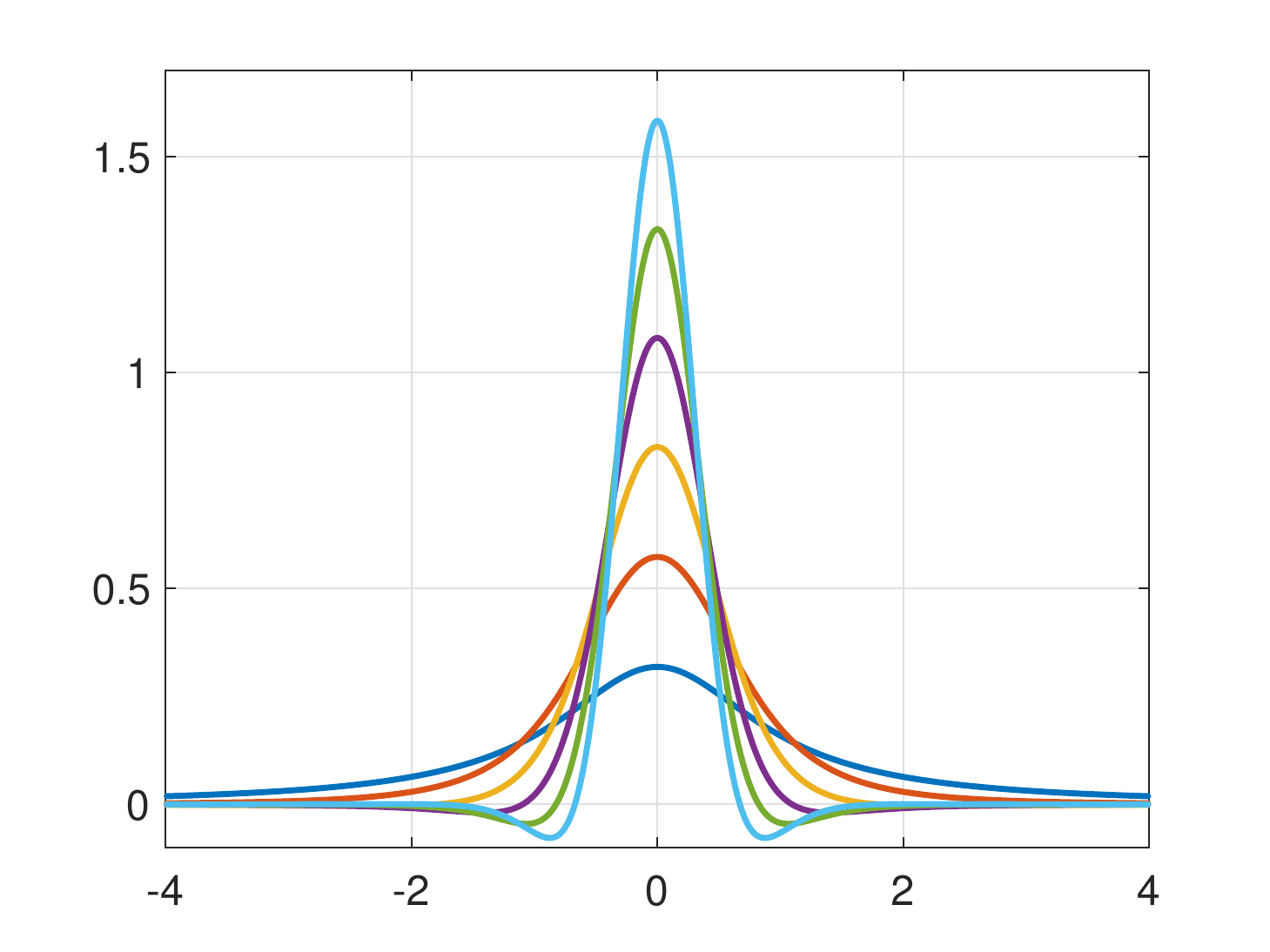}
 	\put (47,73) {$\displaystyle K(x)$}
 	\put (50,-2) {$\displaystyle x$}
	\put (66,23) { {$\displaystyle m=1$}}
	\put(66,23)  {\vector(-1,0){14}}
	\put (19,31.5) { {$\displaystyle m=2$}}
	\put(37.5,31.5)  {\vector(1,0){14}}
	\put (66,40) { {$\displaystyle m=3$}}
	\put(66,40)  {\vector(-1,0){14}}
	\put (19,48) { {$\displaystyle m=4$}}
	\put(37.5,48)  {\vector(1,0){14}}
	\put (66,57) { {$\displaystyle m=5$}}
	\put(66,57)  {\vector(-1,0){14}}
	\put (19,65.5) { {$\displaystyle m=6$}}
	\put(37.5,65.5)  {\vector(1,0){14}}
 	\end{overpic}
  \end{minipage}
\hfill
  \begin{minipage}[b]{0.48\textwidth}
    \begin{overpic}[width=\textwidth]{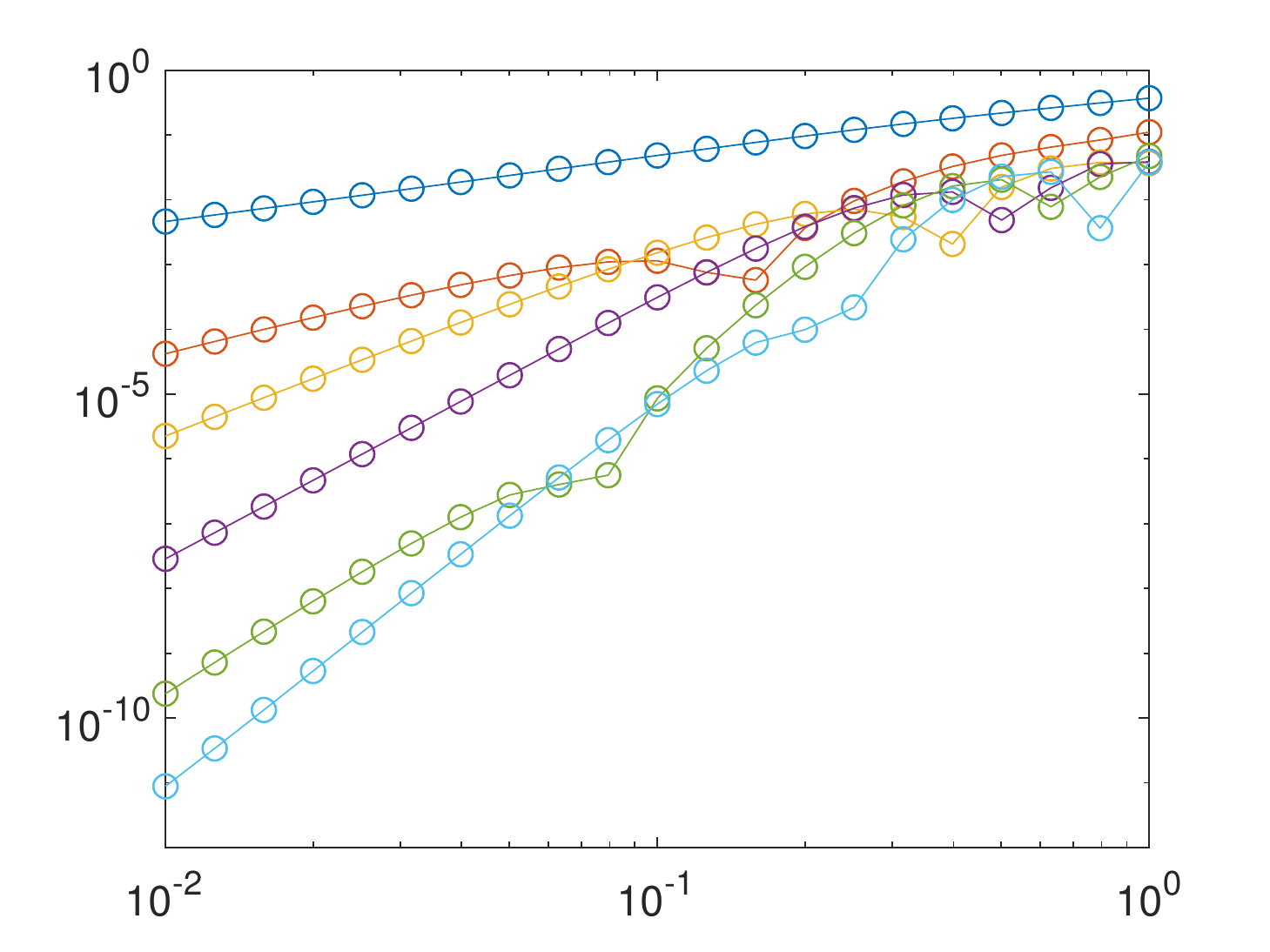}
 	\put (12,73) {$\displaystyle |\rho_f(x_0)-[K_\epsilon*\mu_f](x_0)|/|\rho_f(x_0)|$}
 	\put (50,-2) {$\displaystyle \epsilon$}
		\put (14,60) {\rotatebox{6} {$\displaystyle m=1$}}
	\put (14,49) {\rotatebox{14} {$\displaystyle m=2$}}
	\put (14,37) {\rotatebox{21} {$\displaystyle m=3$}}
	\put (14,27.5) {\rotatebox{29} {$\displaystyle m=4$}}
	\put (30,33) {\rotatebox{29} {$\displaystyle m=5$}}
	\put (14,10) {\rotatebox{38} {$\displaystyle m=6$}}
 	\end{overpic}
  \end{minipage}
  \caption{Left: The $m$th order kernels constructed from~\cref{eqn:vandermonde_condition} with poles in~\cref{eqn:equi_poles} for $1\leq m\leq 6$. Right: The pointwise relative error in smoothed measures of the integral operator in~\cref{eqn:fred_int} computed using the high-order kernels with poles in~\cref{eqn:equi_poles} for $1\leq m\leq 6$. The relative error is computed by comparing with a numerical solution that is resolved to machine precision.\label{fig:kernels}}
\end{figure}

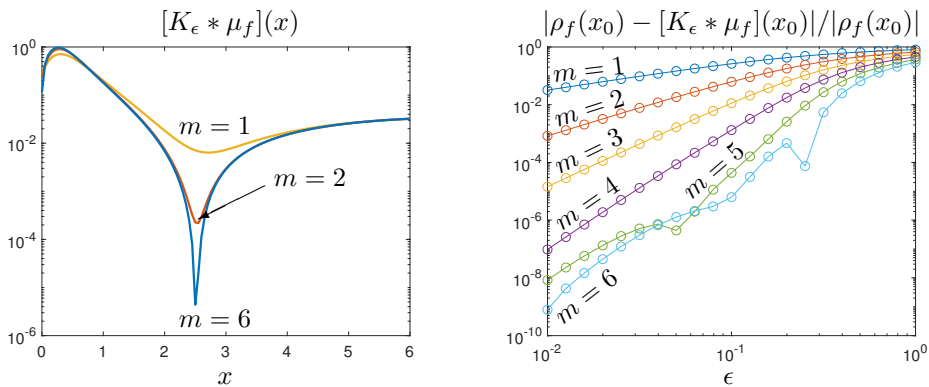
\begin{figure}[!tbp]
  \centering
   \begin{minipage}[b]{0.48\textwidth}
    \begin{overpic}[width=\textwidth]{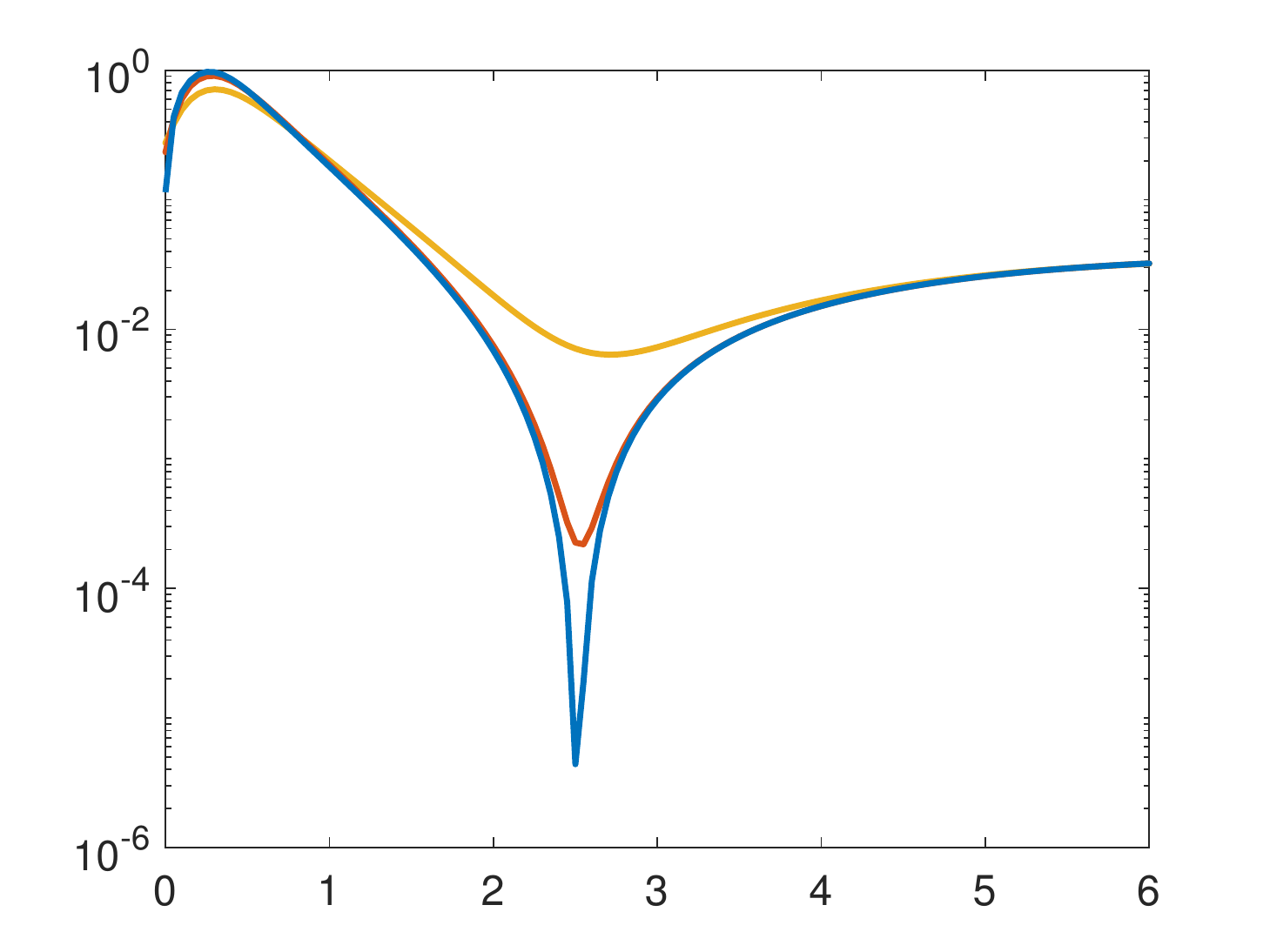}
 	\put (38,73) {$\displaystyle [K_\epsilon*\mu_f](x)$}
 	\put (50,-2) {$\displaystyle x$}
	\put (40,51) { {$\displaystyle m=1$}}
	\put (60,40) { {$\displaystyle m=2$}}
	\put(60,40)  {\vector(-2,-1){14}}
	\put (40,10) { {$\displaystyle m=6$}}
 	\end{overpic}
  \end{minipage}
  \hfill
  \begin{minipage}[b]{0.48\textwidth}
    \begin{overpic}[width=\textwidth]{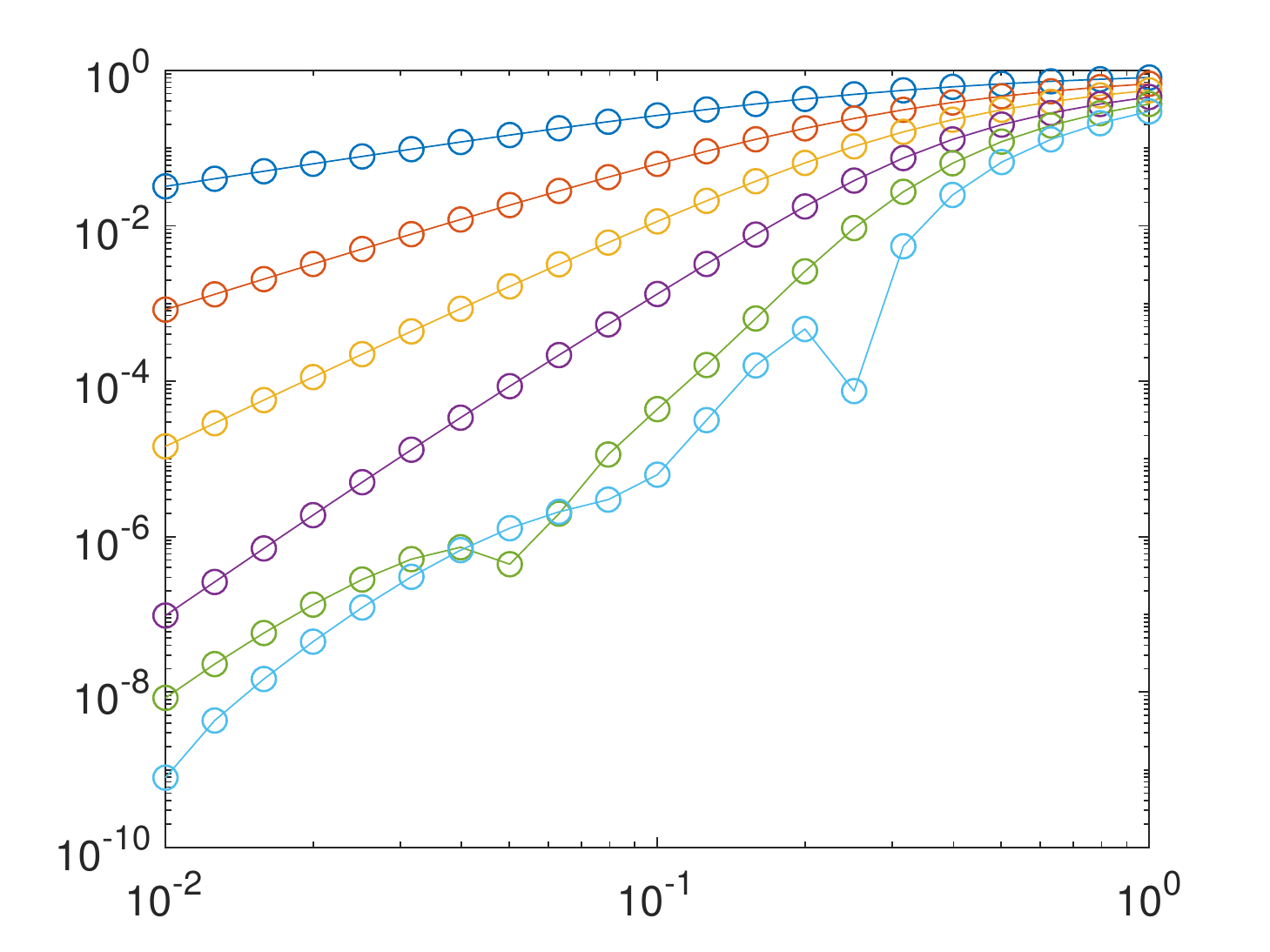}
    \put (12,73) {$\displaystyle |\rho_f(x_0)-[K_\epsilon*\mu_f](x_0)|/|\rho_f(x_0)|$}
 	\put (50,-2) {$\displaystyle \epsilon$}
	\put (14,62) {\rotatebox{7} {$\displaystyle m=1$}}
	\put (14,52.5) {\rotatebox{18} {$\displaystyle m=2$}}
	\put (14,42.5) {\rotatebox{26} {$\displaystyle m=3$}}
	\put (14,30.5) {\rotatebox{35} {$\displaystyle m=4$}}
	\put (41,37) {\rotatebox{41} {$\displaystyle m=5$}}
	\put (14,10.5) {\rotatebox{38} {$\displaystyle m=6$}}
 	\end{overpic}
  \end{minipage}
  \caption{Results for the Schr\"odinger operator in~\cref{eqn:schrodinger} using $m$th order kernels with equispaced poles (see~\cref{eqn:equi_poles}). Left: Smoothed approximations to the spectral measure. Right: Pointwise relative error, computed by comparing with a numerical solution resolved to machine precision.\label{fig:kernels2}}
\end{figure}

To demonstrate the practical advantage of high-order kernels, we revisit the examples from~\cref{sec:comput_meas} and compute the smoothed measure $K_\epsilon*\mu_f$ using $m$th order kernels with equispaced poles. In~\cref{fig:kernels} (right) and~\cref{fig:kernels2} (right), we observe the convergence rates predicted in~\cref{thm:kernel_rates} for the integral operator in~\cref{eqn:fred_int} and the differential operator in~\cref{eqn:schrodinger}, respectively. While the Poisson kernel requires us to solve linear equations with shifts extremely close to the continuous spectrum to achieve a few digits of accuracy in our approximation to $\rho_f$, a sixth-order kernel enables us to achieve about $11$ and $9$ digits of accuracy, respectively, without decreasing $\epsilon$ below $0.01$.~\Cref{fig:kernels2} (left) shows the increased resolution obtained when using high-order kernels for the differential operator in~\cref{eqn:schrodinger} with smoothing parameter $\epsilon=0.1$. Although using a sixth-order kernel requires six times as many resolvent evaluations as that of the Poisson kernel, this is typically favorable because the cost of evaluating the resolvent near the continuous spectrum of $\mathcal{L}$ increases as $\epsilon\downarrow 0$ (see~\cref{num_balance_act}). 

\begin{figure}[!tbp]
  \centering
  \begin{minipage}[b]{0.48\textwidth}
    \begin{overpic}[width=\textwidth]{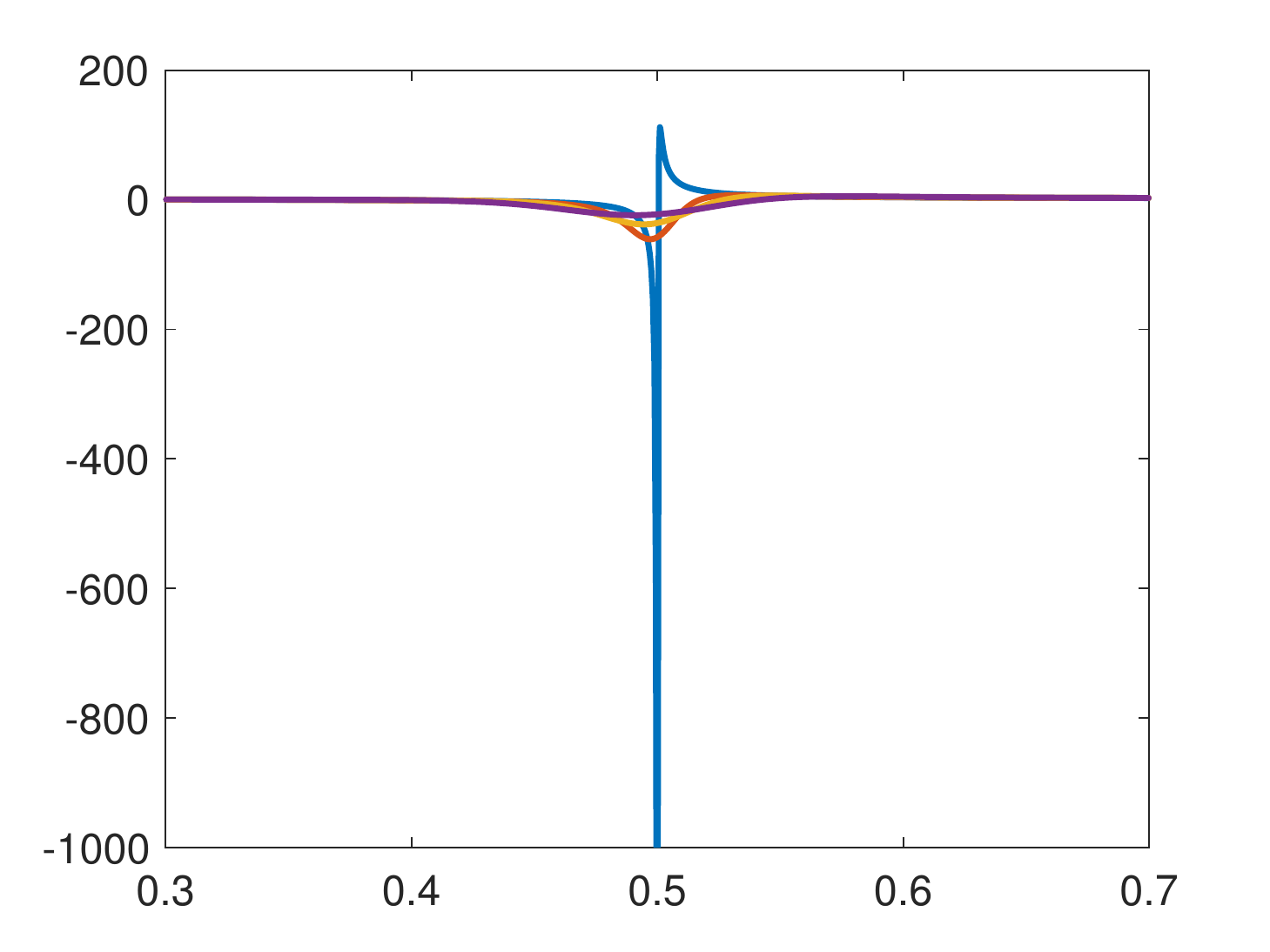}
 	\put (50,-2) {$\displaystyle x$}
 	\end{overpic}
  \end{minipage}
  \hfill
  \begin{minipage}[b]{0.48\textwidth}
    \begin{overpic}[width=\textwidth]{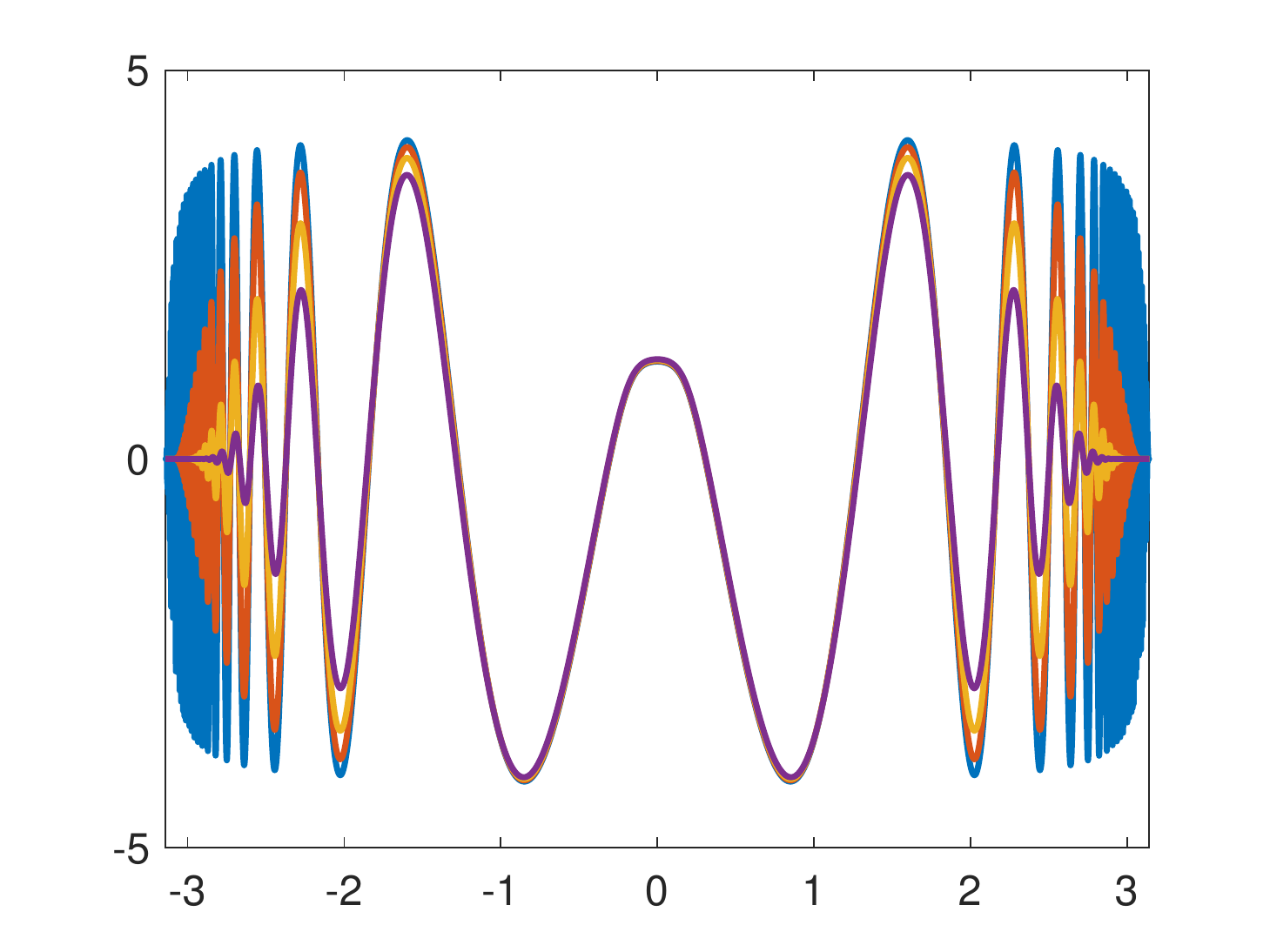}
 	\put (50,-2) {$\displaystyle \theta$}
 	\end{overpic}
  \end{minipage}
  \caption{Real part of $\sum_{j=1}^m\beta_j\mathcal{R}_{\mathcal{L}}(x-\epsilon b_j)f$, where $\epsilon$ is chosen to achieve a relative error of $0.0001$ for the integral operator in~\cref{eqn:fred_int} (left) and $0.005$ for the Schr\"odinger operator in~\cref{eqn:schrodinger} (right), with $m=1$ (blue), $m=2$ (yellow), $m=3$ (orange), and $m=4$ (purple). Recall that the solutions to~\cref{eqn:schrodinger} are mapped to $[-\pi,\pi]$ via $x=10i(1-e^{i\theta})/(1+e^{i\theta})$.\label{fig:singular_solns_new}}
\end{figure}

In~\cref{fig:singular_solns_new} (which should be compared to~\cref{fig:singular_solns}), we plot the real part of the linear combination of solutions, given by
$$
\mathrm{Re}\left(\sum_{j=1}^m\beta_j\mathcal{R}_{\mathcal{L}}(x-\epsilon b_j)f\right).
$$
Here, $\epsilon$ is selected to achieve a relative error of $0.0001$ and $0.005$ in the density of the integral and Schr\"odinger operators, respectively. For a fixed relative error, the high-order kernels lead to numerical solutions that are less peaked (or less oscillatory), which allows us to use much smaller discretizations of the linear operators.

\subsection{Other types of convergence}\label{sec:alt_convergence}
Consider the radial Schr\"odinger operator with a Hellmann potential and angular momentum quantum number $\ell$, given by~\cite{hellmann1935new}
\begin{equation}\label{eqn:RSE_op}
[\mathcal{L}u](r)=-\frac{d^2u}{dr^2}(r)+\left(\frac{\ell(\ell+1)}{r^2}+\frac{1}{r}(e^{-r}-1)\right)u(r), \qquad r>0.
\end{equation}
The spectral properties of $\mathcal{L}$ are of interest in quantum chemistry, where the Hellman potential models atomic and molecular ionization processes~\cite{hamzavi2013approximate}. Ionization rates and related transition probabilities are usually studied by computing bound and resonant states of $\mathcal{L}$; however, we compute this information directly from the spectral measure.

For example, if $\smash{f(r)=Ce^{-(r-r_0)^2}}$ (where $C$ is chosen so that $\|f\|_{L^2(\mathbb{R}_+)}=1$) is the radial component of the wave function of an electron interacting with an atomic core via the Hellmann potential in~\cref{eqn:RSE_op}, then we can calculate the probability that the electron escapes from the atomic core with energy $E\in[a,b]$ (with $0<a<b$) via
\begin{equation}\label{eqn:ion_prob}
\mathbb{P}(a\leq E\leq b)=\mu_f([a,b])\approx\int_a^b\,[K_\epsilon*\mu_f](y)\,dy, \qquad \epsilon\ll 1.
\end{equation}
The error for the approximation in~\cref{eqn:ion_prob} is bounded above by
\[
\left|\mu_f([a,b])-\int_a^b\,[K_\epsilon*\mu_f](y)\,dy\right|\leq \int_a^b|\rho_f(y)-[K_\epsilon*\mu_f](y)|\,dy =
\|\rho_f-K_\epsilon*\mu_f\|_{L^1([a,b])}.
\]
This leads us naturally to the notion of $L^p$ convergence on an interval. The smoothed measure always converges to $\rho_f$ in $L^1([a,b])$ when $\mu_f$ is absolutely continuous on $[a,b]$. However, in analogy with the pointwise results in~\cref{sec:smooth_approx} and~\cref{sec:HigherOrderKernels}, we need to impose some additional regularity on $\rho_f$ to obtain rates of convergence. We let $\mathcal{W}^{k,p}(I)$ denote the Sobolev space of functions in $L^{p}(I)$ such that $f$ and its weak derivatives up to order $k$ have a finite $L^p$ norm~\cite{evans2010partial}.
\begin{figure}[!tbp]
  \centering
  \begin{minipage}[b]{0.48\textwidth}
    \begin{overpic}[width=\textwidth]{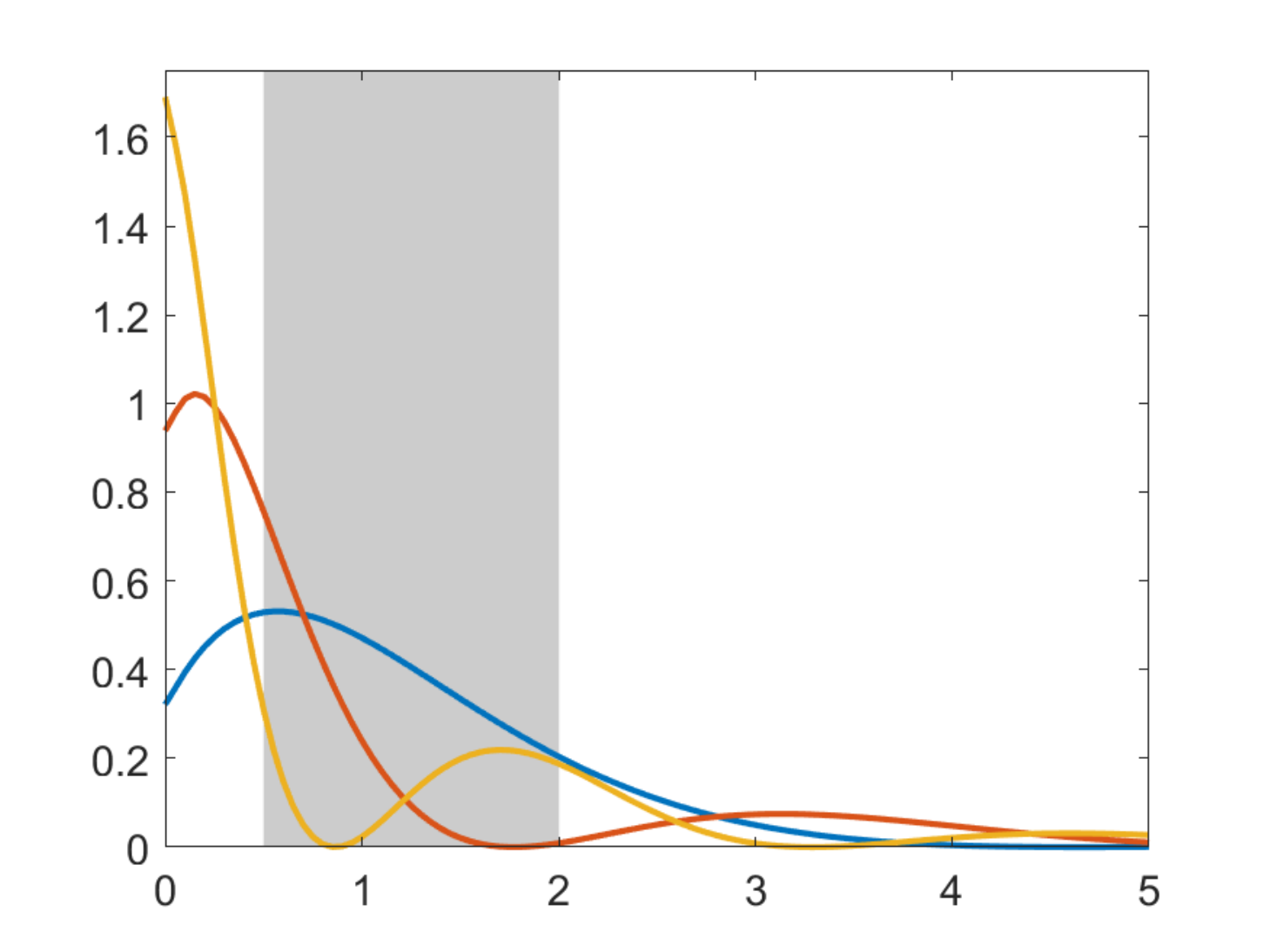}
 	\put (38,73) {$\displaystyle [K_\epsilon*\mu_f](x)$}
 	\put (50,-2) {$\displaystyle x$}
 	\put (16,68) {\rotatebox{-80} {$\displaystyle r_0=4$}}
 	\put (20,45)  {\rotatebox{-67} {$\displaystyle r_0=3$}}
 	\put (32,28) {\rotatebox{-35} {$\displaystyle r_0=2$}}
 	\end{overpic}
  \end{minipage}
  \hfill
  \begin{minipage}[b]{0.48\textwidth}
    \begin{overpic}[width=\textwidth]{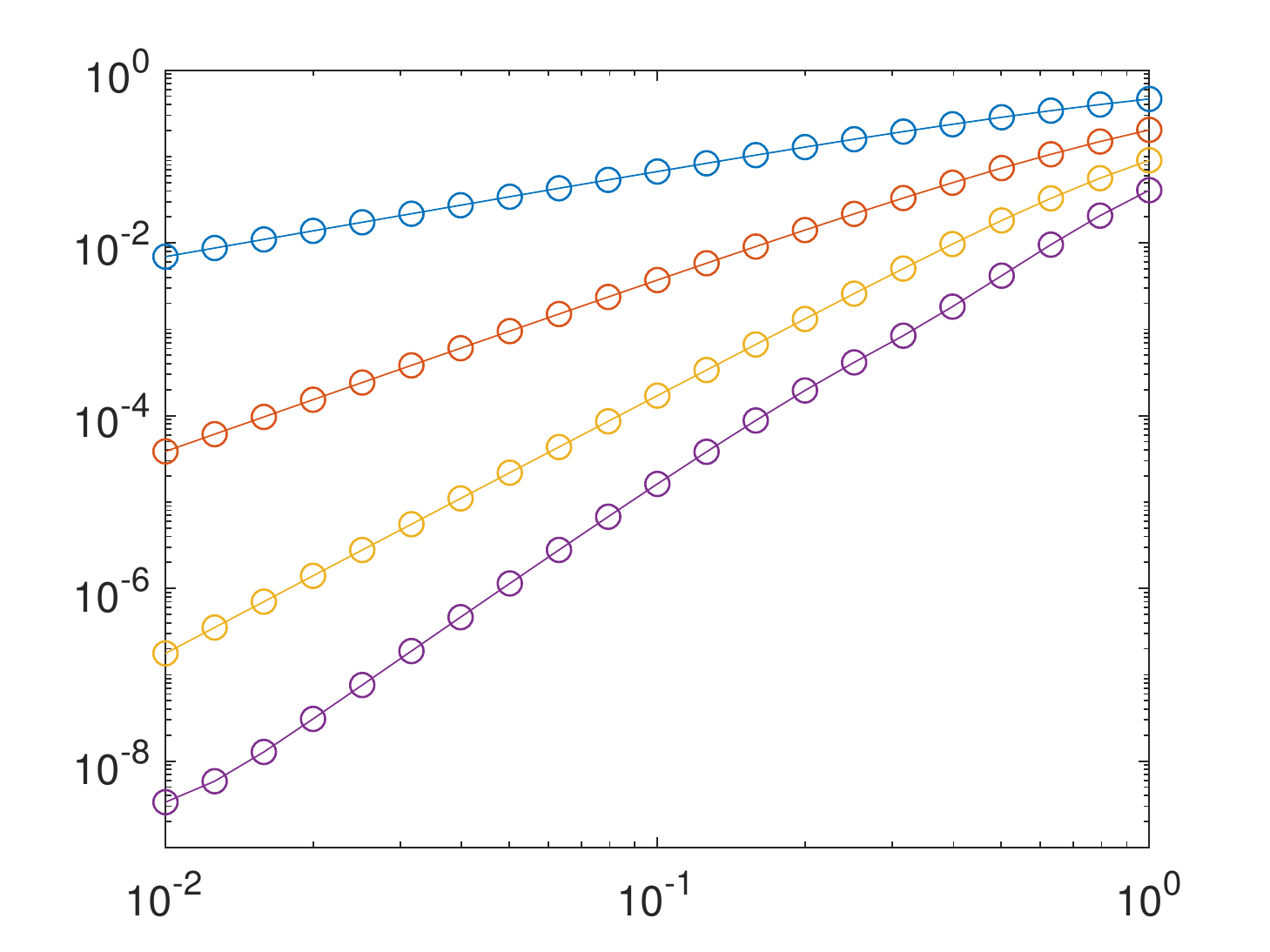}
    \put (19,73) {$\displaystyle \|\rho_f-[K_\epsilon*\mu_f]\|_{L^1}/\|\rho_f\|_{L^1}$}
 	\put (14,58) {\rotatebox{7} {$\displaystyle m=1$}}
	\put (14,43) {\rotatebox{18} {$\displaystyle m=2$}}
	\put (14,28) {\rotatebox{26} {$\displaystyle m=3$}}
	\put (14,16) {\rotatebox{32} {$\displaystyle m=4$}}
 	\put (50,-2) {$\displaystyle \epsilon$}
 	\end{overpic}
  \end{minipage}
  \caption{Left: The smoothed approximation to the density on the absolutely continuous spectrum of $\mathcal{L}$ in~\cref{eqn:RSE_op}, with $\smash{f_{r_0}(r)=C_{r_0}e^{-(r-r_0)^2}}$ and $\ell=1$, for $r_0=2$, $r_0=3$, and $r_0=4$ ($C_{r_0}$ is a normalization constant so that $\smash{\|f_{r_0}\|_{L^2(\mathbb{R}_+)}=1}$). The shaded area under each curve corresponds to $\mathbb{P}(1/2\leq E\leq 2)$ in~\cref{eqn:ion_prob} for the particle with wave function $\smash{f_{r_0}(r)}$.  Right: The $L^1((1/2,2))$ relative error in smoothed measures for the radial Schr\"odinger operator in~\cref{eqn:RSE_op}. The relative error is computed by comparing with a numerical solution that is resolved to machine precision. \label{fig:RSE_measure}}
\end{figure}
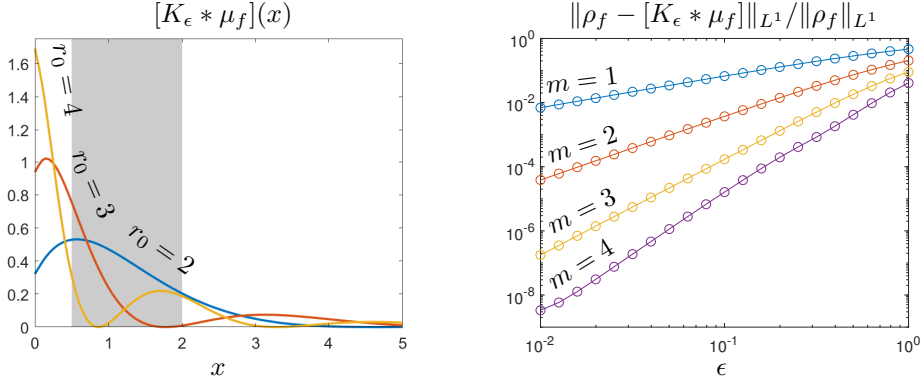

\begin{theorem}\label{thm:Lp_rates}
Let $K$ be an $m$th order kernel and $1\leq p<\infty$. Suppose that the measure $\mu_f$ is absolutely continuous on the interval $I=(a-\eta,b+\eta)$ for $\eta>0$ and some $a<b$. Let $\rho_f$ denote the Radon--Nikodym derivative of the absolutely continuous component of $\mu_f$, and suppose that $\rho_I:=\rho_f|_{I}\in \mathcal{W}^{m,p}(I)$. Then,
\begin{equation*}
\begin{aligned}
\|\rho_I-[K_\epsilon*\mu_f]\|_{L^p((a,b))} \leq &\frac{C_K(b-a)^{1/p}}{\left(\epsilon+\frac{\eta}{2}\right)^{m+1}}\epsilon^m\\&\hspace{-5mm}+C(m)C_K\|\rho_I\|_{\mathcal{W}^{k,p}(I)}\left(1+\eta^{-m}\right)\log\!\left(1+\frac{b-a+2\eta}{\epsilon}\right)\epsilon^m,
\end{aligned}
\end{equation*}
where $C(m)$ is a constant depending only on $m$, and $C_K$ is from~\cref{decay_bound}.
\end{theorem}
\begin{proof}
See~\cref{subsec:Lp_bounds}.
\end{proof}

\Cref{thm:Lp_rates} implies the asymptotic error rate\footnote{\Cref{thm:Lp_rates} for $p=2$ without the first term on the right-hand side and for absolutely continuous probability measures with $\mathcal{W}^{m,2}(\mathbb{R})$ density function is used in kernel density estimation in statistics~\cite[Prop.~1.5]{tsybakov2008introduction}. In this context, the $L^2$ error is used to bound the bias term in the mean integrated squared error. The case of $L^1$ convergence requires a different proof technique.}
\[
\left\|\rho_I-[K_\epsilon*\mu_f]\right\|_{L^p(I)}=\mathcal{O}(\epsilon^m\log(1/\epsilon)), \qquad \text{as}\quad \epsilon\downarrow 0.
\]
The $L^1$ convergence for the approximation to the probabilities in~\cref{eqn:ion_prob} is shown in~\cref{fig:RSE_measure} (right), which agrees with the asymptotic rates implied by~\cref{thm:Lp_rates}.

If one wishes to compute dynamics of the electron interacting with the atomic core via the Hellman potential, then we need a slightly weaker form of convergence. For instance, the time autocorrelation of the electron's wave function can be computed by integrating the function $F_t(E)=e^{-iEt}$ against the measure $\mu_f$, so that
$$
\mu_f(F_t) = \langle e^{-i\mathcal{L}t}f,f\rangle=\int_{-\infty}^\infty e^{-iyt}\,d\mu_f(y) \approx \int_{-\infty}^\infty  e^{-iyt}\,[K_\epsilon*\mu_f](y)\,dy, \quad \epsilon\ll 1.
$$
Unlike the previous cases of pointwise and $L^1$ convergence, we do not need any additional requirements on the measure $\mu_f$, which may be singular and have discrete components, to obtain convergence rates. Instead, we require that the function $F$ be sufficiently smooth. For example, if $F\in \mathcal{C}^{n,\alpha}(\mathbb{R})$ and $K$ is an $m$th order kernel, then approximating $F$ via convolutions and applying Fubini's theorem shows that
$$
|\mu_f(F)-[K_\epsilon*\mu_f](F)|=\mathcal{O}(\epsilon^{n+\alpha})+\mathcal{O}(\epsilon^m\log(1/\epsilon)), \qquad \text{as} \quad \epsilon\downarrow 0.
$$

Finally, note that a kernel cannot be non-negative everywhere and have an order greater than two. This is not a problem in practice since we can replace $[K_\epsilon*\mu_f](x)$ by $\max\{0,[K_\epsilon*\mu_f](x)\}$ with the same error bounds in~\cref{thm:kernel_rates,thm:Lp_rates}.

\section{The resolvent framework in practice}\label{sec:practical_considerations}
Given an $m$th order rational kernel, defined by distinct poles $a_1,\dots,a_m$ in the upper half-plane, the resolvent-based framework for evaluating an approximate spectral measure is summarized in~\cref{alg:spec_meas}. This algorithm, which can be performed in parallel for several $x_0$, forms the foundation of \texttt{SpecSolve}. \texttt{SpecSolve} uses equispaced poles (see~\cref{sec:equi_pts}) by default, but users may select other options with the name-value pair \texttt{`PoleType'}.

In practice, the resolvent in~\cref{alg:spec_meas} is discretized before being applied. We compute an accurate value of $\mu_f^\epsilon$ provided that the resolvent is applied with sufficient accuracy (see~\cref{fig:fred_int_meas}), which can be done {\em{}adaptively} with {\em{}a posteriori} error bounds~\cite{colbrook2019computing}. For an efficient adaptive implementation, \texttt{SpecSolve} constructs a fixed discretization, solves linear systems at each required complex shift, and checks the approximation error at each shift. If further accuracy is needed at a subset of the shifts, then the discretization is refined geometrically, applied at these shifts, and the error is recomputed. This process is repeated until the resolvent is computed accurately at all shifts. The user may (optionally) specify initial and maximum discretization sizes with the name-value pairs \texttt{`DiscMin'} and \texttt{`DiscMax'}.

\texttt{SpecSolve} supports three types of operators: (1) ordinary differential operators, (2) integral operators, and (3) infinite matrices with finitely many non-zeros per column. For more general operators and inner products, the user must supply a command that solves the shifted linear equations in~\cref{alg:spec_meas} and a command that evaluates the inner products, allowing a user to evaluate spectral measures for exotic problems and employ their favorite discretization.

\begin{algorithm}[t]
\textbf{Input:} $\mathcal{L}:\mathcal{D}(\mathcal{L})\rightarrow\mathcal{H}$, $f\in\mathcal{H}$, $x_0\in\mathbb{R}$, $a_1,\dots,a_m\in\{z\in\mathbb{C}:{\rm Im}(z)>0\}$, and $\epsilon>0$. \\
\vspace{-4mm}
\begin{algorithmic}[1]
\STATE Solve the Vandermonde system~\cref{eqn:vandermonde_condition} for the residues $\alpha_1,\dots,\alpha_m\in\mathbb{C}$.
	\STATE Solve $(\mathcal{L}-(x_0-\epsilon a_j))u_{j}^\epsilon=f$ for $1\leq j\leq m$.
	\STATE Compute $\mu_f^\epsilon(x_0)=\frac{-1}{\pi}\mathrm{Im}\left(\sum_{j=1}^m\alpha_j\langle u_{j}^\epsilon,f\rangle\right)$.
\end{algorithmic} \textbf{Output:} $\mu_f^\epsilon(x_0)$.
\caption{A practical framework for evaluating an approximate spectral measure of an operator $\mathcal{L}$ at $x_0\in\mathbb{R}$ with respect to a vector $f\in\mathcal{H}$.}\label{alg:spec_meas}
\end{algorithm}

\subsection{Ordinary differential operators}\label{specfun_ODE_meth}
As part of its capabilities, \texttt{SpecSolve} computes samples from a smoothed approximation to the spectral measure of a self-adjoint, regular ordinary differential operator on the real-line or on the half-line, i.e.,
\begin{equation}\label{eqn:diff_op}
[\mathcal{L}u](x)=c_{p}(x)\frac{d^{p}u}{dx^{p}}(x)+\cdots+c_1(x)\frac{du}{dx}(x)+c_0(x)u(x), \qquad p\geq 0,
\end{equation}
with the standard inner products. Here, the variable coefficients $c_0,\ldots,c_{p}$ are smooth functions and $c_{p}\neq 0$ on the relevant domain (real-line or half-line). Note that $\mathcal{L}$ in~\cref{eqn:diff_op} is not necessarily self-adjoint: the user provides the variable coefficients $c_0,\ldots,c_{p}$ and must verify that $\mathcal{L}$ is self-adjoint.

To demonstrate, recall the Schr\"odinger operator defined on the real line in~\cref{eqn:schrodinger}. We can compute a smoothed approximation to its spectral measure using the function \texttt{diffMeas} as follows.
\begin{verbatim}
  xi = linspace(0,6,121);                          % Evaluation pts
  f = @(x) x.^2./(1+x.^6) * sqrt(9/pi);            % Measure wrt f(x)
  c = {@(x) x.^2./(1+x.^6), @(x) 0, @(x) -1};      % Schrodinger op
  mu = diffMeas(c, f, xi, 0.1, `order', 1);        % epsilon=0.1, m=1
\end{verbatim}
The differential operator is specified by its coefficients $c_0,\ldots,c_2$, which are input as a cell array of function handles. Given evaluation points \texttt{xi} and function handle \texttt{f}, \texttt{diffMeas} computes the smoothed measure, with respect to \texttt{f}, using the specified smoothing parameter and kernel order (the default kernel is $m=2$).% To work on the half-line, the user simply adds \texttt{`dom',`half'} to the argument list for \texttt{diffMeas}.}

To apply the resolvent of a differential operator acting on functions on the real line, the associated differential equation (see~\cref{alg:spec_meas}) is automatically transplanted to the periodic interval $[-\pi,\pi]$ with an analytic map and solved with an adaptive Fourier spectral method~\cite{boyd2001chebyshev}. Typically, the differential equation has singular points at $\pm\pi$ after mapping, and the Fourier spectral method usually converges to a bounded analytic solution~\cite[Ch.~17.8]{boyd2001chebyshev}. Similarly, on the half-line, the differential equation is mapped to the unit interval $[-1,1]$ with an analytic map and solved with an adaptive nonperiodic analogue of the Fourier spectral method known as the ultraspherical spectral method~\cite{olver2013fast}. After solving the differential equation on the mapped domain, the inner products in~\cref{eqn:smoothed_meas} are computed using a trapezoidal rule (for the unit circle)~\cite{trefethen2014exponentially} or a Clenshaw--Curtis rule (for the unit interval)~\cite[Ch.~19]{trefethen2019approximation}.

In many applications, differential operators on the half-line may have a singular point at the origin. This makes an efficient and automatic representation of variable coefficients somewhat subtle. For example, the radial Schr\"odinger operator in~\cref{eqn:RSE_op} has a singular point at the origin for $\ell\geq 1$, and the shifted linear equations in~\cref{alg:spec_meas} should be multiplied through by $r^2$ so that subsequent discretizations yield sparse, banded matrices~\cite{olver2013fast}. In addition to \texttt{diffMeas}, \texttt{SpecSolve} contains a small gallery of functions that sample smoothed spectral measures for common operators with singular points, such as \texttt{rseMeas}, which samples the smoothed measure of the radial Schr\"odinger operator with a user-specified potential.

To illustrate, we use \texttt{rseMeas} to compute $\mathbb{P}(1/2\leq E\leq 2)$ from~\cref{eqn:ion_prob}.
\begin{verbatim}
  normf = sqrt(pi/8)*(2-igamma(1/2,8)/gamma(1/2)); % Normalization
  f = @(r) exp(-(r-2).^2)/sqrt(normf);             % Measure wrt f(r)
  V={@(r) 0, @(r) exp(-r)-1, 1};                   % Potential, l=1
  [xi, wi] = chebpts(20,  [1/2 2]);                % Quadrature rule
  mu = rseMeas(V, f, xi, 0.1, `Order', 4)          % epsilon=0.1, m=4
  ion_prob = wi * mu;                              % Ionization prob
\end{verbatim}
The user specifies the potential of the radial Schr\"odinger operator through a cell array of function handles: \texttt{V\{1\}} is the nonsingular part of the potential, \texttt{V\{2\}} is the variable coefficient for the $r^{-1}$ Coulomb term, and \texttt{V\{3\}} is the quantum angular momentum number that defines the coefficient for the $r^{-2}$ centrifugal term.

\subsection{Integral operators}\label{sec:fred_disc}
In \texttt{SpecSolve}, the function \texttt{intMeas} computes samples from a smoothed approximation of the spectral measure of an integral operator, acting on functions defined on $[-1,1]$, of the form
\[
[\mathcal{L}u](x)=a(x)u(x)+\int_{-1}^1 g(x,y)\,u(y)\,dy, \qquad x\in [-1,1], \qquad u\in L^2([-1,1]).
\]
We assume that the multiplicative coefficient $a(x)$ and the kernel $g(x,y)$ are smooth functions (well-approximated by polynomials), and that $g(x,y)=\overline{g(y,x)}$ so that $\mathcal{L}$ is self-adjoint with respect to the standard inner product. Revisiting the integral operator from~\cref{eqn:fred_int}, we can compute the smoothed measure with a few simple commands.
\begin{verbatim}
  xi = linspace(-2.5,2.5,501);                    % Evaluation pts
  f = @(x) sqrt( 3/2 ) * x;                       % Measure wrt f(x)
  a = { @(x) x, @(x,y) exp(-(x.^2+y.^2)) };       % Integral operator
  mu = intMeas(a, f, xi, 0.1, `Order', 1);        % epsilon=0.1, m=1
\end{verbatim}
The integral operator is specified by a cell array containing function handles for the kernel and multiplicative coefficient. Given smoothing parameter and kernel order, the smoothed measure is approximated at the evaluation points \texttt{xi}.

To apply the resolvent, we use an adaptive Chebyshev collocation scheme to solve the shifted linear systems in~\cref{alg:spec_meas}. For efficient storage and computation, we exploit low numerical rank structure in the discretization of the smooth kernels when possible~\cite{townsend2013extension}. We apply a Clenshaw--Curtis quadrature rule to compute the inner products required to sample $\mu_f^\epsilon$~\cite{trefethen2019approximation}.

\subsection{Infinite sparse matrices}\label{discrete_specfun_sec}
In \texttt{SpecSolve}, the function \texttt{infmatMeas} deals with discrete systems. We consider the canonical Hilbert space $\ell^2(\mathbb{N})$ (with the standard inner product) and assume that $\mathcal{L}$ is realized as an infinite matrix $A$ such that 
$$
A=
\begin{pmatrix} 
a_{11} & a_{12}  &\hdots \\
a_{21} & a_{22} &\hdots\\
\vdots & \vdots & \ddots 
\end{pmatrix}, \qquad a_{ij}=\langle \mathcal{L}e_j,e_i\rangle=\overline{a_{ji}}, 
$$
where $e_i$ is the $i$th canonical unit vector. We assume that the span of the canonical basis forms a core\footnote{This technical condition means that the closure of $\mathcal{L}$ restricted to the span of the canonical basis is $\mathcal{L}$, and hence we can equate $\mathcal{L}$ with the infinite matrix $A$.} of $\mathcal{L}$ and that there is known function $F:\mathbb{N}\rightarrow\mathbb{N}$ such that $a_{ij}=0$ if $i>F(j)$.\footnote{Weaker assumptions such as known asymptotic decay of each column are also possible.} There is no loss of generality in working in $\ell^2(\mathbb{N})$ since we can always choose an orthonormal basis of a separable Hilbert space to obtain $\mathcal{H}\cong \ell^2(\mathbb{N})$. The majority of graph operators that are encountered in physics can be put in this framework. For example, given a finite range interaction Hamiltonian on $\ell^2(\mathbb{Z}^d)$, one can enumerate the vertices of the graph to obtain a realization of $\ell^2(\mathbb{Z}^d)\cong \ell^2(\mathbb{N})$ as well as an associated function $F$. The value of $[K_\epsilon*\mu_f](x_0)$ for some $f\in \ell^2(\mathbb{N})$ is then approximated through least-squares solutions of the rectangular systems~\cite{colbrook2019computing}
$$
P_{F(N)}(A-(x_0+\epsilon a_j))P_Nu_j^\epsilon=P_{F(N)}f,
$$
where $P_n$ denotes the orthogonal projection onto the span of the first $n$ basis vectors. For a rectangular truncation $H=P_{F(N)}AP_{N}$ supplied by the user, we can, for example, compute the smoothed measure with respect to the first canonical basis vector via the following commands.
\begin{verbatim}
  xi = linspace(-3.1,3.1,125);                  % Evaluation pts
  b = zeros(size(H,1),1); b(1) = 1;             % Measure wrt vector b
  mu = infmatMeas(H,b,xi,0.05,`Order',2);       % epsilon=0.05, m=2
\end{verbatim}
An example for a magnetic Schr\"odinger equation on a graphene lattice (see~\cref{sec:graphene}) is provided in the gallery of examples in \texttt{SpecSolve}.

\section{Examples}\label{sec:examples}
We now provide three examples to demonstrate the versatility of our computational framework. 

\subsection{Example 1: Beam and two-dimensional Schr\"odinger equations}\label{sec:diff_ops}
The increased computational efficiency achieved through high-order kernels allows us to treat PDEs and high-order ODEs. First, consider a fourth-order differential operator associated with the elastic beam equation, given by
\begin{equation}\label{eqn:elastic_beam}
[\mathcal{L}u](x)=\frac{d^4u}{dx^4}(x)-\frac{d}{dx}\left[\left(1-e^{-x^2}\right)\frac{du}{dx}\right](x)+\frac{a\sin(x)}{1+x^2}u(x), \qquad x\in\mathbb{R},
\end{equation}
for some constant $a\in\mathbb{R}$. \cref{fig:hard_examples} (left) shows $K_\epsilon*\mu_f$, for a second-order kernel with $\epsilon=0.05$ and $f(x)=\sqrt{2\pi^{-1}}/(1+x^2)$, when $a=0$, $5$, and $10$. When $a=0$, the operator is positive with continuous spectrum in $[0,\infty)$. When $a\neq 0$, there is also an eigenvalue below the continuous spectrum, corresponding to the spikes in~\cref{fig:hard_examples} (left). We also observe that different values of $a$ alter the profile of $\rho_f$ on $[0,\infty)$.

Next, consider the two-dimensional Schr\"odinger operator given by
\begin{equation}\label{eqn:2D_PDE}
[\mathcal{L}u](x_1,x_2)=-\nabla^2 u(x_1,x_2)+\left(\frac{e^{-x_1^2}}{1+x_2^2}+a\left(\mathrm{erf}(x_1)+\mathrm{erf}(x_2)\right)\right)u(x_1,x_2),\quad  x_j\in\mathbb{R},
\end{equation}
for some constant $a\in\mathbb{R}$, where $\mathrm{erf}(\cdot)$ is the error function. To apply the resolvent we map $\mathbb{R}^2$ to the torus $[-\pi,\pi]^2$ via $x_j\rightarrow10i(1-e^{i\theta_j})/(1+e^{i\theta_j})$. We then used a tensorized Fourier spectral method with hyperbolic cross ordering of the basis functions~\cite[Ch.~III]{lubich_qm_book}. \cref{fig:hard_examples} (right) shows $K_\epsilon*\mu_f$, for a fourth-order kernel with $\epsilon=0.2$  and $f(x_1,x_2)=\exp(-x_1^2-x_2^2)\sqrt{2\pi^{-1}}$, when $a=0,1$, and $2$. The spectrum of the operator is $[-2a,\infty)$ and we observe that the convolution $[K_\epsilon*\mu_f](x)$ takes small negative values in the vicinity of the lower boundary of the spectrum.

\begin{figure}[!tbp]
  \centering
  \begin{minipage}[b]{0.48\textwidth}
    \begin{overpic}[width=\textwidth]{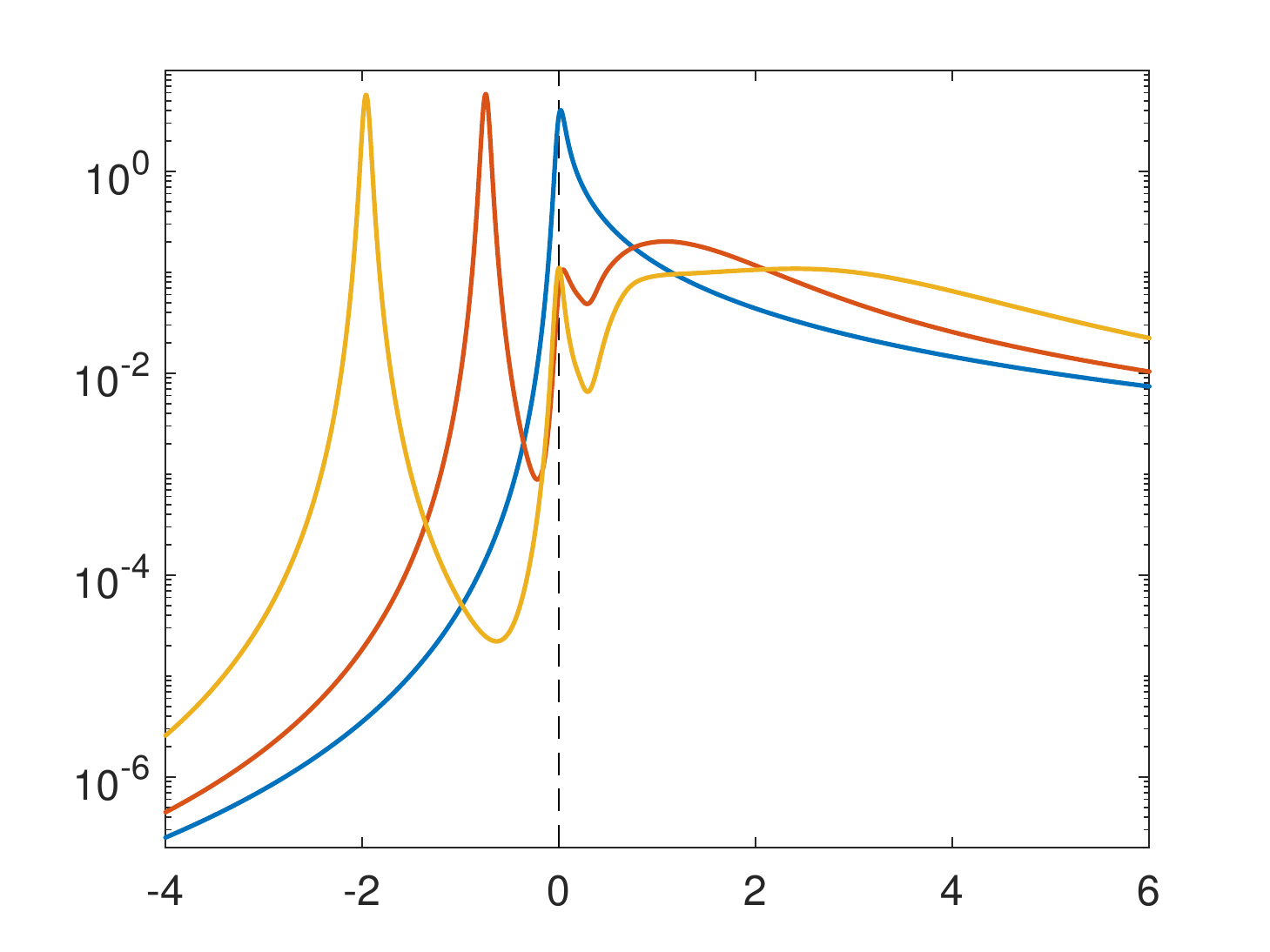}
 	\put (38,73) {$\displaystyle [K_\epsilon*\mu_f](x)$}
 	\put (50,-2) {$\displaystyle x$}
	\put (12.5,21) {\rotatebox{57} {$\displaystyle a=10$}}
	\put (18,17) {\rotatebox{47} {$\displaystyle a=5$}}
	\put (23,11) {\rotatebox{31} {$\displaystyle a=0$}}
 	\end{overpic}
  \end{minipage}
  \hfill
  \begin{minipage}[b]{0.48\textwidth}
    \begin{overpic}[width=\textwidth]{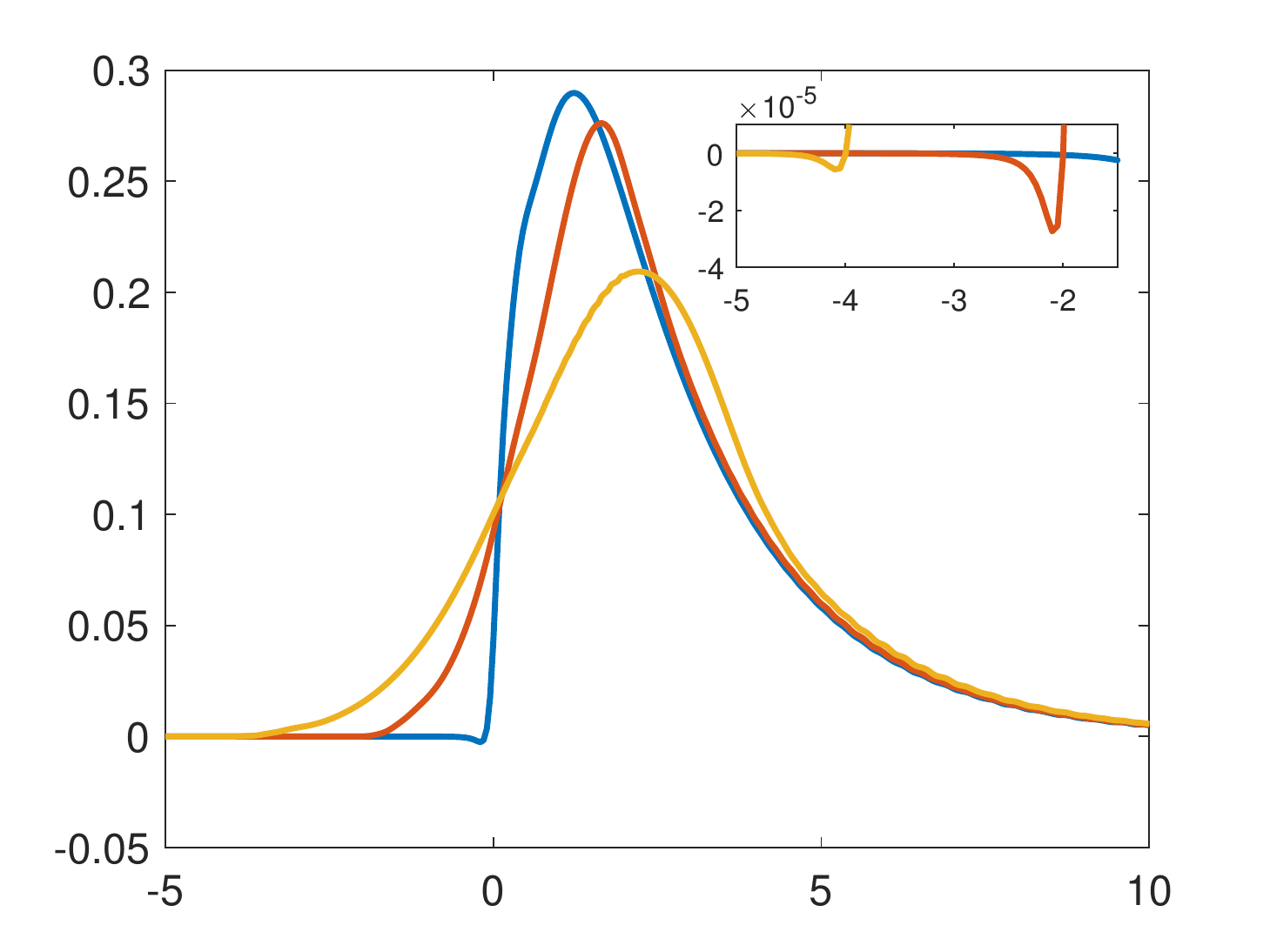}		
	\put(45,10)  {\vector(-1,1){7}}
	\put (45,10) { {$\displaystyle a=0$}}	
	\put(45,22)  {\vector(-1,0){9}}
	\put (45,22) { {$\displaystyle a=1$}}	
	\put(20,30)  {\vector(1,-1){9}}
	\put (15,30) { {$\displaystyle a=2$}}		
 	\put (38,73) {$\displaystyle [K_\epsilon*\mu_f](x)$}
 	\put (50,-2) {$\displaystyle x$}
 	\end{overpic}
  \end{minipage}
  \caption{Left: Smoothed approximations to the spectral measures of the elastic beam operators in~\cref{eqn:elastic_beam} for a second-order kernel with $a = 0, 5,$ and $10$. Right: Smoothed approximations to the spectral measures of the two-dimensional Schr\"odinger operators in~\cref{eqn:2D_PDE} for a fourth-order kernel with $a=0,1$, and $2$. The magnified region demonstrates that $[K_\epsilon*\mu_f](x)$ is not always positive for kernels of order greater than two.}
	\label{fig:hard_examples}
\end{figure}
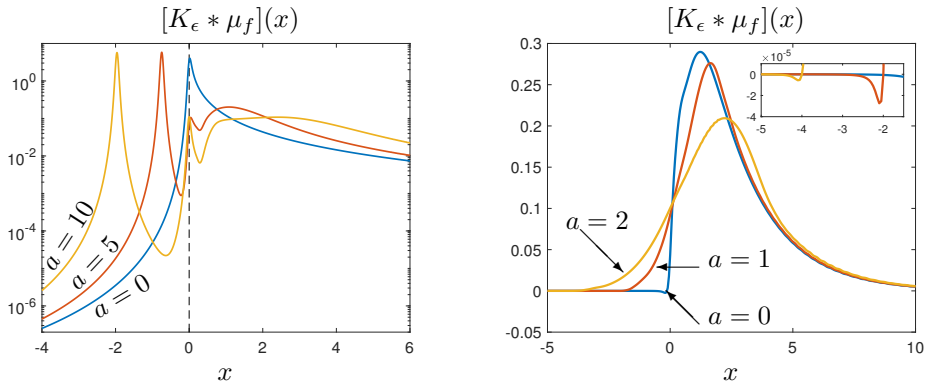

\subsection{Example 2: The Schr\"odinger equation on a graphene lattice}\label{sec:graphene}
We now apply our method to a magnetic tight-binding model of graphene, which involves a discrete graph operator~\cite{agazzi2014colored}. Graphene is a two-dimensional material with carbon atoms situated at the vertices of a honeycomb lattice (see~\cref{honeycomb}), whose unusual properties are studied in condensed-matter physics~\cite{neto2009electronic,16FW}. The magnetic properties of graphene are important because of the experimental observation of the quantum Hall effect and Hofstadter's butterfly~\cite{ponomarenko2013cloning}, and the exciting new area of twistronics \cite{chang_2019,lu2019superconductors}.

\subsubsection{The model}
A honeycomb lattice can be decomposed into two bipartite sub-lattices (see~\cref{honeycomb} (left)) and thus, the wave function of an electron can be modeled as the spinor~\cite{agazzi2014colored}
\[
\psi_{m,n}=(\psi_{m,n}^{[1]},\psi_{m,n}^{[2]})^T\in\mathbb{C}^2,\qquad \psi=(\psi_{m,n})\in l^2(\mathbb{Z}^2;\mathbb{C}^2)\cong \ell^2(\mathbb{N}).
\]
Here, $(m,n)\in\mathbb{Z}^2$ labels a position on the sub-lattices and $\ell^2(\mathbb{Z}^2;\mathbb{C}^2)$ denotes the space of square summable $\mathbb{C}^2$-valued sequences indexed by $\mathbb{Z}^2$. To define the Hamiltonian, consider the following three magnetic hopping operators $T_1,T_2,T_3:\ell^2(\mathbb{Z}^2;\mathbb{C}^2)\rightarrow \ell^2(\mathbb{Z}^2;\mathbb{C}^2)$ for a given magnetic flux per unit cell $\Phi$ (in dimensionless units):
\[
(T_1\psi)_{m,n}\!=\!\begin{pmatrix}
\psi_{m,n}^{[2]}\\
\psi_{m,n}^{[1]}
\end{pmatrix}\!,\quad (T_2\psi)_{m,n}\!=\!\begin{pmatrix}
\psi_{m+1,n}^{[2]}\\
\psi_{m-1,n}^{[1]}
\end{pmatrix}, \quad (T_3\psi)_{m,n}\!=\!\begin{pmatrix}
e^{-2\pi i\Phi m}\psi_{m,n+1}^{[2]}\\
e^{2\pi i\Phi m}\psi_{m,n-1}^{[1]}
\end{pmatrix}.
\]
After a suitable gauge transformation, the free Hamiltonian can be expressed as $H_0=T_1+T_2+T_3$ and $\Lambda(H_0)\subset[-3,3]$. A suitable ordering of lattice points leads to a sparse discretization of $H_0$, where the $k$th column contains $\mathcal{O}(\sqrt{k})$ non-zero entries (see~\cref{honeycomb} (right)). Therefore, for an approximation using $N$ basis sites, the action of the resolvent can be computed in $\mathcal{O}(N^{3/2})$ operations~\cite{trefethen1997numerical}.

\begin{figure}
\centering
\includegraphics[width=0.45\textwidth,clip,trim= 0mm -30mm 0mm 0mm]{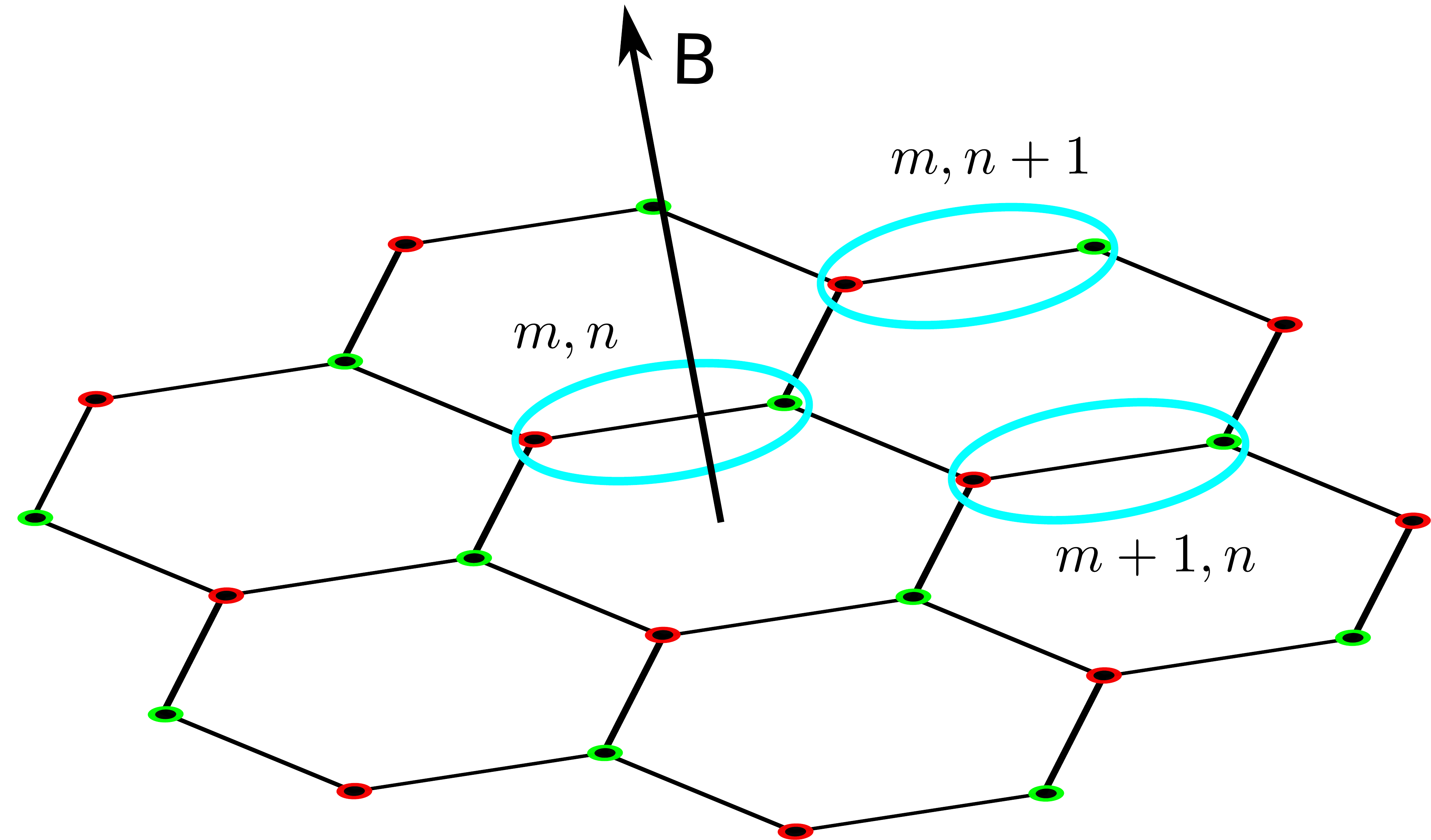}
\includegraphics[width=0.45\textwidth,clip,trim={32mm 97.5mm 32mm 92mm}]{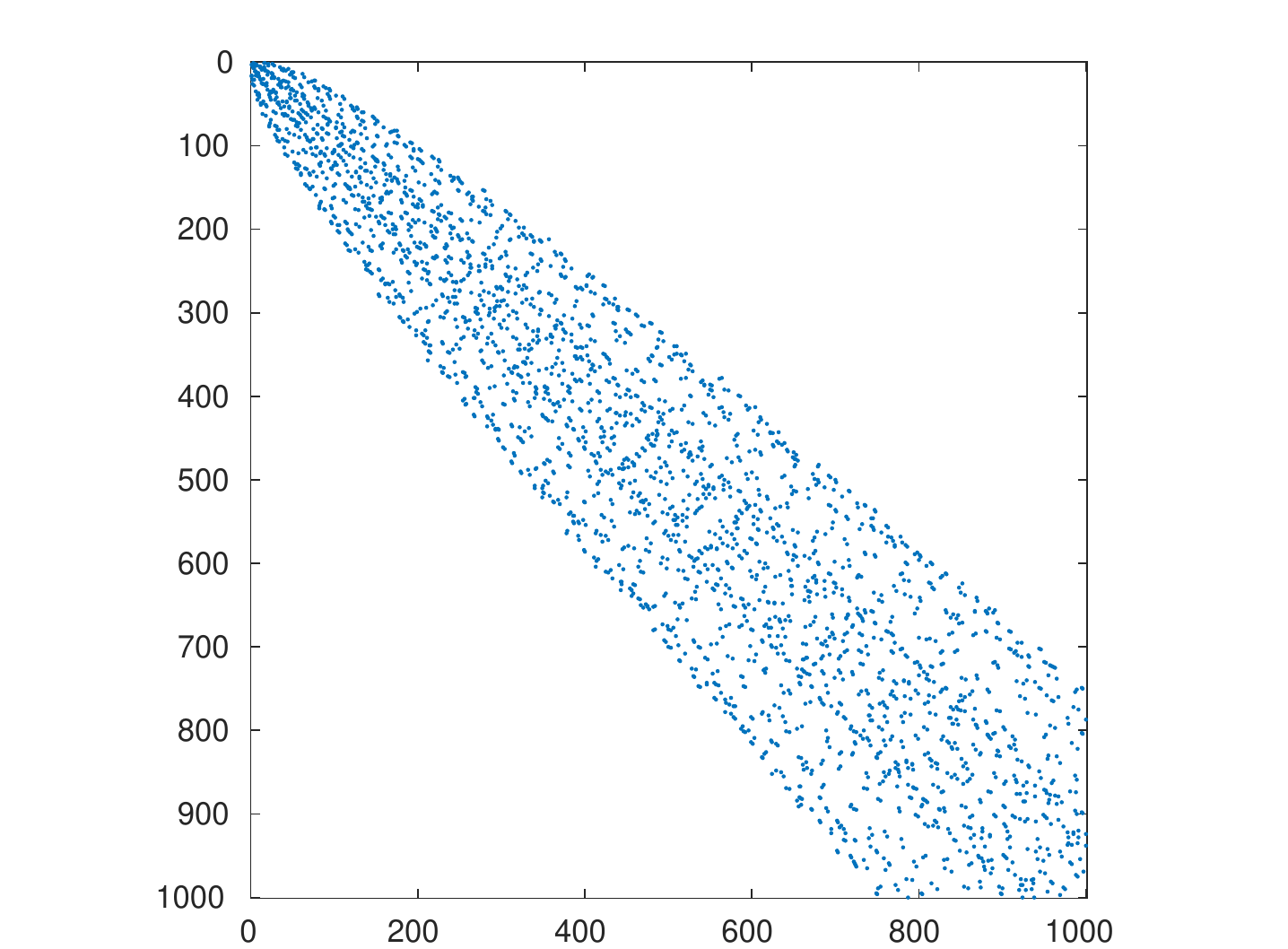}
\caption{Left: Honeycomb structure of graphene as a bipartite graph. We have shown the spinor structure via the circled lattice vertices and the indexing via $(m,n)$. The arrow shows the perpendicular magnetic field $\textbf{{\rm B}}$. Right: Sparsity structure of the first $10^3\times 10^3$ block of the infinite matrix, and the corresponding growing local bandwidth.}
\label{honeycomb}
\end{figure}

\subsubsection{The computed measures}

\cref{butterfly} shows how the spectral measure of $H_0$, taken with respect to the vector $e_1$, varies with $\Phi$. For $\Phi\in\mathbb{Q}$, the spectrum is absolutely continuous, and we show the Radon--Nikodym derivative of the measure, $\rho_{e_1}$. The calculations, performed with a fourth-order kernel and $\epsilon=0.01$, show a sharp Hofstadter-type butterfly.\footnote{Hofstadter's butterfly \cite{hofstadter1976energy} is the visual representation of the fractal, self-similar nature of the spectrum of a Hamiltonian describing non-interacting two-dimensional electrons in a magnetic field in a lattice. The most famous example is for the almost Mathieu operator on $\ell^2(\mathbb{Z})$.}

\begin{figure}
\centering
\begin{overpic}[width=.9\textwidth,clip,trim={41mm 55mm 41mm 55mm}]{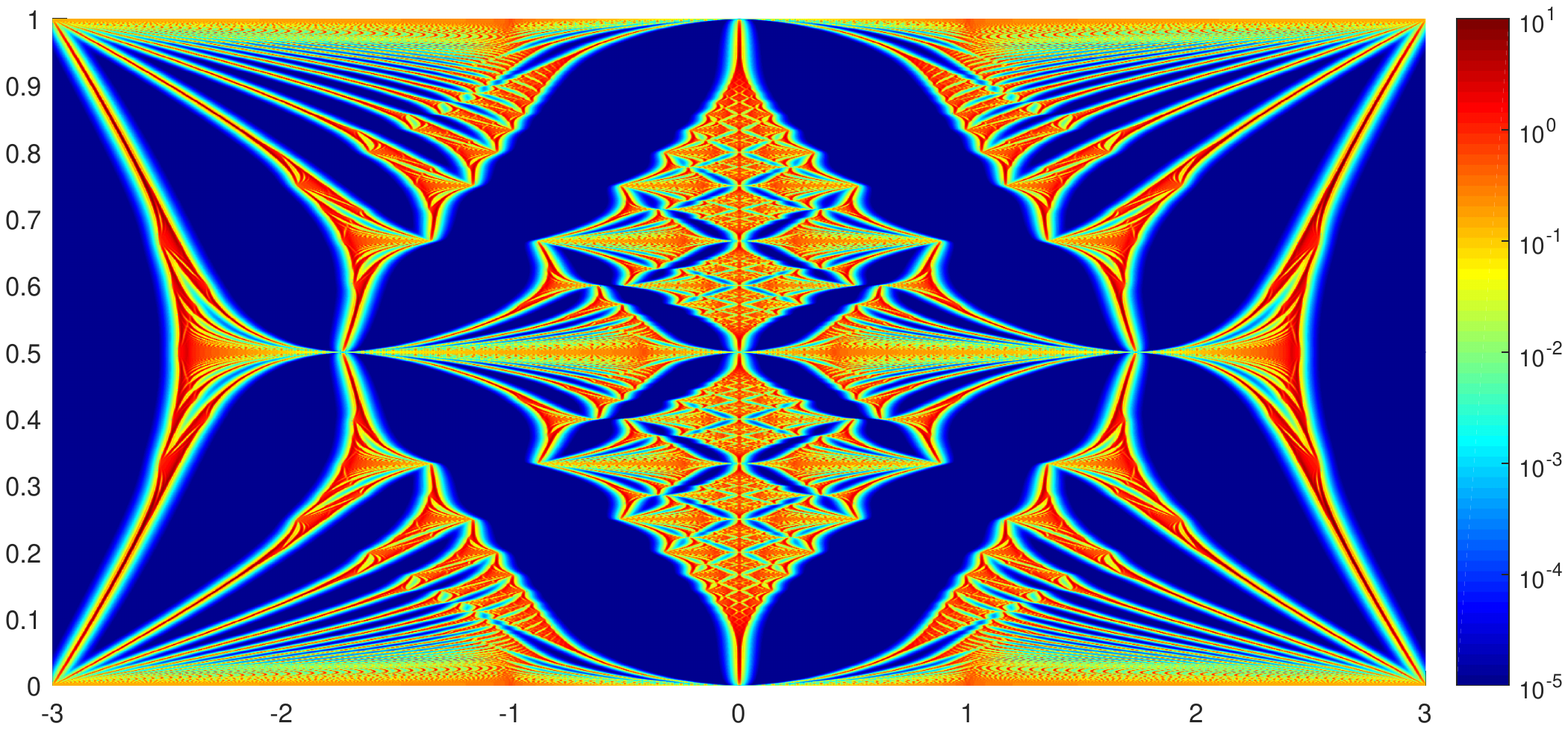}
\put (-2,23.4) {$\Phi$}
\put (46.8,-2) {$x$}
\end{overpic} 
\caption{The Radon--Nikodym derivative of the measure for various $\Phi$, computed with $\epsilon=0.01$. The spectrum is fractal for irrational $\Phi$, which is approximated by rational $\Phi$. The small gaps in the spectrum are clearly visible (corresponding to the blue shaded regions) and the logarithmic scale shows the sharpness of the approximation to $\rho_{e_1}$, which vanishes in these gaps.}
\label{butterfly}
\end{figure}

\cref{graphene_spec_poll} (left) shows an approximation of $\rho_{e_1}$ when $\Phi=1/4$ using a fourth-order kernel and $\epsilon=0.01$. We also show, as shaded vertical strips, the output of the algorithm in~\cite{PhysRevLettcolb} which computes the spectrum with error control (we used an error bound of $10^{-3}$) and without spectral pollution.\footnote{With a non-periodic potential~\cref{non_per_pot}, this is a highly non-trivial problem since finite truncation methods typically suffer from spectral pollution inside the convex hull of the essential spectrum.} The support of $K_\epsilon*\mu_f$ is the whole real line due to the non-compact support of the kernel $K$. However, if $x\not\in\Lambda({H_0})$, then applying~\cref{decay_bound} directly to the definition of convolution shows that $|[K_\epsilon*\mu_f](x)|\leq C_K\epsilon^m/(\epsilon+\mathrm{dist}(x,\Lambda({H_0})))^{m+1}$, where $C_K$ is the constant in~\cref{decay_bound}, so $|[K_\epsilon*\mu_f](x)|$ decays rapidly off of the spectrum. We also consider a multiplication operator (potential) perturbation, modeling a defect, of the form
\begin{equation}
\label{non_per_pot}
V(\textbf{x})=\frac{\cos(\|\textbf{x}\|_2\pi)}{(\|\textbf{x}\|_2+1)^2},
\end{equation}
where $\textbf{x}$ denotes the position of a vertex normalized so each edge has length $1$. The perturbed operator is then $H_0+V$. Since the perturbation is trace class, the absolutely continuous part of the spectrum remains the same (though the measure changes) and the potential induces additional eigenvalues (see~\cref{graphene_spec_poll} (right)). Again, we see that $|[K_\epsilon*\mu_f](x)|$ decays rapidly off of the spectrum. In particular, the measure is not corrupted by spikes in the gaps in the essential spectrum or similar artifacts caused by spectral pollution.

\begin{figure}
\centering
\begin{minipage}[b]{0.48\textwidth}
    \begin{overpic}[width=\textwidth]{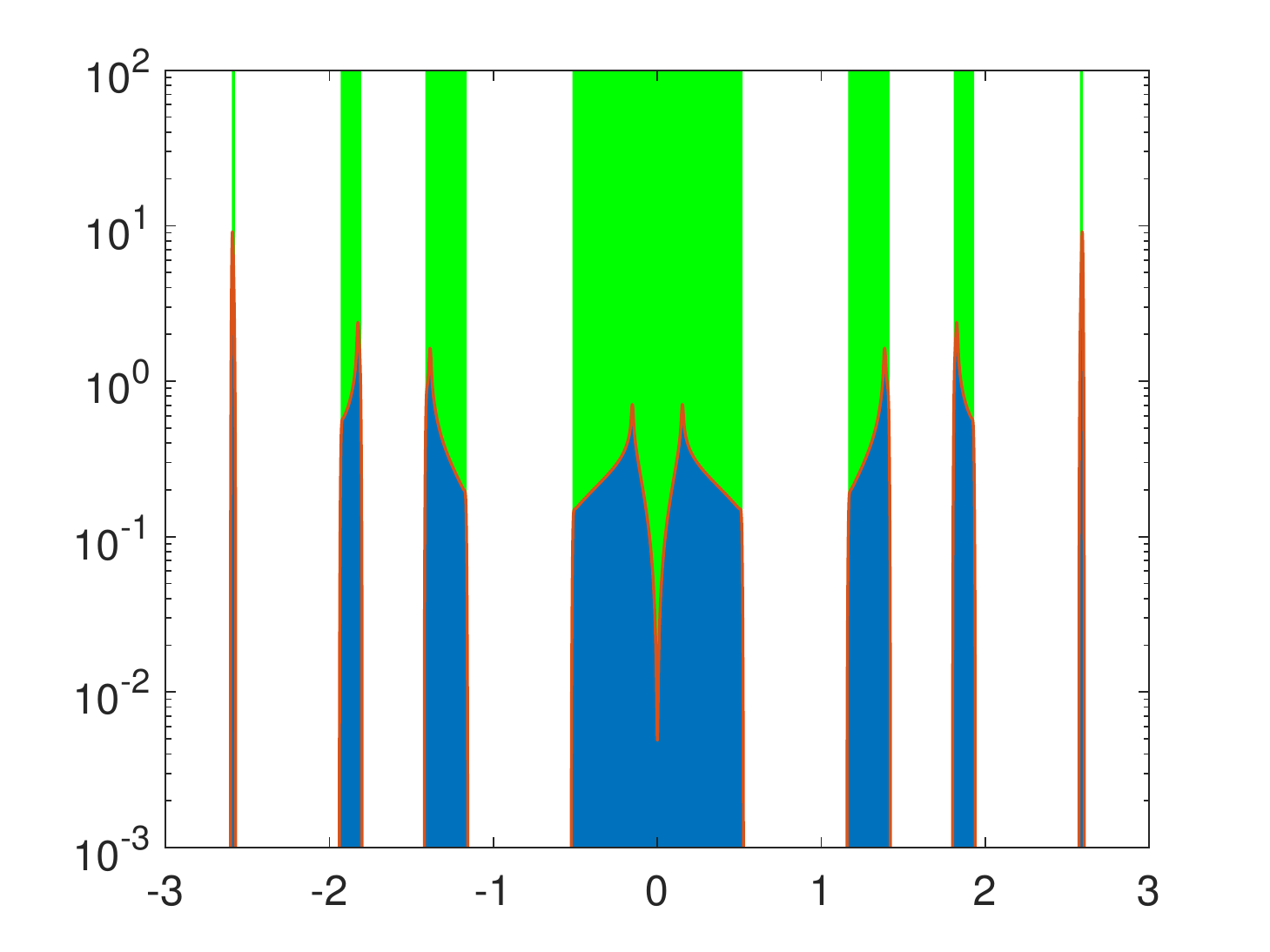}
 	\put (50,-2) {$\displaystyle x$}
 	\end{overpic}
  \end{minipage}
  \hfill
	\begin{minipage}[b]{0.48\textwidth}
    \begin{overpic}[width=\textwidth]{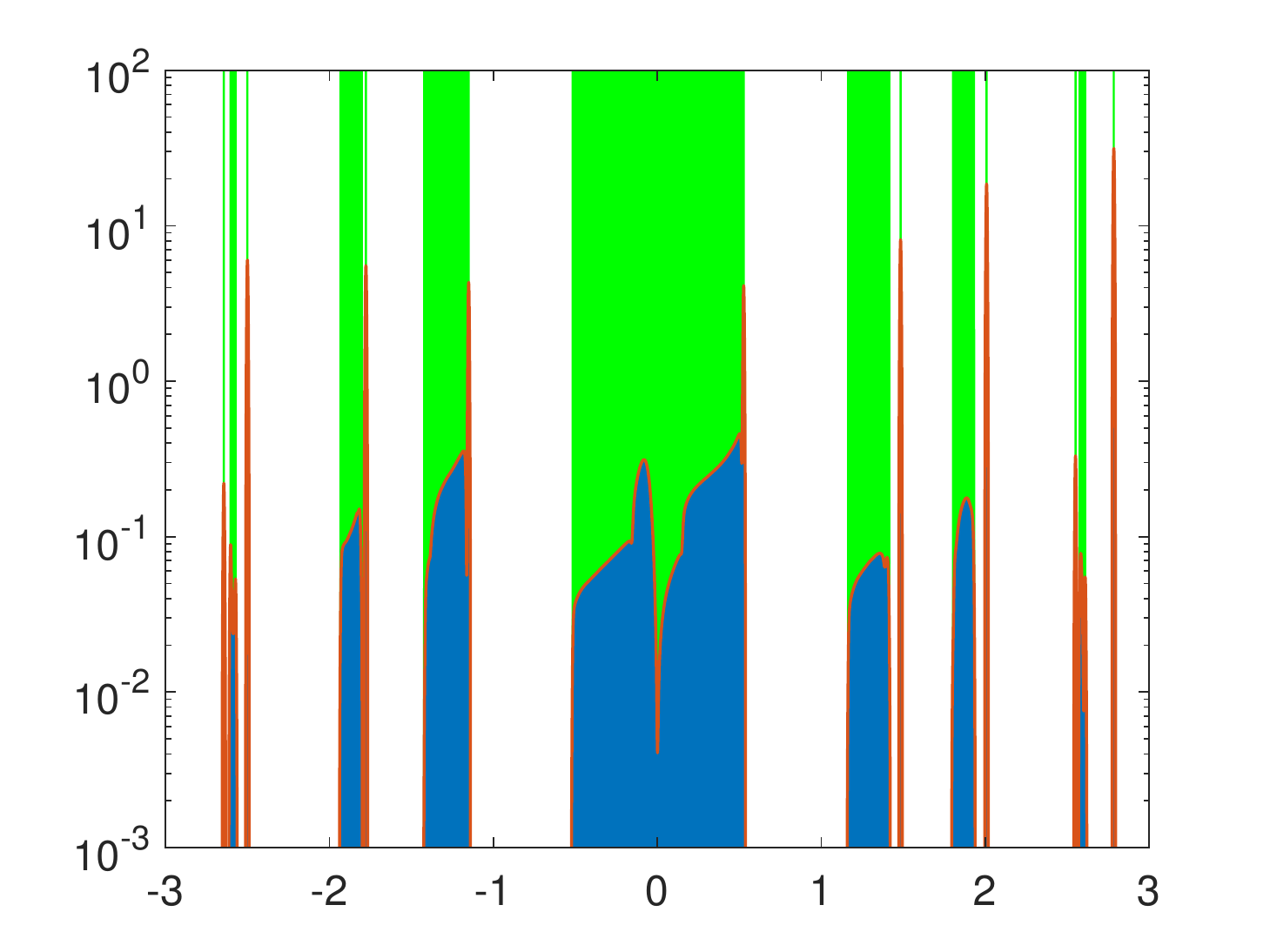}
 	\put (50,-2) {$\displaystyle x$}
 	\end{overpic}
  \end{minipage}
\caption{Left: Smoothed measure with no potential. We show the algorithm in~\cite{PhysRevLettcolb} as shaded strips (green) for comparison. Right: The same computation but with the added potential in~\cref{non_per_pot}. The additional eigenvalues correspond to spikes in the smoothed measure.}
\label{graphene_spec_poll}
\end{figure}

\subsection{Example 3: Discrete spectra and Dirac operators}\label{sec:dirac}
In this example, we consider the case of the Dirac operator $\mathcal{L}=\mathcal{D}_V$ defined below. Often this operator has discrete eigenvalues in the interval $(-1,1)$, which forms a gap in the essential spectrum. This means that standard Galerkin methods used to compute $\Lambda(\mathcal{D}_V)$ typically suffer from spectral pollution in the gap. That is, as the discretization size increases, the discrete spectrum of the Galerkin discretization clusters in a way that does not approximate $\Lambda(\mathcal{D}_V)$. There is a vast literature on methods that seek to avoid spectral pollution when computing $\Lambda(\mathcal{D}_V)$ \cite{drake1981application,kutzelnigg1984basis,talman1986minimax,kutzelnigg1997relativistic,shabaev2004dual}. The majority of existing approaches work for certain classes of potentials and avoid spectral pollution on particular subsets of $(-1,1)$. Even for Coulomb-type potentials, spectral pollution can be a difficult issue to overcome, and computations typically achieve a few digits of precision for the ground state and a handful of the first few excited states. A popular approach is the so-called kinetic balance condition, which does not always work for Coulomb potentials \cite{stanton1984kinetic,dyall1990kinetic,dirac_spec_poll}. Our approach does not suffer from spectral pollution and can compute the first thousand eigenvalues to near machine precision accuracy.

\subsubsection{Recovering eigenvalues and projections onto eigenspaces}
The dominated convergence theorem applied to~\cref{eqn:pk_identity} shows that, for any $x\in\mathbb{R}$, we have
\begin{equation}\label{pt_spectra_lim}
\lim_{\epsilon\downarrow0}\epsilon\cdot\mathrm{Im}\left(\langle\mathcal{R}_\mathcal{L}(x+i\epsilon)f,f\rangle\right)=\sum_{\lambda\in\Lambda^{\mathrm{p}}(\mathcal{L})\cap\{x\}}\langle\mathcal{P}_\lambda f,f\rangle.
\end{equation}
Moreover, if there is no singular continuous spectra in a neighborhood of $x$, and $x$ is not an accumulation point of $\Lambda^{\mathrm{p}}(\mathcal{L})$, then~\cref{pt_spectra_lim} can be sharpened to
\begin{equation}\label{pt_spectra_lim2}
\epsilon\cdot\mathrm{Im}\left(\langle\mathcal{R}_{\mathcal{L}}(x+i\epsilon)f,f\rangle\right)=\sum_{\lambda\in\Lambda^{\mathrm{p}}(\mathcal{L})\cap\{x\}}\langle\mathcal{P}_\lambda f,f\rangle+\mathcal{O}(\epsilon).
\end{equation}
These formulas allow us to compute the locations of eigenvalues of the operator, and the corresponding projection coefficients onto the eigenspaces for vectors $f$. 

\subsubsection{The Dirac operator}
We consider a differential operator $\mathcal{D}_V$ associated with a coupled first-order system of differential equations that describes the motion of a relativistic spin-$1/2$ particle in a radially symmetric potential $V(r)$, defined by
\[
\mathcal{D}_V=
\begin{pmatrix} 1 + V(r) & -\frac{d}{dr} +\frac{\kappa}{r}\\ \frac{d}{dr} + \frac{\kappa}{r} & -1+V(r) 
\end{pmatrix}.
\]
Here, $\kappa=j+1/2$ for $j\in\mathbb{Z}$ (related to the angular momentum of the particle) and $\mathcal{D}_V$ is a special case of the Dirac operator with a radially symmetric potential~\cite{thaller_dirac}.

If $V$ satisfies suitable conditions~\cite{thaller_dirac}, then $\mathcal{D}_V$ is a self-adjoint operator with essential spectrum supported on $(-\infty,-1]\cup [1,\infty)$. Depending on $V(r)$, the spectrum may also contain discrete eigenvalues in $(-1,1)$. Generally, in computational chemistry, positive eigenvalues correspond to bound states of a relativistic quantum electron in the external field $V$, and negative eigenvalues correspond to bound states of a positron~\cite{thaller_dirac}.

\subsubsection{Computing eigenvalues while avoiding spectral pollution}
Assuming that $f$ in~\cref{pt_spectra_lim} is not orthogonal to any of the eigenfunctions, it follows from~\cref{pt_spectra_lim} and~\cref{pt_spectra_lim2} that the positions of the peaks of the function
$$
\nu_f^\epsilon(x):=\epsilon\cdot\mathrm{Im}\left(\langle\mathcal{R}_{\mathcal{D}_V}(x+i\epsilon)f,f\rangle\right)
$$
correspond to the eigenvalues. To test this, we consider the case of $\kappa=-1$ and the Coulomb-type potentials $V(r)=\gamma/r$ for $-\sqrt{3}/2<\gamma<0.$ For these potentials, the eigenvalues are known analytically as~\cite[Ch.~7]{thaller_dirac}
$$
E_j(\mathcal{D}_V)=\left(1+{\gamma^2}{(j+\sqrt{1-\gamma^2})^{-2}}\right)^{-1/2},\qquad j\geq 0.
$$
Note that the eigenvalues accumulate at $1$. This makes computing $E_j(\mathcal{D}_V)$ difficult when $j$ is large, even in the absence of spectral pollution.

\Cref{Dirac} (left) shows $\nu_f^\epsilon$ with $\epsilon=10^{-10}$, $f(r)=(\sqrt{2}re^{-r},\sqrt{2}re^{-r})$, and $\gamma=-0.8$. One can robustly compute $\nu_f^\epsilon$ for a fixed $\epsilon>0$ by using the techniques in~\cref{specfun_ODE_meth} and adaptively selecting the discretization size. For $\epsilon=10^{-10}$, we can accurately compute $E_1(\mathcal{D}_V),\ldots,E_{1000}(\mathcal{D}_V)$ by the location of the local maxima of $\nu_f^\epsilon$.  Moreover, the size of the peaks correspond to $\|\mathcal{P}_{E_j} f\|^2$. \Cref{Dirac} (left) shows that $\|\mathcal{P}_{E_j} f\|^2$ decreases at an algebraic rate as $j\rightarrow \infty$.  If one is not satisfied with the accuracy of the computed eigenvalues, then one can decrease $\epsilon$ at the expense of an increased computational cost. In~\cref{Dirac} (right), we show the absolute error in the computed eigenvalues $E_j(\mathcal{D}_V)$ for $j = 0, 5, 10, 100, 500,$ and $1000$ as $\epsilon\downarrow 0$. We find that our algorithm can resolve hundreds of eigenvalues, even when highly clustered, to an accuracy of essentially machine precision.

\begin{figure}
\centering
\centering
\begin{minipage}[b]{0.48\textwidth}
    \begin{overpic}[width=\textwidth]{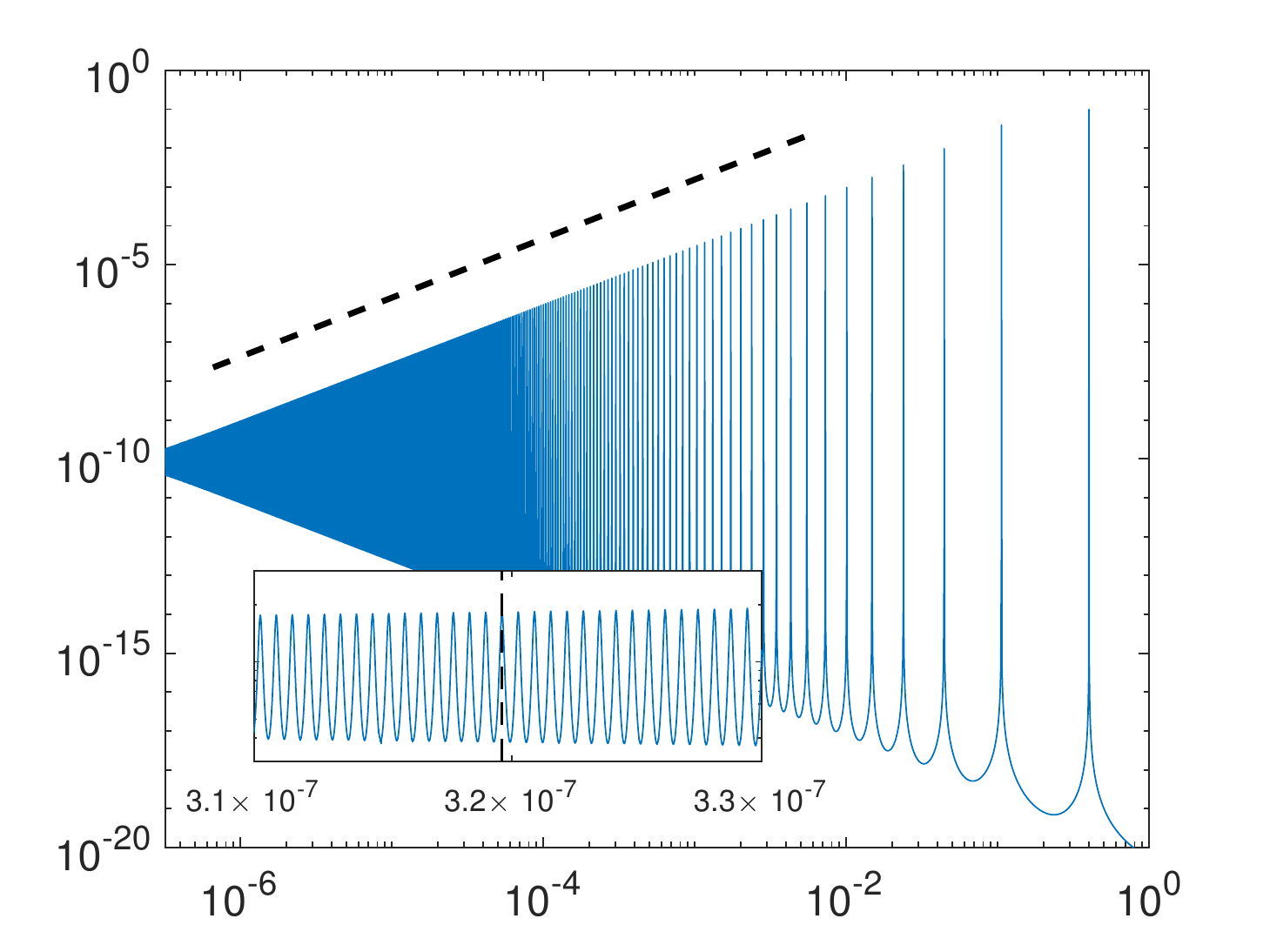}
 	\put (45,-2) {$\displaystyle 1-x$}
	\put (45,73) {$\displaystyle \nu_f^\epsilon(x)$}
 	\end{overpic}
  \end{minipage}
  \hfill
	\begin{minipage}[b]{0.48\textwidth}
    \begin{overpic}[width=\textwidth]{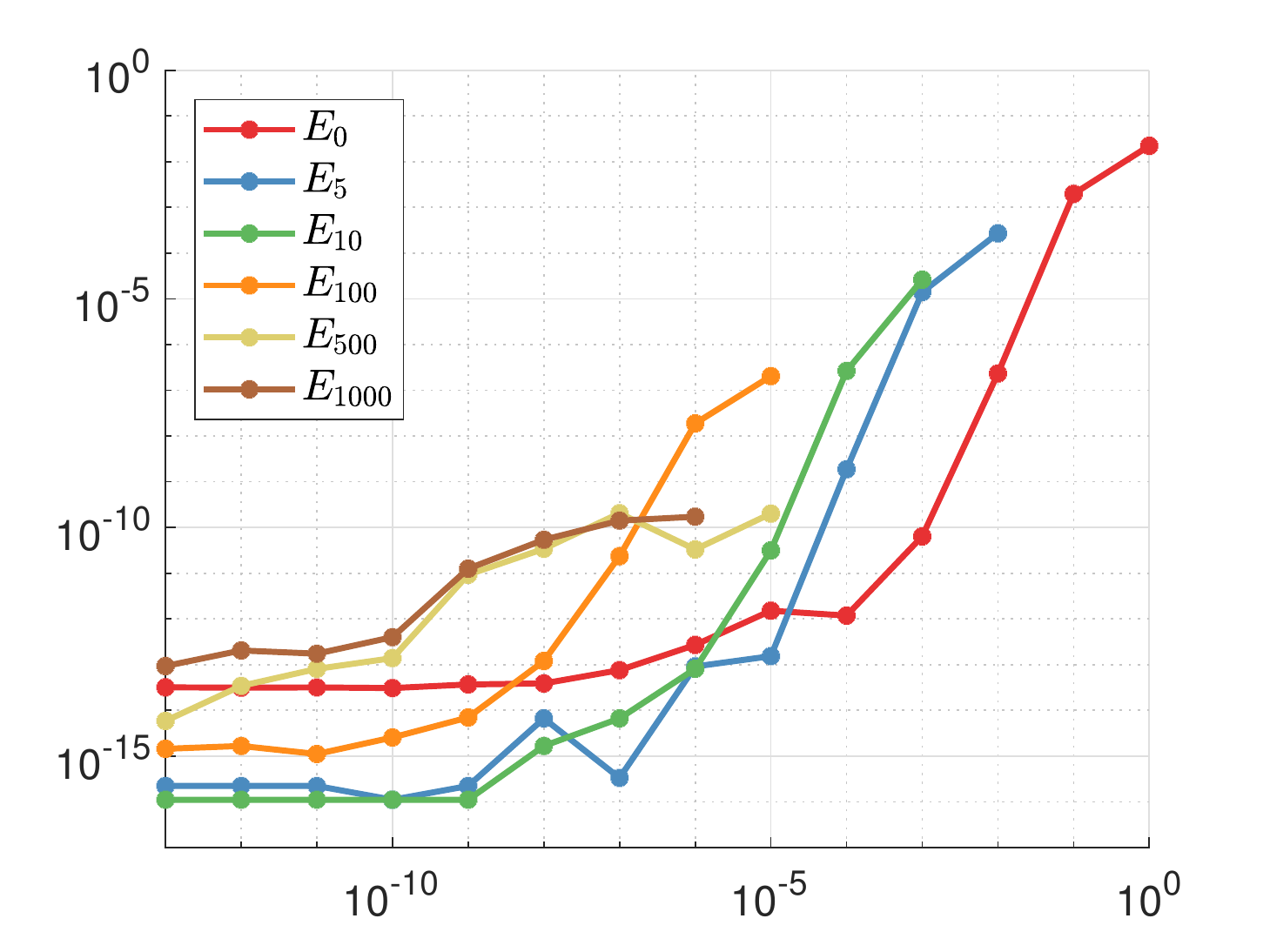}
 	\put (50,-2) {$\displaystyle \epsilon$}
	\put (34,73) {Absolute Error}
 	\end{overpic}
  \end{minipage}
\caption{Left: The function $\nu_f^\epsilon(x)$ for $x$ near $1$. The sloped dashed line shows the algebraic decay of $\smash{\|\mathcal{P}_{E_j}f\|^2}$ (approximately $\mathcal{O}(j^{-3})$). The magnified region shows the extreme clustering, where the vertical dashed line corresponding to $E_{1000}$.  Right: The absolute error in the computed eigenvalues $E_j(\mathcal{D}_V)$ for $j = 0, 5, 10, 100, 500, 1000$ as $\epsilon\downarrow 0$.}
\label{Dirac}
\end{figure}

\section{Conclusions and additional potential applications}\label{sec:conclusions}
In this paper, we have developed a general framework for evaluating smoothed approximations to the spectral measures of self-adjoint operators. We have highlighted the theoretical and practical aspects of the algorithm in the context of differential, integral, and lattice operators. The resolvent-based framework robustly captures discrete and continuous spectral properties of the operator, rather than any underlying discretizations, yielding a flexible and efficient method for a variety of spectral problems. 

A general computational framework for computing the spectral measure $\mu_f$ opens the door to a new set of algorithms for computing with operators and studying their spectral properties. As spectral characterizations of continuous and discrete models draw renewed interest in the context of data-centered applications, our algorithms may be useful in understanding the behavior of large real-world networks and new random graph models. The development of rational kernels and corresponding local evaluation schemes may also be useful for local explorations of the spectral density of operators of large finite dimension, such as in DOS calculations in physics~\cite{lin2016approximating} or real-world networks~\cite{dong2019network}.

Our framework can be used to compute the vector-valued functional calculus via
$$
F(\mathcal{L})f\approx \frac{-1}{2\pi i}\int_{-\infty}^\infty F(y)\sum_{j=1}^{m}\left[\alpha_j \mathcal{R}_{\mathcal{L}}(y-\epsilon a_j)f-\beta_j\mathcal{R}_{\mathcal{L}}(y-\epsilon b_j)f\right]dy,
$$
which is useful in the solution of time-evolution problems. For example, taking $F(y)=\exp(-ity)$ gives an approximation of the solution to the linear Schr\"odinger equation with initial state $f$ at time $t$. The vector-valued functional calculus may also be used to solve more complicated evolution systems, such as non-autonomous Cauchy problems and non-linear problems, through splitting methods~\cite{mclachlan2002splitting,lubich_qm_book}. Therefore, our approach may aid the development of discretization-oblivious exponential integrators for PDEs or sampling from stochastic processes with self-adjoint generators~\cite[Ch.~17]{kallenberg2006foundations}.

\appendix

\section{Convergence rates and error bounds}\label{sec:appendix}
In this Appendix, we prove the pointwise and $L^p$ convergence bounds of $K_\epsilon*\mu_f$ to $\rho_f$ as $\epsilon\downarrow 0$.

\subsection{Pointwise error bounds}\label{subsec:pointwise_bounds}
The pointwise convergence shows that samples of $K_\epsilon*\mu_f$ are meaningful because they converge to $\rho_f$, at a rate determined by the local regularity of $\rho_f$ and the order of the kernel. Recall that in~\cref{thm:kernel_rates}, $K$ is an $m$th order kernel, the measure $\mu_f$ is absolutely continuous on $I=(x_0-\eta,x_0+\eta)$ for $\eta>0$ and a fixed $x_0\in\mathbb{R}$, and that $\rho_f\in\mathcal{C}^{n,\alpha}(I)$ with $\alpha\in[0,1)$.
 
\begin{proof}[Proof of~\cref{thm:kernel_rates}]
First, we decompose $\rho_f$ into two non-negative parts $\rho_f=\rho_1+\rho_2$, where $\rho_1$ is compactly supported on $I$ and $\rho_2$ vanishes on $(x_0-\eta/2,x_0+\eta/2)$. Using the convolution representation for $K_\epsilon*\mu_f$, we have
\begin{equation}\label{eqn:approx_errorAPPENDIX}
\left|\rho_f(x_0)-[K_\epsilon*\mu_f](x_0)\right|\!\leq\! \left|\int_\mathbb{R}\!\!K_{\epsilon}(y)\left(\rho_1(x_0-y)-\!\rho_1(x_0)\right)dy\right| \!+\!\left|\!\left[K_\epsilon*\mu_f^{(\mathrm{r})}\right]\!\!(x_0)\right|\!.
\end{equation}
Here, the measure $d\mu_f^{(\mathrm{r})}(y)=d\mu_f(y)-\rho_1(y)dy$ is non-negative and supported in the complement of $(x_0-\eta/2,x_0+\eta/2)$. Since $\mu_f$ is a probability measure, we have that $\smash{\int_\mathbb{R}d\mu_f^{(\mathrm{r})}(y)\leq 1}$, and the second term on the right-hand side of~\cref{eqn:approx_errorAPPENDIX} is bounded by
\begin{equation}\label{eqn:approx_error_term2APPENDIX}
\left|\left[K_\epsilon*\mu_f^{(\mathrm{r})}\right](x_0)\right| = \left|\int_\mathbb{R}K_\epsilon(x_0-y)d\mu_f^{(\mathrm{r})}(y)\right|\leq \sup_{|y|\geq\eta/2}|K_{\epsilon}(y)|\leq\frac{C_K\epsilon^m}{(\epsilon+\frac{\eta}{2})^{m+1}}.
\end{equation}
where the constant $C_K$ is given in~\cref{def:mth_order_kernel}. 

To bound the first term in~\cref{eqn:approx_errorAPPENDIX}, we expand $\rho_1(x_0-y)$ using Taylor's theorem:
\begin{equation}\label{eqn:taylor}
\rho_1(x_0-y)=\sum_{j=0}^{k-1}\frac{(-1)^j\rho_1^{(j)}(x_0)}{j!}y^j + (-1)^k\frac{\rho_1^{(k)}(\xi_y)}{k!}y^k, \qquad k=\min(n,m),
\end{equation}
where $|\xi_y-x_0|\leq |y|$. We consider two cases separately.

{\bf Case (i): $\mathbf{n+\alpha<m}$.} In this case $k=n$ and we can select $\rho_1$ so that
$$
\frac{1}{n!}\left|\rho_1^{(n)}\right|_{\mathcal{C}^{0,\alpha}(I)}\leq C(n,\alpha)\|\rho_{f}\|_{\mathcal{C}^{n,\alpha}(I)}\left(1+\eta^{-n-\alpha}\right),
$$
for some universal constant $C(n,\alpha)$ that only depends on $n$ and $\alpha$. Existence of such a decomposition follows from standard arguments with cut-off functions. Plugging~\cref{eqn:taylor} into~\cref{eqn:approx_errorAPPENDIX} and applying the vanishing moment condition (\cref{def:mth_order_kernel} (ii)), we obtain
\begin{equation}\label{eqn:approx_error_term1APPENDIX}
\int_\mathbb{R}K_\epsilon(y)\left(\rho_1(x_0-y)-\rho_1(x_0)\right)dy=(-1)^n\int_\mathbb{R}K_\epsilon(y)\frac{\rho_1^{(n)}(\xi_y)}{n!}y^n\,dy.
\end{equation}
Since $n<m$, we can use the vanishing moment condition again to obtain
\begin{equation}\label{eqn:approx_error_term1Holder form}
\int_\mathbb{R}K_\epsilon(y)\frac{\rho_1^{(n)}(\xi_y)}{n!}y^n\,dy=\int_\mathbb{R}K_\epsilon(y)\frac{\rho_1^{(n)}(\xi_y)-\rho_1^{(n)}(x_0)}{n!}y^n\,dy.
\end{equation}
Since $\rho_1^{(n)}\in\mathcal{C}^{0,\alpha}(I)$ and $|\xi_y-x_0|\leq |y|$, we have $|\rho_1^{(n)}(\xi_y)-\rho_1^{(n)}(x_0)|\leq|\rho_1^{(n)}|_{\mathcal{C}^{0,\alpha}(I)}|y|^\alpha$. Applying this bound to the integrand in~\cref{eqn:approx_error_term1Holder form} and changing variables $y\rightarrow\epsilon y$,
\begin{equation}\label{eqn:approx_error_term1Final}
\left|\int_\mathbb{R}K_\epsilon(y)\frac{\rho_1^{(n)}(\xi_y)-\rho_1^{(n)}(x_0)}{n!}y^n\,dy\right|\leq\frac{\epsilon^{n+\alpha}}{n!}\left|\rho_1^{(n)}\right|_{\mathcal{C}^{0,\alpha}(I)}\int_\mathbb{R}|K(y)||y|^{n+\alpha}\,dy.
\end{equation}
Recalling our selection of $\rho_1$ and combining~\cref{eqn:approx_error_term1Final} with~\cref{eqn:approx_error_term2APPENDIX} proves case (i).

{\bf Case (ii): $\mathbf{n+\alpha\geq m}$.} 
In this case $k=m$ and we can select $\rho_1$ such that
$$
2e\left\|\rho_1^{(m)}\right\|_{\infty}\leq C(m)\|\rho_{f}\|_{\mathcal{C}^{m}(I)}\left(1+\eta^{-m}\right),
$$
for some universal constant $C(m)$ that only depends on $m$. Again, existence of such a decomposition follows from standard arguments with cut-off functions. Since $\rho_1$ has compact support in $I$, we have that $\rho_1(x_0-y)=0$ if $\left|y\right|\geq\eta$. We split the range of integration in~\cref{eqn:approx_error_term1APPENDIX}, substitute the Taylor expansion in~\cref{eqn:taylor}, and change variables $y\rightarrow\epsilon y$, to obtain 
\begin{equation}\label{eqn:split_integration}
\begin{split}
&\left|\int_\mathbb{R}\!K_{\epsilon}(y)\left(\rho_1(x_0-y)-\rho_1(x_0)\right)dy\right|\leq \left|\rho_1(x_0)\right|\left|\int_{\left|y\right|\geq \eta/\epsilon}\!\!K(y)\,dy\right|\\
&\quad+\sum_{j=1}^{m-1}\frac{\epsilon^j}{j!}\left|\rho_1^{(j)}(x_0)\right|\!\left|\int_{\left|y\right|<\eta/\epsilon}K(y)y^j\,dy\right|+\frac{\epsilon^m}{m!}\left\|\rho_1^{(m)}\right\|_{\infty}\!\int_{\left|y\right|< \eta/\epsilon}\!\!\!\!\left|K(y)\right|\left|y\right|^m\,\!dy.
\end{split}
\end{equation}
By the vanishing moment condition (see~\cref{def:mth_order_kernel} (ii)), we have that
\begin{equation}\label{eqn:equal_moments}
\left|\int_{\left|y\right|< \eta/\epsilon}K(y)y^j\,dy\right|=\left|\int_{\left|y\right|\geq \eta/\epsilon}K(y)y^j\,dy\right|, \qquad 1\leq j\leq m-1.
\end{equation}
\cref{def:mth_order_kernel} (iii) implies that $|K(x)| |x|^{m+1}\leq|K(x)|(1+|x|)^{m+1}\leq C_K$. Substituting~\cref{eqn:equal_moments} into~\cref{eqn:split_integration} with the bound for $|K(x)|$ and integrating, yields an upper bound for the right hand side of~\cref{eqn:split_integration}:
\begin{equation}\label{eqn:upper_bound1}
\sum_{j=0}^{m-1}\frac{\epsilon^{j}}{j!}\left|\rho^{(j)}_1(x_0)\right|\frac{2C_K}{m-j}\left(\frac{\epsilon}{\eta}\right)^{m-j} + \frac{\epsilon^m}{m!}\left\|\rho_1^{(m)}\right\|_{\infty}\int_{\left|y\right|< \eta/\epsilon}\left|K(y)\right|\left|y\right|^m\,dy.
\end{equation}
Since we can write $\rho_1^{(j)}(x_0)$ as an iterated integral of $\rho_1^{(m)}$, we find that
$$
\rho_1^{(j)}(x_0)=\int_{x_0-\eta}^{x_0}\int_{x_0-\eta}^{t_1}\cdots\int_{x_0-\eta}^{t_{m-j-1}}\rho_1^{(m)}(t_{m-j})dt_{m-j}\cdots dt_1, \qquad 0\leq j\leq m-1,
$$
and so it follows that $|\rho_1^{(j)}(x_0)|\leq \eta^{m-j}\|\rho_1^{(m)}\|_{\infty}$. Thus, we have
\begin{equation}\label{eqn:upper_bound2}
\sum_{j=0}^{m-1}\frac{\epsilon^{j}}{j!}\left|\rho^{(j)}_1(x_0)\right|\frac{2C_K}{m-j}\left(\frac{\epsilon}{\eta}\right)^{m-j}\leq 2eC_K\|\rho_1^{(m)}\|_{\infty}\epsilon^m.
\end{equation}
Recalling our selection of $\rho_1$, case (ii) follows from~\cref{eqn:approx_error_term2APPENDIX},~\cref{eqn:upper_bound1}, and~\cref{eqn:upper_bound2}.
\end{proof}

\subsection{${L^p}$ error bounds}\label{subsec:Lp_bounds}
In~\cref{sec:alt_convergence} we motivate error bounds for $\|\rho_f-K_\epsilon\ast \mu_f\|_{L^1}$ to ensure that the calculation of ionization probabilities is meaningful. In this subsection, we prove the $L^p$ error bounds stated in~\cref{thm:Lp_rates}. It is often easier to prove these kind of results in Fourier space so we begin by understanding the regularity properties of $\widehat K$ for an $m$th order kernel (see~\cref{def:mth_order_kernel}), where
\begin{equation} 
\widehat K(\omega):=\int_\mathbb{R} K(x)e^{-2\pi i x \omega}dx, \qquad \omega\in\mathbb{R}.
\label{eq:FourierTransform} 
\end{equation} 

\begin{lemma}[Regularity of Fourier Transform]\label{lemma:FT_kernel}
Let $K$ be an $m$th order kernel (see~\cref{def:mth_order_kernel}). For any $\alpha\in (0,1)$, we have that $\widehat K\in \mathcal{C}^{m-1,\alpha}(\mathbb{R})$ and
\begin{equation}\label{eqn:kernel_origin}
|\widehat K(\omega) - 1|\leq \frac{|\widehat K^{(m-1)}|_{\mathcal{C}^{0,\alpha}}}{(m-1)!}|\omega|^{m-1+\alpha}.
\end{equation}
\end{lemma}
\begin{proof}
Using~\cref{decay_bound} we can differentiate through the integral sign in~\cref{eq:FourierTransform} to conclude that $\widehat K$ is $(m-1)$-times continuously differentiable. Moreover,~\cref{decay_bound} implies that $\widehat K^{(m-1)}\in \mathcal{W}^{s,2}(\mathbb{R})$ for any $s<3/2$ (see~\cite{mclean2000strongly} for definition of fractional Sobolev spaces). Therefore, $\widehat K\in \mathcal{C}^{m-1,\alpha}(\mathbb{R})$ for any $\alpha\in (0,1)$ \cite[Thm.~3.26]{mclean2000strongly}.

For~\cref{eqn:kernel_origin}, note that the normalization condition (\cref{def:mth_order_kernel} (i)) implies that $\widehat K(0)=1$, while the vanishing moment criterion (\cref{def:mth_order_kernel} (ii)) implies that $(\widehat K)^{(j)}(0)=(-2\pi i)^j\int_{\mathbb{R}}K(x)x^jdx =0$ for $1\leq j\leq m-1.$ The bound~\cref{eqn:kernel_origin} then follows by using the $(m-1)$th order Taylor expansion for $\widehat K$ at the origin and applying the H\"older condition to the remainder.
\end{proof}

We can now use this to bound the $L^p$ error of a smoothed approximation $K_\epsilon*g$ when $g\in\mathcal{W}^{m,p}(\mathbb{R})$ and has compact support. 
\begin{lemma}\label{lemma:convolution_bound}
Let $K$ be an $m$th order kernel and let $g$ be any function such that $g\in \mathcal{W}^{m,p}(\mathbb{R})$ for $1\leq p< \infty$ and ${\rm supp}(g)\subset I = (a-\eta,b+\eta)$ for some $\eta>0$. Then, for any $\epsilon>0$, we have that\footnote{The $\log(1/\epsilon)$ factor is avoided if extra decay---beyond~\cref{def:mth_order_kernel} (iii)---is assumed on $K$.}
\begin{equation}\label{eqn:convolution_bound}
\|[K_\epsilon*g]-g\|_{L^p(I)}\leq \frac{2\epsilon^m C_K}{m!} \|g^{(m)}\|_{L^p(\mathbb{R})}\log(1+(b-a+2\eta)/{\epsilon}).
\end{equation}
\end{lemma}
\begin{proof}
Since $K\in L^1(\mathbb{R})$, we can define the function
\begin{equation}
\label{phi1}
\phi_1(x)=\int_{-\infty}^x K(y)dy - H(x){\int_{\mathbb{R}} K(y)dy}=\begin{cases} 
\int_{-\infty}^x K(y)dy, & x<0, \\ 
-\int_x^\infty K(y)dy, & x>0,
\end{cases}
\end{equation}
where $H(x)$ denotes the Heaviside step function. Using~\cref{decay_bound} and integrating directly, we see that $\phi_1\in L^2(\mathbb{R})$. Furthermore, since $\int_{\mathbb{R}}K(y)dy=1$, we can differentiate $\phi_1$ in the sense of tempered distributions to obtain $\phi_1'=K-\delta_0.$ Taking Fourier transforms, we see that
$$
(2\pi i\omega)\widehat\phi_1(\omega)=\widehat K(\omega)-1.
$$
However, $\widehat\phi_1,\widehat K\in L^2(\mathbb{R})$ and hence we must have $\widehat\phi_1(\omega)=(\widehat K(\omega)-1)(2\pi i\omega)^{-1}$ almost everywhere, and in particular that $(\widehat K(\omega)-1)(2\pi i\omega)^{-1}\in L^2(\mathbb{R})$.

If $m>1$, then by~\cref{decay_bound} and the case definition of $\phi_1$ in~\cref{phi1}, we have $\phi_1\in L^1(\mathbb{R})$ and hence $\widehat\phi_1$ can be identified with a continuous function. Furthermore,~\cref{eqn:kernel_origin} implies that $\widehat\phi_1(0)=\int_\mathbb{R}\phi_1(y)dy=0$ and hence we can define
\[
\phi_2(x)=\int_{-\infty}^x \phi_1(y)dy- H(x){\int_{\mathbb{R}} \phi_1(y)dy}=\begin{cases} 
\int_{-\infty}^x \phi_1(y)dy, & x<0, \\ 
-\int_x^\infty \phi_1(y)dy, & x>0.
\end{cases}
\]
Again by using~\cref{decay_bound} and integrating directly, we see that $\phi_2\in L^2(\mathbb{R})$. We can take distributional derivatives and Fourier transforms as before to deduce that $\widehat\phi_2(\omega)=(\widehat K(\omega)-1)(2\pi i\omega)^{-2}$ almost everywhere. We continue this argument inductively, using~\cref{lemma:FT_kernel}, and for $j=2,\ldots,m$ define $\phi_j(x)=\smash{\int_{-\infty}^x\phi_{j-1}(y)dy}$. The argument shows that $\widehat\phi_j(\omega)=(\widehat K(\omega)-1)(2\pi i\omega)^{-j}$. Using~\cref{decay_bound} and integrating, we have
\begin{equation}\label{eqn:decay_phim}
|\phi_j(x)|\leq C_K(m-j)!/(m!(1+|x|)^{m-j+1}), \qquad 1\leq j\leq m.
\end{equation}

Let $g_n\in C_c^\infty(\mathbb{R})$ for $n\geq 1$ be a sequence of functions with ${\rm supp}(g_n)\subseteq (a-\eta-n^{-1},b+\eta+n^{-1})$ such that $\|g - g_n\|_{\mathcal{W}^{m,p}(\mathbb{R})}\rightarrow 0$ as $n\rightarrow\infty$. Let $\phi_{m,\epsilon}=\epsilon^{-1}\phi_m(x\epsilon^{-1})$, so that
$$
\widehat\phi_{m,\epsilon}(\omega)=\frac{\widehat K_\epsilon(\omega)-1}{(2\pi i\epsilon\omega)^m}
$$
It follows, by the convolution theorem and Carleson's theorem, that for a.e. $x\in\mathbb{R}$
\begin{equation}\label{eqn:fourier_conv}
[K_\epsilon*g_n](x)-g_n(x)=\int_\mathbb{R}\frac{\widehat K_\epsilon(\omega)-1}{(2\pi i\omega)^m}(2\pi i\omega)^m\widehat g_n(\omega)e^{2\pi i\omega x} d\omega =\epsilon^m[\phi_{m,\epsilon}*g_n^{(m)}](x).
\end{equation}
Letting $L_n = ((b-a)+2(\eta+n^{-1}))$, we have $[\phi_{m,\epsilon}*g_n^{(m)}](x)=[\chi_{[-L_n,L_n]}\phi_{m,\epsilon}*g_n^{(m)}](x)$ for $x\in I$, where $\chi_U$ denotes the indicator function of a set $U$. Moreover, $\chi_{[-L_n,L_n]}\phi_{m,\epsilon}\in L^1(\mathbb{R})$ by H\"older's inequality. Taking the $L^p$ norm on both sides of~\cref{eqn:fourier_conv} and applying Young's convolution inequality, yields
\begin{equation}\label{eqn:smoothed_conv_bound}
\|[K_\epsilon*g_n]-g_n\|_{L^p(I)}\leq\epsilon^m\|g_n^{(m)}\|_{L^p(\mathbb{R})}\|\phi_{m,\epsilon}\|_{L^1([-L_n,L_n])}.
\end{equation}
By taking the limit $n\rightarrow\infty$ in~\cref{eqn:smoothed_conv_bound}, we have that
\begin{equation}\label{eqn:smoothed_conv_bound2}
\|[K_\epsilon*g]-g\|_{L^p(I)}\leq\epsilon^m\|g^{(m)}\|_{L^p(\mathbb{R})}\int_{|y|\leq (b-a)+2\eta}|\phi_{m,\epsilon}(y)|\,dy.
\end{equation}

Finally, by~\cref{eqn:decay_phim} with $j=m$, we have that $|\phi_{m,1}(x)|\leq C_K(m!(1+|x|))^{-1}$. Changing variables $y\rightarrow\epsilon y$ in the last integral in~\cref{eqn:smoothed_conv_bound2}, applying the bound for $\phi_{m,1}$, and integrating yields the upper bound in~\cref{eqn:convolution_bound}.
\end{proof}

We are now ready to prove the $L^p$ error bounds when $1\leq p<\infty$. 

\begin{proof}[Proof of~\cref{thm:Lp_rates}]
Let $I'=(a-\eta/2,b+\eta/2)$. Since $\rho_f|_I\in \mathcal{W}^{m,p}(I)$, we can decompose $\rho_f=\rho_1+\rho_2$ such that $\rho_1$ is non-negative, supported in $I$ with $2\|\rho_1^{(m)}\|_{L^p(\mathbb{R})}/m!\leq C(m)\|\rho_f\|_{\mathcal{W}^{m,p}(I)}(1+\eta^{-m})$ for some constant $C(m)$ (that depends only on $m$) and $\rho_2$ is non-negative with support contained in $\mathbb{R}\setminus I'$. Therefore, $\rho_f = \rho_1$ on $(a,b)$ and
\begin{equation}\label{eqn:three_eps}
\|\rho_f-K_\epsilon*\mu_f\|_{L^p((a,b))}\leq \|\rho_1-K_\epsilon*\rho_1\|_{L^p((a,b))}+\|K_\epsilon*\rho_1-K_\epsilon*\mu_f\|_{L^p((a,b))}.
\end{equation}
The first term is bounded via~\cref{lemma:convolution_bound}. To bound the second term, we note that the measure $\mu_f^{(\mathrm{r})}=\mu_f-\rho_1$ is non-negative, supported in $\mathbb{R}\setminus I'$, and has $\smash{\int_\mathbb{R}d\mu_f^{(\mathrm{r})}(y)\leq 1}$. Applying property (iii) in~\cref{def:mth_order_kernel}, we see that
\[
\begin{aligned}
\|K_\epsilon*\rho_1-K_\epsilon*\mu_f\|_{L^p((a,b))}^p &\leq\int_{a}^b \left(\epsilon^{-1}\int_{\mathbb{R}\setminus I'}\left|K\left((x-y)/\epsilon\right)\right|d\mu^{(\mathrm{r})}(y)\right)^p dx \\
&\leq {C_K^p(b-a)}{(\epsilon+\eta/2)^{-(m+1)p}}\epsilon^{mp}.
\end{aligned}
\]
Combining the bounds for the terms in~\cref{eqn:three_eps} concludes the proof.
\end{proof}

\section*{Acknowledgements}
The authors are grateful to St John's College, Cambridge for funding the first author to visit Cornell University, during which the collaboration started. We thank Anthony Austin and Mikael Slevinsky for carefully reading a draft version of this manuscript and the referees whose careful comments helped us improve the manuscript.

\bibliography{bib_file}
\bibliographystyle{plain}
\end{document}